\documentclass[12pt,twoside,english,a4paper]{article}
\usepackage{babel,amssymb,amsthm}
\usepackage{graphics}
\usepackage{amssymb,amsthm}
\usepackage{graphicx}
\usepackage{amsmath}
\usepackage{mathrsfs}
\usepackage{bm}
\usepackage{color}
\newcommand{\jump}[1]{[\![#1]\!]}

\newcommand{\triple}[1]{|\!|\!|#1|\!|\!|}
\newtheorem{proposition}{Proposition}[section]
\newtheorem{corollary}[proposition]{Corollary}

\newtheorem{theorem}[proposition]{Theorem}

\numberwithin{equation}{section}

\title{A Time-Dependent Wave-Thermoelastic Solid  Interaction}

\date{\today \\ \it This paper is dedicated to Wolfgang L. Wendland \\[1mm] on the occasion of his 80th birthday.}

\author{
{\sc G. C. Hsiao$^*$, T. S\'anchez-Vizuet$^\dagger$, F.--J.Sayas$^*$ \& R. J. Weinacht$^*$}\\
{\small $*$Department of Mathematical Sciences,  University of Delaware, USA.}\\
{\small {\tt \{ghsiao, fjsayas, weinacht\}@udel.edu}}\\
{\small $\dagger$ Courant Institute of Mathematical Sciences, New York University, USA.}\\
{\small {\tt tonatiuh@cims.nyu.edu}}
}
\begin{document}
\maketitle

\begin{abstract}
\noindent 
This paper presents a combined field and boundary integral equation method for solving the time-dependent scattering problem of a thermoelastic body immersed in a compressible,
inviscid  and  homogeneous fluid. The approach here is a generalization of the coupling procedure employed by the authors for the treatment of the time-dependent fluid-structure interaction problem. Using an integral representation of the solution in the infinite exterior domain occupied by the fluid, the problem is reduced to one defined only over the finite region occupied by the solid, with nonlocal boundary conditions. The nonlocal boundary problem is analyzed with Lubich's approach for time-dependent boundary integral equations. Existence and uniqueness results are established in terms of time-domain data with the aid of Laplace-domain techniques. Galerkin semi-discretization approximations are derived and error estimates are obtained.  A full discretization based on the Convolution Quadrature method is also outlined.  Some numerical experiments in 2D are also included in order to demonstrate the accuracy and efficiency of the procedure.
\end{abstract}


%
% ======================
›\section{Introduction}
% ======================
%
The mathematical study of the thermodynamic response of a linearly elastic solid to mechanical strain dates back at least to Duhamel's 1837 pioneering work \cite{Duhamel:1837}  on thermoelastic materials where he proposed the constitutive relation linking the temperature variations and elastic strains with the thermoelastic stress now known as \textit{Duhamel-Neumann law} \cite{Ca:1972, Maugin:2014}.

Kupradze's encyclopedic works \cite{Ku:1979} can be considered the standard reference for a modern mathematical treatment of the purely thermoelastic problem. The dynamic problem is dealt with in more recent works like \cite{OrWa:1992,Wagner:1994} where the matrix of fundamental solutions for the dynamic equations is revisited, while \cite{Jakubowska:1982,Jakubowska:1984} provide generalized Kirchhoff-type formulas for thermoelastic solids. 

In the case of the scattering of thermoelastic waves, major theoretical contributions have been made by \c Cakoni and Dassios in \cite {CaDa:1998}. The  unique solvability of a boundary integral formulation is established in \cite{Cakoni:2000}  and  the interaction of elastic and thermoelastic waves is explored for homogeneous materials in  \cite{DaKo:1994}.  The study of time-harmonic interaction between a scalar field and a thermoelastic solid has been the subject of works like \cite{Lopatev:1979} where the interface is taken to be a plane, or \cite{LiRa:1983,JeNa:1997} where time-harmonic scattering by bounded obstacles is considered.

In this paper, we present a combined field and boundary integral method for a time-dependent fluid-thermoelastic solid interaction problem. The approach here is a generalization of the method employed by the authors for treating time-dependent fluid -structure interaction problems in 
\cite{HSW:2013, HsSaSa:2016}. The present communication is an improvement over those previous efforts in the sense that it considers a more general constitutive law that accounts for the coupling between elastic and thermal effects. To our knowledge no attempt has been made to investigate with rigorous justifications the time-dependent acoustic scattering by a thermoelastic obstacle.

The setting is introduced in Section \ref{sec:2}, along with the physical assumptions and the constitutive relation under consideration leading to the time-domain system of governing equations. The problem is then recast in Section \ref{sec:3} where the Laplace domain system is transformed into an equivalent integro-differential non-local problem that will be formulated variationally for discretization later on. The question of existence and uniqueness of the solutions to the non-local problem is dealt with on Section \ref{sec:4}. The error analysis of the proposed discretization is addressed in Section \ref{sec:5}, where semi-discrete error estimates for spatial discretization are provided.

The final Section \ref{sec:6} discusses the computational considerations related to the numerical solution of the discrete problem. A full discretization using Convolution Quadrature (CQ) in time is outlined and the coupling of boundary and finite elements for the spatial discretization is discussed. Convergence experiments in 2D are performed for test problems in both frequency and time domains as a demonstration of the applicability of the formulation which remains valid also in 3D. For time discretization both second order backward differentiation formula (BDF2) and Trapezoidal Rule CQ are used, providing evidence that the approximation is stable and of second order globally. Time-domain illustrative experiments using the proposed formulation are included. 

In closing, we remark that for homogeneous thermoelastic  solid medium, a pure boundary integral equation formulation may be adapted as in the fluid-structure interaction problem \cite{HsSaSa:2016}.  We will pursue these investigations in a separate communication. 
%
% ===============================================
\section{Formulation of the problem}\label{sec:2}
% ===============================================
%\subsection{An initial -boundary transmission problem}
Consider a thermoelastic solid with constant density $\rho_\Sigma$ in an undeformed reference configuration and at thermal equilibrium at temperature $\Theta_0$. Under the action of external forces the body will be subject to internal stresses that will induce local variations of temperature. Reciprocally, if a heat source induces a change in temperature, the body will react by dilating or contracting and this will create internal stresses and deformations. We will denote by $U$ the elastic deformation with respect to the reference configuration and by $\Theta$ the variation of temperature with respect to the equilibrium temperature. In the classical linear theory \cite{Ku:1979,LaLi7:1986}, the coupling between the mechanic strain and the thermal gradient is modeled by the Duhamel-Neumann law which defines the \textit{thermoelastic stress} $\boldsymbol \sigma(\mathbf U,\Theta)$ and the \textit{thermoelastic heat flux} $\mathbf F(\mathbf U,\Theta)$  (also known as free energy) as functions of the elastic displacement and the variation in temperature by
\begin{align*}
\boldsymbol\sigma :=\,& \mathbf C\boldsymbol\varepsilon(\mathbf U) - \zeta\Theta\mathbf I\,,\\
\label{eq:c5:1b}
\mathbf F :=\,& -\eta\,\frac{\partial \mathbf U}{\partial t} + \kappa \nabla\Theta.
\end{align*}
In the previous expressions
\[
\bm \varepsilon ({\bf U}) := \frac{1}{2}( \nabla{\bf U} + (\nabla{\bf U})^t)
\]
is the elastic strain tensor, $\mathbf I$ is the $3\times3$ identity matrix, $\kappa$ is the \textit{thermal diffusivity} coefficient, which from physical principles \cite{HaOz:2012} is required to be positive, $\zeta$ is the product of the volumetric \textit{thermal expansion} coefficient and the bulk modulus of the material, and $\eta$ is given by the relation
\begin{equation*}\label{eq:c5:2}
\eta = \Theta_0\zeta/c_{vol}.
\end{equation*}
Here the volumetric heat capacity $c_{vol}$ is the ratio between the thermal diffusivity and the thermal conductivity, and it can also be expressed as the product of the mass density and the specific heat capacity. For the case of homogeneous isotropic material that we are considering, the \textit{elastic stiffness tensor} $\mathbf C$ is  given by
\[
\mathbf C_{ijkl}:= \lambda\delta_{ij}\delta_{kl} +\mu(\delta_{ik}\delta_{jl} + \delta_{il}\delta_{jk}), 
\]
where the constants $\lambda$ and $\mu$ are Lam\'e's second parameter and the shear modulus respectively, and $\delta_{ij}$ is Kronecker's delta.
 
We are concerned with a time-dependent direct scattering problem in fluid-thermoelastic solid  interaction, which can be simply described as follows: an acoustic wave propagates in a
fluid domain of infinite extent in which a bounded thermoelastic body is immersed. Throughout the paper, we let $\Omega_-$ be the bounded domain in $\mathbb{R}^3$ occupied by the thermoelastic  body with a Lipschitz boundary $\Gamma$ and we let  $\Omega_+ := \mathbb{R}^3  \setminus  \overline{\Omega}_-$  be its exterior, occupied by a compressible fluid.  The problem is then to determine the scattered velocity potential $V$ in the fluid domain,  the deformation of the solid $\mathbf U$ and the variation of the temperature $\Theta$  in the obstacle.  It is assumed that $|\Theta/\Theta_0|<<1$. 

The governing equations of  the displacement field $\bf {U}$ and temperature field  $\Theta$ are  the thermo-elastodynamic equations:
\begin{alignat}{6}
\rho_\Sigma \frac{\partial^2\mathbf{U}} {\partial t^2} - \Delta^{*} \mathbf{U} + \zeta\,  \nabla\Theta 
 =\,& &\mathbf{0}  &\quad& \hbox{ in }\Omega_- \times (0,T),  \label{eq:2.1}\\[2mm]
 \frac{1}{\kappa}\frac{\partial \Theta}{\partial t} - \Delta \Theta +  \eta\; \frac{\partial}{\partial t}(\nabla\cdot\mathbf{U})&  =\,&  0 &\quad& \hbox{ in }\Omega_- \times (0,T),    \label{eq:2.2}
\end{alignat}
where $T$ is a given positive final time, and as usual the symbol $\Delta^*$ is  the Lam\'e  operator defined by
\[
 \Delta^* \mathbf{U} :=  \mu \Delta \mathbf{U} + (\lambda + \mu) \nabla(\nabla\cdot \mathbf U).
\]
We remark that if the thermal effect is neglected ($\zeta=0$) Duhamel-Neumann's law reduces to the usual expression for Hooke's law of the classical theory for an arbitrary isotropic medium (see, e.g.  \cite{Ca:1972, Ku:1979}).  In the thermoelastic medium, the given physical constants $ \rho_\Sigma, \lambda, \mu, \zeta, \eta,
\kappa$, are assumed to satisfy the inequalities:
\[
\rho_\Sigma > 0, \;\;\mu > 0 , \;\;3 \lambda + 2 \mu > 0, \;\;\frac{\zeta}{\eta} > 0, \;\;\kappa > 0.
\]
In the  fluid domain  $\Omega_+$,  we consider a barotropic and irrotational  flow of an inviscid and compressible  fluid with density $\rho_f$ as in \cite{HSW:2013}.  The formulation can be presented in terms of a scalar potential $V = V(x,t)$ such that the scattered velocity field $\mathbf V$ and the pressure $P$ are given by
$$
{\bf V} =  - \nabla\; V \quad\mbox{and}\quad 
P= {\rho_f} \frac{\partial V}{\partial t}. 
$$ 
Then we arrive  at the wave equation 
\begin{equation}\label{eq:2.3}
\frac{1}{c^2} \frac{\partial^2 V}{\partial  t^2} - \Delta V = 0 \quad\mbox{in}\quad \Omega_+ \times (0, T)
\end{equation}
where $c$ is the sound speed.  

On the interface $\Gamma$ between the solid and the fluid we have  the transmission conditions 
\begin{subequations}
\begin{alignat}{6}
 \boldsymbol\sigma(\mathbf U, \Theta)^- {\bf n} =\, & - {\rho_f}\;( \frac{\partial V}{\partial t} + 
  \frac{\partial V^{inc}}{\partial t})^+{\bf n} &\quad& \hbox{ on }\Gamma \times (0, T),  \label{eq:2.4}\\
 \frac{ \partial {\bf U}^-} {\partial t} \cdot {\bf n} = \, & - (\frac{\partial  V}{\partial n} +
  \frac{\partial  V^{inc}}{\partial n})^+  &\quad& \hbox{ on }\Gamma \times (0, T), \label{eq:2.5}\\
  \frac{\partial \Theta}{\partial n}^-  = \, & 0 &\quad& \hbox{ on }\Gamma \times (0, T), \label{eq:2.6}
\end{alignat}
\end{subequations}
 where ${\bf n}$ is the exterior unit normal to $\Omega_-$, and $V^{inc} $ denotes the given incident field, which is assumed to be supported away from $\Gamma$ at $t=0$. Here and in the sequel, we adopt the notation that $q^{\mp}$ denotes the limit of the function $q$ 
on $\Gamma$ from $\Omega_\mp$ respectively. Regarding the transmission conditions we remark that from the physical point of view, equation \eqref{eq:2.4} enforces the equilibrium of pressure at the solid-fluid interface, the condition \eqref{eq:2.5} expresses the continuity of the normal component of the velocity field, and \eqref{eq:2.6} refers to a thermally insulated body. We  assume the causal initial conditions
\begin{subequations}
\begin{align}
 {\bf U}(x,t) = \frac{\partial {\bf U}(x,t)}{\partial t} = {\bf 0},
 \quad \Theta(x, t) = 0 \quad & \mbox{for}\quad  x \in \Omega_-,\;\; t \leq 0, \label{eq:2.7}\\
 V(x,t) = \frac{\partial V} {\partial t} (x,t) = 0,\quad & \mbox{for} \quad  x\in \Omega_+,\;\; t \leq 0. \label{eq:2.8}
\end{align}
\end{subequations}
We will study the the time-dependent scattering problem consisting of the partial differential equations (\ref{eq:2.1})-(\ref{eq:2.3}) together with the transmission conditions  (\ref{eq:2.4})-(\ref{eq:2.6})  and  the homogeneous initial conditions 
(\ref{eq:2.7})-(\ref{eq:2.8}).
%
% ====================================================
\section{Reduction to a nonlocal problem}\label{sec:3}
% ====================================================
%
In order to apply Lubich's approach as in the case of fluid-structure interaction \cite{HSW:2013, HsSaSa:2016},  we first need to transform the initial-boundary transmission problem (\ref{eq:2.1})-(\ref{eq:2.6}) in the Laplace domain. Then we reduce the corresponding problem to a nonlocal boundary value problem. We begin with the Laplace transform for a restricted class of distributions. Let $X$ be a Banach space and $\mathcal S(\mathbb R)$ denote the Schwartz class of functions. We say that $F: \mathcal S(\mathbb R)\rightarrow X$ is a causal tempered distribution with values in $X$ if it is a continuous linear map such that
\[
F(\varphi) = 0 \quad \forall\varphi\in\mathcal S (\mathbb R) \hbox{ such that } supp~\varphi \subset (-\infty,0).
\]
For such a distribution and
\[
s\in\mathbb C_+ :=\{ s \in \mathbb C : \mathrm{Re}~s > 0\},
\]
the Laplace transform of $F$  can be defined in a natural way by 
\[
f(s)= \mathcal{L}\{F\}(s) := \int_0^\infty e^{-st} F(t) dt,
\] 
where the integral must be understood in the sense of Bochner \cite{Sa:2015}. We remark that the Laplace transform can be defined for a much broader class of distributions \cite{BeWo:1966,DaLi:1992}, but this restricted class suffices for the current application.

Let then
\[
\mathbf{u}:=  \mathbf{u}(x,s)= \mathcal{L}\{ {\bf U}(x,t)\}, \;\;  \theta:=\theta(x,s) = \mathcal{L}\{{\Theta(x,t)} \},\;\;   v:=v(x,s) = \mathcal{L}\{V(x,t)\}.
\]
Then the initial-boundary transmission problem consisting of
(\ref{eq:2.1}) - (\ref{eq:2.6}) in the Laplace transformed domain becomes  the following transmission boundary value problem: 
\begin{subequations}\label{eq:3}
\begin{alignat}{6}
 - \Delta^{*} \mathbf{u} +  \rho_\Sigma s^2 \mathbf{u} + \zeta \nabla\theta   = \,& \mathbf{0}
	&\quad& \mbox{in}\quad \Omega_-, \label{eq:3.1} \\
 - \Delta \theta + \frac{s}{\kappa} ~\theta + s~ \eta ~\nabla\cdot \mathbf{u}  =\,&  0 
	&\quad& \mbox{in} \quad \Omega_-, \label{eq:3.2} \\
-\Delta v + \frac{s^2}{c^2} ~v = \,& 0 
	&\quad& \mbox{in}\quad  \Omega_+, \label{eq:3.3}\\
\bm{\sigma} (\mathbf{u}, \theta )^- \mathbf{n} + \rho_f\; s v^+ \mathbf{n} =\,& - \rho_fs v^{inc}\mathbf{n} 
	&\quad& \mbox{on}\quad \Gamma, \label{eq:3.4}\\
s ~\mathbf{u}^- \cdot \mathbf{n} + \frac{ \partial v}{\partial n}^+ =\,& -\frac {\partial v^{inc}}{\partial n} 
	&\quad& \mbox{on}\quad \Gamma,  \label{eq:3.5}\\
\frac{\partial \theta}{\partial n}^- = \,& 0  
	&\quad& \mbox{on} \quad \Gamma. \label{eq:3.6}
\end{alignat}
\end{subequations}
We remark that \eqref{eq:3} is an exterior scattering  problem for which normally a radiation condition is needed in order to guarantee the uniqueness  of the solution. However, in the present case no additional radiation condition is required and global $H^1$ behavior at infinity suffices. 
 
To derive a proper nonlocal boundary problem, as usual, we begin via Green's third identity with 
the representation of the solutions of (\ref{eq:3.3}) in the form: 
\begin{equation}\label{eq:3.7}
 v = \mathcal{D}(s) {\phi} - \mathcal{S}(s) {\lambda} \quad \mbox{in} \quad \Omega_+,
\end{equation}
where ${\phi}:= v^+(s)$ and ${\lambda}:= \partial v^+ /\partial n$ are the Cauchy data for $v$ in (\ref{eq:3.3}) and $\mathcal{S}(s)$ and $\mathcal{D}(s)$ are the simple-layer and double-layer potentials, respectively defined by 
\begin{alignat}{6}
\label{eq:3.8}
\mathcal{S}(s) {\lambda} (x) :=\,& \int_\Gamma  E_{s/c}(x,y) {\lambda}(y) d\Gamma_y, 
	&\quad& x \in \mathbb R^3\setminus\Gamma, \\
\label{eq:3.9}
\mathcal{D}(s) {\phi} (x)  :=\,& \int_\Gamma \frac{\partial}{\partial n_y} E_{s/c}(x,y) {\phi}(y)  d\Gamma_y, 
	&\quad& x \in \mathbb R^3\setminus\Gamma.
\end{alignat}
Here 
$$E_{s/c}(x,y) =  \frac{e^{- s/c~ |x-y|}} {4 \pi |x-y|}$$
 is the fundamental solution of the  operator  in  (\ref{eq:3.3}). By  standard arguments in potential theory, we have the relations for the the Cauchy data ${\lambda}$ and ${\phi} $:
\begin{equation}\label{eq:3.10}
\begin{pmatrix}
{\phi} \\[3mm]
{\lambda}\\
\end{pmatrix}
= \left (
\begin{matrix}  % or pmatrix or bmatrix or Bmatrix or ...
      \frac{1}{2}I + K(s) & -V(s) \\[3mm]
       -W(s)  &  ( \frac{1}{2}I - K(s))'  \\
    \end{matrix}
    \right )\begin{pmatrix}
{\phi} \\[3mm]
{\lambda}\\
\end{pmatrix}
\quad \mbox{on} \quad \Gamma.
 \end{equation}
Here $V, K, K^\prime $ and $W$  are the four  basic boundary integral operators familiar from potential theory \cite{HsWe:2008}  such that 
\begin{alignat*}{6}
V(s) \lambda (x) :=\,& \int_\Gamma  E_{s/c}(x,y) {\lambda}(y) d\Gamma_y 
 	&\quad& x \in \Gamma,\\
K(s){\phi} (x)  :=\,& \int_\Gamma \frac{\partial}{\partial n_y} E_{s/c}(x,y){\phi}(y)  d\Gamma_y 
	&\quad& x \in \Gamma,\\
K^\prime(s)\lambda (x) :=\,& \int_\Gamma  \frac{\partial}{\partial n_x} E_{s/c}(x,y) {\lambda}(y) d\Gamma_y 
 	&\quad& x \in \Gamma,\\
W(s) {\phi} (x)  :=\,&- \frac{\partial}{\partial n_x}  \int_\Gamma \frac{\partial}{\partial n_y} E_{s/c}(x,y){\phi}(y)  d\Gamma_y 
	&\quad& x \in \Gamma.
\end{alignat*}
By using the transmission condition  (\ref{eq:3.5}), we obtain  from the second  boundary integral equation in  (\ref{eq:3.10}),  
\begin{equation}\label{eq:3.11}
-s \mathbf{u}^-  \cdot \mathbf{n} + W(s){\phi}   - ( \tfrac{1}{2} I -  K(s))'  {\lambda} =
\frac{\partial v^{inc }}{\partial n} \;\quad \mbox{on}\quad \Gamma
\end{equation} 
while  the first boundary integral equation in (\ref{eq:3.10})  is  simply   
\begin{equation} \label{eq:3.12}
( \tfrac{1}{2}I - K(s)) {\phi} + V(s) {\lambda} = 0 \quad \mbox{on} \quad \Gamma.
\end{equation}
With the Cauchy data $\phi$ and $\lambda$ as new unknowns, the partial differential equation 
(\ref{eq:3.3}) in $\Omega_+$  may  be eliminated.  This leads to a {\it nonlocal boundary value problem}
in $\Omega_-$ for the unknowns $(\mathbf u, \theta, \phi, \lambda)$  satisfying the partial differential equations  (\ref{eq:3.1}), (\ref{eq:3.2}), and the boundary integral equations (\ref{eq:3.11}), (\ref{eq:3.12}) together with the conditions (\ref{eq:3.4}) and (\ref{eq:3.6})
on $\Gamma$.

Here and in the sequel let $\gamma^{\mp} $ and $\partial_n^{\mp}$ denote trace operators of the functions and their normal derivatives from inside and outside  $\Gamma$, respectively. We will use the symbol $(\cdot,\cdot)_{\mathcal O}$ interchangeably to denote the scalar, vector, or Frobenius $L^2$ inner products of functions defined on the open set $\mathcal O$, while the angled brackets $\langle\cdot,\cdot\rangle_{\Gamma}$ will be reserved for pairings between elements of the trace space and its dual. All the forms will be kept linear and conjugation will be done explicitly when needed. Finally, the space $\mathbf H^1(\mathcal O)$ should be understood as the Cartesian product of copies of the standard scalar Sobolev space $H^1(\mathcal O)$ endowed with the natural product norm.

Let us first consider the unknowns $(\mathbf u,\theta) \in {\mathbf H}^1(\Omega_-)\times H^1(\Omega_-)$. Then multiplying (\ref{eq:3.1}) by the testing function ${\mathbf v}$ and integrating by parts, we obtain the weak formulation of (\ref{eq:3.1}): 
\begin{equation}\label{eq:3.13}
a({\mathbf u} , {\mathbf v}; s) - \langle {\bm \sigma}(\mathbf u,\theta)\mathbf n , \gamma^-{\mathbf v}\rangle_{\Gamma}  - \zeta(\theta, \; \nabla\cdot\mathbf v)_{\Omega_-} = 0, 
\end{equation}
where  $ a(\cdot,\cdot; s)$ is the bilinear form defined by 
\begin{equation}\label{eq:3.14}
a(\mathbf u, \mathbf v; s): = \left( \mathbf C\bm \varepsilon (\mathbf u) , \bm \varepsilon ( \mathbf v)\right)_{\Omega_-}  + s^2\rho_\Sigma( \mathbf u, \mathbf v)_{\Omega_-}.
\end{equation}
In terms of the transmission condition (\ref{eq:3.4}), we obtain from (\ref{eq:3.13})
\begin{equation}\label{eq:3.15}
a({\mathbf u} , {\mathbf v}; s)   - \zeta (\theta, \; \nabla\cdot\mathbf v)_{\Omega_-} + 
\rho_f s \langle \phi~ {\mathbf n} , \gamma^-{\mathbf v}\rangle_{\Gamma}  = - s\rho_f  \langle v^{inc} \mathbf n, \gamma^- \mathbf v \rangle_{\Gamma}.
\end{equation}
Similarly, multiplying (\ref{eq:3.2}) by the test function $\vartheta$, integrating by parts and making use of the condition \eqref{eq:3.6}, we have
\begin{equation}\label{eq:3.16}
b( \theta, \vartheta; s)  + s~ \eta(\nabla\cdot\mathbf u, \vartheta)_{\Omega_-}= 0
\end{equation}
with
\begin{equation}\label{eq:3.17}
b(\theta, \vartheta;s) :=  (\nabla \theta, \nabla \vartheta)_{\Omega_-} +  \frac{s}{\kappa} ( \theta, \vartheta)_{\Omega_-}.
\end{equation}

Now let 
\begin{alignat*}{6}
\mathbf A_s:\;& \mathbf H^1(\Omega_-) &\;\longrightarrow\;&\; (\mathbf H^1(\Omega_-))' \\
 &\;\mathbf u \;&\; \longmapsto \;&\; a(\mathbf u, \,\cdot\,;s) \\
B_s :\;&  H^1(\Omega_-)\;& \;\longrightarrow\;&\; (H^1(\Omega_-))' \\ & \; \theta \;&\; \longmapsto \,& \; b(\theta,\,\cdot\,;s) 
\end{alignat*}
be the operators associated to the bilinear forms \eqref{eq:3.14} and \eqref{eq:3.17}, respectively.  Then from \eqref{eq:3.15}, \eqref{eq:3.17}, \eqref{eq:3.11}, and \eqref{eq:3.12}, the nonlocal problem may be formulated as a system of  operator equations: Given data $(d_1, d_2, d_3, d_4)  \in X ^\prime,$ find $(\mathbf u, \theta, \phi, \lambda) \in X $ such that 
\begin{equation}\label{eq:3.18}
\pmb{\mathscr{A}}    \begin{pmatrix}
      \mathbf{u}   \\
      \theta \\
      \phi\\
      \lambda \\
   \end{pmatrix}:=
\left (   \begin{matrix}
   \mathbf A_{s}& -\zeta~( div)^{\prime} & s ~\rho_f\;  {\gamma^-}^{\prime} ~\mathbf n &  0 \\
   s ~\eta ~div  & B_{s} &0&0  \\
   - s~ {\mathbf n}^{\top} \gamma^- & 0 & W(s) & \!\!\! - ( \frac{1}{2} I -  K(s))' \\
   0 & 0 & \!\!\! \frac{1}{2}I - K(s) & V(s)\\
   \end{matrix} 
\right )
   \begin{pmatrix}
    \mathbf{u}   \\
    \theta\\
    \phi\\
    \lambda \\
    \end{pmatrix}
= \begin{pmatrix}
d_1\\
d_2\\
d_3\\
d_4
\end{pmatrix}. 
\end{equation}
In the above expression $\mathbf n^\top$ denotes the transpose (row) of the outward unit normal (column) to $\Gamma$, and data $(d_1,d_2,d_3, d_4)$ is given by
\begin{equation} \label{eq:3.19}
d_1 = -s~ \rho_f\; {\gamma^-}^{\prime}(\gamma^+v^{inc}{\mathbf n}), \quad d_2 = 0, \quad 
d_3 =  {\partial_n^+ v^{inc}}, \quad  d_4 = 0.
\end{equation}
We have made use of the product spaces:
\begin{align*}
X := \,& {\mathbf H}^1( \Omega_-)  \times H^1(\Omega_-) \times H^{1/2}(\Gamma) \times H^{-1/2}(\Gamma), \\
X^{\prime} :=\,&  ({\mathbf H}^1(\Omega_-))^{\prime} \times (H^1(\Omega_-) )^{\prime} \times H^{-1/2}(\Gamma) \times H^{1/2}(\Gamma)
\end{align*}
(i. e., $X^{\prime}$ is the dual of $X$).  Our aim is to show that equation \eqref{eq:3.18} has a unique solution in $X$. We will do this in the next section. 
%
% ========================================================
\section{Existence and uniqueness results}\label{sec:4}
% ========================================================
%
Before considering the existence and uniqueness results, we first discuss the invertibility of  the operator $\pmb{\mathscr{A} }$  in \eqref{eq:3.18}.  We begin with the definitions of  the following energy norms:
\begin{alignat}{6}
\triple{\mathbf u}_{|s|, \Omega_-}^2 :=\,& ( \mathbf C\boldsymbol\varepsilon( {\mathbf u}),  \bm{\varepsilon} (\bar {\mathbf {u}} ) )_{\Omega_-} +  \rho_\Sigma \|  s \; \mathbf u \|^2_{\Omega_-}, \quad & \mathbf u \in {\mathbf H}^1(\Omega_-),  \label{eq:4.1}\\
 \triple{\theta}^2_{|s|, \Omega_-} :=\, &\| \nabla \theta\|^2_{\Omega_-}  + \kappa^{-1}
 \| \sqrt{ |s|} \; \theta \|_{\Omega_-}^2 ,\quad & \theta \in H^1(\Omega_-),  \label{eq:4.2}\\
 \triple{v}^2_{|s|, \Omega_+} :=\,&  \| \nabla v\|^2_{\Omega_+} +  c^{-2} \| s \;v \|^2_{\Omega_+}, \quad & v \in H^1(\Omega_+). \label{eq:4.3}
\end{alignat}
For the complex Laplace parameter $s$ we will denote
\[
\sigma:= \mathrm{Re}~s, \quad \underline{\sigma}:= \min\{1,\sigma\},
\]
and will make use of the following equivalence relations for the norms
 \begin{gather}
\underline{\sigma }\triple{\mathbf{u}}_{1, \Omega_-} \leq \triple{\mathbf{u}}_{|s|, \Omega_-}\leq
\frac{|s|}{\underline{\sigma}} \triple{\mathbf{u}}_{1, \Omega_-} ,\label{eq:4.4}\\
\sqrt{\underline{\sigma}}\triple{\theta}_{1, \Omega_-} \leq \triple{\theta}_{|s|, \Omega_+}  \leq 
\sqrt{\frac{|s|}{\underline {\sigma} } }\triple{\theta}_{1,\Omega_-},   \label{eq:4.5} \\
\underline{\sigma } \triple{v}_{1, \Omega_+} \leq \triple{v}_{|s|, \Omega_+}\leq
\frac{|s|}{\underline{\sigma}} \triple{v}_{1, \Omega_+} ,\label{eq:4.6}
 \end{gather}
which can be obtained  from the  inequalities: 
\[
\min\{1, \sigma\} \leq \min\{1, |s|\},\quad \mbox{and} \quad  \max\{1, |s|\} \min\{1, \sigma\}
 \leq |s|,~ \;\forall  s \in \mathbb{C}_+.
\]
We remark that the norms $\triple{\theta}_{1, \Omega_-}$  and  $ \triple{ v}_{1, \Omega_+}$ are  equivalent to  $\| \theta\|_{H^1(\Omega_-)} $ and  $ \|{ v}\|_{H^1(\Omega_+)}$, respectively,  and  so is the energy norm $\triple{\mathbf{u}}_{1, \Omega_-}$  equivalent to the $\mathbf{H}^1(\Omega_-)$ norm of ${ \mathbf{u}} $ by the second Korn  inequality \cite{Fi:1972}.

For the invertibility of $\pmb{\mathscr{A} }$, let  us introduce the diagonal matrix $\pmb{\mathscr{Z}}=\pmb{\mathscr{Z}}(s)$:
\begin{equation}\label{eq:4.7} 
\pmb{\mathscr{Z}} = \left( \begin{matrix}
1 & 0 & 0 &0\\
0 & \zeta (\eta |s|)^{-1}& 0&0\\
0& 0 & \rho_f& 0\\
0 & 0 & 0  &\rho_f\\
\end{matrix}
\right )\
\end{equation}
and consider the modified operator
\begin{equation}\label{eq:4.8}
\pmb{\mathscr{B}} :=
\pmb{\mathscr{Z}} \pmb{\mathscr {A}} =  
\left (   \begin{matrix}
   \mathbf A_{s}& -\zeta~( div)^{\prime} & s ~\rho_f\;  {\gamma^-}^{\prime} ~\mathbf n &  0 \\
   s{|s|^{-1}} ~\zeta ~div  & \zeta~ (\eta~ |s|)^{-1} B_{s} &0&0  \\
   - s~ \rho_f ~ {\mathbf n}^{\top} \gamma^- & 0 & \rho_f~W(s) &  - \rho_f ( \frac{1}{2} I -  K(s))' \\
   0 & 0 & \rho_f ( \frac{1}{2}I - K(s)) &\rho_f V(s)\\
   \end{matrix} 
\right ).
\end{equation}
It will be clear that in order to show the invertibility of $\pmb{\mathscr{A} }$ it suffices to prove that of $\pmb{\mathscr{B} }$. By the Gaussian elimination procedure (as in \cite{LaSa:2009b}), a simple computation shows that $\pmb{\mathscr{B}} $ can be decomposed in the form:
\begin{equation} \label{eq:4.9}
\pmb{\mathscr{B}} = \pmb{\widetilde{\mathscr{P}}}\, \pmb{\mathscr{C}}\,\pmb{\mathscr{P}},
\end{equation}
where
\[ 
\pmb{\mathscr{C}} = 
\left ( \begin{matrix}
   \mathbf A_{s}& -\zeta~( div)^{\prime} & s ~\rho_f\;  {\gamma^-}^{\prime} ~\mathbf n &  0 \\
   s{|s|^{-1}} ~\zeta ~div  & \zeta~ (\eta~ |s|)^{-1} B_{s} &0&0  \\
   - s~ \rho_f ~ {\mathbf n}^{\top} \gamma^- & 0 & \rho_f~ (W(s) + C_{\Gamma}(s)) & 0 \\
   0 & 0 & 0 & \rho_f V(s)\\
   \end{matrix} 
\right ), 
\]
\[
\widetilde{\pmb{\mathscr{P}}} = 
\left ( \begin{matrix}
I &0&0 & 0\\
0 & I& 0 & 0\\
 0 & 0& I &  - (\frac{1}{2} I - K(s))' V^{-1}(s)\\
 0& 0& 0& I
 \end{matrix}\right ) \quad \mbox{and} \quad 
\pmb{\mathscr{P}} = 
\left ( \begin{matrix}
I &0&0 & 0\\
0 & I& 0 & 0\\
 0 & 0& I & 0\\
 0& 0&V^{-1}(s) ( \frac{1}{2}I - K(s))&I
 \end{matrix}\right )
\]
with  $ C_{\Gamma}(s) =  (\frac{1}{2} I -  K(s))' V^{-1}(s)(\frac{1}{2} I -  K(s))$.

Since both $\pmb{\mathscr{P}}$ and $\widetilde{\pmb{\mathscr{P}}}$ are invertible, the invertibility of $\pmb{\mathscr{C}}$  will imply that of $\pmb{\mathscr{B}}$, but as in the time-dependent fluid-structure interaction \cite{HSW:2013},  the operator matrix $\pmb{\mathscr{C}} $ is indeed invertible. In fact,  it is not difficult to show that the operator matrix $\pmb{\mathscr{C}} $ is {\em strongly elliptic}  \cite{HsWe:2008, Mi:1970}  in the sense that
\begin{eqnarray}\label{eq:4.10}
 \mathrm{Re} \Big\{ \langle Z(s)\,\pmb{\mathscr{C}}  (\mathbf v, \vartheta, \psi, \chi) , \overline{ (\mathbf v, \vartheta, \psi, \chi) } \rangle \Big\}
 &\geq & C (\zeta, \eta, \rho_f ) \; \frac {\sigma \underline{\sigma}^3}{|s|^2}   \| (\mathbf v, \vartheta, \psi, \chi)\|^2_X 
 \end{eqnarray}
for all $(\mathbf{v}, \vartheta, \psi, \chi) \in X $, where $C (\zeta, \eta, \rho_f )$ is a constant depending only on the physical parameters and on the geometry of $\Omega_-$,  and  $ Z(s)$ is the matrix defined by 
\[
Z(s):= \left (\begin{matrix}
\bar{s}/|s|&0 &0&0\\
0& 1&0 &0 \\
0&0& \bar{s}/|s| & 0\\
0&0&0& s/|s|
\end{matrix} \right ).
\]
The action of $Z(s)$ amounts only for a rotation that reveals the elliptic nature of the original system. As such, the analyticity of $Z(s)$ is immaterial, since only the inverse Laplace transform of $(\mathbf u, \theta, \varphi, \lambda)$ is sought for. As for the proof of \eqref{eq:4.10}, we will just point out that it is simple to show that
\begin{alignat}{6}
\mathrm{Re} \Big\{ & 
 -\bar{s}/|s| (\zeta~( div)^{\prime} \vartheta,  \bar{\mathbf v } )_{\Omega_-} 
+  \bar{s}/|s| (s ~\rho_f\;  {\gamma^-}^{\prime} \;\psi~\mathbf n , \bar{\mathbf v} )_{\Omega_-} & \nonumber\\
\label{eq:4.10a}
& +   s/|s|(\zeta ~div ~\mathbf v, \bar{\vartheta} )_{\Omega_-} 
- \bar{s}/|s|   \langle s~ \rho_f ~ {\mathbf n}^{\top} \gamma^- \mathbf v, \bar{\psi}  \rangle_{\Gamma}
\Big\} 
 = & \quad 0. 
\end{alignat}
Further details are omitted,  since a similar proof will be repeated when we discuss the existence and uniqueness results for the solution of the nonlocal problem \eqref{eq:3.18}.

We now return to the solutions of the modified system of equations \eqref{eq:4.8} from \eqref{eq:3.18}:
\begin{equation}\label{eq:4.12}
\pmb{\mathscr{B}}  \begin{pmatrix}
      \mathbf{u}   \\
      \theta \\
      \phi\\
      \lambda \\
   \end{pmatrix}:=
\pmb{\mathscr{Z}} \pmb{\mathscr{A}}  \begin{pmatrix}
      \mathbf{u}   \\
      \theta \\
      \phi\\
      \lambda \\
   \end{pmatrix} =\pmb{\mathscr{Z}}  \begin{pmatrix}
   d_1\\
d_2\\
d_3\\
d_4
 \end{pmatrix} =  \begin{pmatrix}
 d_1\\
\zeta/\eta ~|s|^{-1} ~d_2\\
\rho_f \;d_3\\
\rho_f\; d_4
\end{pmatrix}.
\end{equation}
Suppose that $(\mathbf{u}, \theta,  \phi, \lambda) \in X$ is a solution of (\ref{eq:4.12}).  Let
\begin{equation}\label{eq:4.13}
v:= \mathcal{D}(s) \phi -  \mathcal{S}(s) \lambda  \quad \mbox{in} \quad \mathbb{R}^3 \setminus \Gamma.
\end{equation}
Then  $v\in H^1(\mathbb{R}^3 \setminus \Gamma)$  is the solution of the transmission problem:
\begin{equation}\label{eq:4.14}
 - \Delta v + (s/c)^2 v = 0 \quad \mbox{in}\quad  \mathbb{R}^3 \setminus \Gamma
 \end{equation}
satisfying  the following jump relations across $\Gamma$,  
\[
\jump{\gamma  v}:=  \gamma^+ v - \gamma^- v = \phi \in H^{1/2}(\Gamma), \quad  \jump{\partial_n v}  :=\partial_n^+ v - \partial_n^- v = \lambda \in H^{-1/2}(\Gamma).
\]
First, from (\ref{eq:4.14}) we see that 
\begin{alignat}{6}
\label{eq:4.15} 
 {\mathbf A}_{ s} \; {\mathbf u } - \zeta ~(div)' \theta  +  s\; \rho_f  {\gamma^-}^{\prime}~\jump{\gamma v} \;\mathbf n =\;& d_1 \quad && \mbox{in} \quad \Omega_- ,\\
\label{eq:4.16}
\frac{s}{|s|}  \zeta~ div \;\mathbf u + \frac{\zeta}{\eta} \frac{1}{|s|} B_{s}~\theta  = \;& \frac{\zeta}{\eta}\frac{1}{|s|} d_2  \quad && \mbox{in} \quad   \Omega_-  , \\
\label{eq:4.17}
-s ~{ \rho_f \; \mathbf n}\cdot \; \gamma^- {\mathbf u }- \rho_f  \partial_n^+ v = \; & \rho_f ~d_3 \quad && \mbox{on} \quad \Gamma, \\
\label{eq:4.18}
-  \rho_f \; \gamma^- v = \; & \rho_f ~d_4 \quad && \mbox{on} \quad \Gamma. 
\end{alignat}
Since $ d_4 =0 $, this means that $u$ is a solution of the homogeneous Dirichlet problem for the partial differential equation (\ref{eq:4.14}) in $\Omega_-$. Hence by the uniqueness of the the solution, we obtain  $v \equiv 0$ in  $\Omega_-$. Consequently, we have
\begin{equation} \label{eq:4.19}
 \jump{\gamma v} = \gamma^+ v = \phi \quad \mbox{and} \quad \jump{\partial_n v} = \partial_n^+ v = \lambda.
\end{equation}

Next, we  consider the variational formulation of the problem for equations  \eqref{eq:4.14} , \eqref{eq:4.15}  and \eqref{eq:4.16} together with the boundary conditions \eqref{eq:4.17} and  \eqref{eq:3.6}.   We seek  a  solution
\begin{equation*}
(\mathbf{u}, \theta, v) \in \pmb{\mathbb{H}} =  {\mathbf H}^1(\Omega_-) \times H^1(\Omega_-) \times H^ 1(\Omega_+)
\end{equation*}
with the corresponding  test functions  $(\mathbf{v}, \vartheta, v )$ in the same function space.
 To derive the variational equations, we should keep in mind that the variational formulation should be formulated not in terms of  the Cauchy data $\phi$ and $\lambda$ directly, but only through the jumps $\jump{\gamma v}$ and $\jump{\partial_n v} $ as indicated. 
 
Let  $(f,g) \in  H^1(\Omega_+) \times  H^1(\Omega_+)$, we introduce the bilinear form
\[
c_{\Omega_+} (f, g; s):= \left(\nabla f,\nabla g\right)_{\Omega_+} + (s/c)^2\left(f,g\right)_{\Omega_+},
\]
and its associated operator
\begin{alignat*}{6}
C_{s, \Omega_+}:&& \; H^1(\Omega_+) \;& \longrightarrow\;& \left(H^{1}(\Omega_+)\right)'\\
 	&&\; f \;& \longmapsto \;& c_{\Omega_+}(f,\;\cdot\;;s)
\end{alignat*}
Note that in these definitions the domain of integration is indicated explicitly. Using this notation we can use the first Green formula for equation \eqref{eq:4.14} and condition \eqref{eq:4.17} to obtain
\begin{align}
\nonumber 
\rho_f  (C_{s, \Omega_+ }\, v,w) _{\Omega_+} =\; &- \rho_f \langle \partial^+_n v, \gamma^+ w 
 \rangle\\
\label{eq:4.20}
 =\;& \rho_f  \langle  d_3, \gamma^+ w\rangle_{\Gamma}  +  \rho_f  \langle  s ~\gamma^- \mathbf{u}  \cdot \; \mathbf{n}, \gamma^+ w  \rangle_{\Gamma}.   
\end{align} 
Together with the weak formulations of \eqref{eq:4.15} and \eqref{eq:4.16},  we arrive at the following variational formulation: Find $(\mathbf{u}, \theta,  v) \in  \pmb{\mathbb{H}} $ satisfying
\begin{alignat}{6} 
({\mathbf A}_{s}\mathbf{u}, \mathbf v )_{\Omega_-}
\! \! - \!  \zeta( \theta, div \;  \mathbf v )_{\Omega_-}\! \! + \!s \rho_f \langle  \gamma^+ v \; \mathbf n ,  \gamma^- \mathbf v \rangle_{\Gamma}   = \,&  (d_1,  \mathbf v  )_{\Omega_-} \quad && \forall ~\mathbf v \in {\mathbf H}^1(\Omega_-), \nonumber \\
 \label{eq:4.21}
 \frac{\zeta}{|s|\eta}(B_{s}\;  \theta, ~ \vartheta   )_{\Omega_-}  + \frac{s}{|s|}  \zeta ~ ( div ~  \mathbf u,    {\vartheta } )_{\Omega_-} =\;& \frac{\zeta}{|s|\eta} ( d_2 , ~ {\vartheta})_{\Omega_-} \quad&& \forall ~ \vartheta \in H^1(\Omega_-), \\
 \nonumber
 \rho_f  (C_{s, \Omega_+ }\, v,w)_{\Omega_+} -  s~\rho_f  \langle  \gamma^- \mathbf{u} \; , ~  {\gamma^+ v} ~\mathbf{n} \rangle_{\Gamma}  =\; &
\rho_f  \langle  d_3,  {\gamma^+ w} \rangle_{\Gamma} \quad&& \forall ~ w \in H^1(\Omega_+).
 \end{alignat}
We remark that by construction, it can be shown that  as in \cite{Sa:2015} this variational problem is equivalent  to   the transmission problem defined by \eqref{eq:4.14}, \eqref{eq:4.15}, \eqref{eq:4.16},
 and  \eqref{eq:4.17}. The latter is equivalent to the nonlocal problem defined by \eqref{eq:4.12}, which is equivalent to \eqref{eq:3.18}.   Consequently, the variational problem \eqref{eq:4.21} is equivalent to the nonlocal problem \eqref{eq:3.18}.  Hence for the existence of the solution of  \eqref{eq:3.18},  it is sufficient to show the existence of the solution of \eqref{eq:4.21}.
 
We  have the following basic results.
%%%%%%%%%%
%%%%%%%%
\begin{theorem}\label{th:4.1}
The variational problem {\em (\ref{eq:4.21})} has a unique solution $(\mathbf{u}, \theta, v) \in
 \pmb{\mathbb{H}} $. Moreover,  the following estimate holds:
\begin{equation} \label{eq:4.22}
 ( \triple{\mathbf{u}}^2_{|s|, \Omega_-} + \triple{\theta}^2_{|s|, \Omega_-} + \triple{v}^2_{|s|, \Omega_+} )^{1/2}
\leq c_0 \; \frac{|s|^{3/2} } {\sigma {\underline{\sigma}}^{3/2}} ~ \|(d_1, 0, d_3, 0)\|_{X^{\prime}}, 
\end{equation}
where $c_0$ is  a constant depending only on the physical parameters $ \rho_f, \zeta, \eta $. 
\end{theorem} 
 \begin{proof} 
Starting with the system \eqref{eq:4.21}, a simple computation shows that
\begin{align}
\nonumber
\mathrm{Re} \Big\{\frac{\bar{s}}{|s|} ({\mathbf A}_{s} \mathbf u, \overline{\mathbf u})_{\Omega_-}\!\! +  \frac{\zeta}{\eta|s|}(B_{s}\theta ,\overline{\theta})_{\Omega_-} \!\! & + \frac{\bar{s}}{|s|}\rho_f (C_{s, \Omega_+} v, \overline{v})_{\Omega_+} \Big\} \\
\label{eq:4.24}
= \;&  \mathrm{Re}\Big\{\frac{\overline s}{|s|}(d_1,\overline{\mathbf u})_{\Omega_-} + \frac{\zeta}{\eta}\frac{1}{|s|}(d_2,\overline\theta)_{\Omega_-} + \frac{\overline s}{|s|}\rho_f\langle d_3,\overline{\gamma^+v}\rangle_\Gamma \Big\}.
\end{align}
On the other hand, it is not hard to verify that
\begin{align}
\label{eq:4.25} 
\mathrm{Re} \Big\{\frac{\bar{s}} {|s| }~ ({\mathbf A}_{s} \mathbf u,  \overline{\mathbf u} )_{\Omega_-} \Big\}
=\;& \frac{\sigma}{|s|}~ \triple{\mathbf u}^2_{|s|,\Omega_-},\\
\label{eq:4.26}
\mathrm{Re} \Big\{ \frac{\zeta}{\eta} ~ \frac{1}{|s|} (B_{s} \;\theta ,\overline{\theta})_{\Omega_-} \Big\}
\geq\;& \frac{\zeta}{\eta} ~\frac{\sigma}{|s|^2}~ \triple{\theta}^2_{|s|, \Omega_-}, \\
\label{eq:4.27}
\mathrm{Re} \Big\{ \frac{\bar{s}}{|s| }~ \rho_f ~(C_{s, \Omega_+} \;v, \overline{v})_{\Omega_+} \Big\} 
=\;&\rho_f  \frac{\sigma}{|s|} \triple{v}^2_{|s|,\Omega_+} . 
\end{align}
Therefore, combining \eqref{eq:4.25} - \eqref{eq:4.27}, substituting into \eqref{eq:4.24} and recalling that $d_2=0$, it follows that
\begin{eqnarray*}
\frac{\sigma}{|s|} \Big( \triple{\mathbf{u }}^2_{|s|, \Omega_-} +  \frac{\zeta}{\eta} \frac{1}{|s|}~
\triple{\theta}^2_{|s|, \Omega_-} + \rho_f \triple{v}^2_{|s|, \Omega_+}\Big)
&\leq& \Big|(d_1, \overline{\mathbf{u}})_{\Omega_-} + 
\langle  d_3, \overline{\gamma^+ {v}}\rangle_{\Gamma}\Big|.
 \end{eqnarray*}
However, from the definition $\underline\sigma:= \min\{1,\sigma\}$ we see that
the left hand side of this expression satisfies
\[
\frac{\underline\sigma}{|s|}\Big(\triple{\mathbf u}^2_{|s|, \Omega_-}  + \frac{\zeta}{\eta} ~
\triple{ \sqrt{\frac{\underline\sigma}{|s|}}\theta}^2_{|s|, \Omega_-}+ \rho_f  \triple{v}^2_{|s|, \Omega_+} \Big) \leq\; LHS.
\] 
Consequently, we have the estimate 
\begin{equation}\label{eq:4.28}
\Big(\triple{\mathbf{u}}^2_{|s|, \Omega_-} + \triple{\theta}^2_{|s|, \Omega_-} + \triple{v}^2_{|s|, \Omega_+} \Big)^{1/2}
\leq c_0 \frac{|s|^{3/2} }{\sigma {\underline{\sigma}}^{3/2}} \;\|(d_1,0,d_3, 0)\|_{X^\prime},
\end{equation}
where $c_0$ is  a constant depending only on the physical parameters $ \rho_f, \zeta, \eta $. 
In deriving the estimate \eqref{eq:4.28}, we   have tacitly applied the relations \eqref{eq:4.4}, \eqref{eq:4.5} and \eqref{eq:4.6}.
\end{proof}
As we will see  the  estimate (\ref{eq:4.28}) will lead us to show  the invertibility  of the operator  $\pmb{\mathscr{A}} $ in (\ref{eq:3.18}) (or (\ref{eq:4.12}) rather).  In fact,  the following result holds for the operator $\pmb{\mathscr{A}}$ of (\ref{eq:4.12}). 
%%%%%%
%%%%%%
\begin{theorem} \label{thm:4.2}
Let
\begin{align*}
X :=\;& \mathbf{ H}^1(\Omega_-) \times H^1(\Omega_-) \times H^{1/2}(\Gamma) \times H^{-1/2}(\Gamma), \\
X^\prime :=\; &(\mathbf{ H}^1(\Omega_-))^\prime \times (H^1(\Omega_-))^\prime \times H^{-1/2}(\Gamma) \times H^{1/2}(\Gamma), \\ 
X^\prime_0 :=\;& \{ (d_1, d_2 , d_3, d_4) \in X^\prime\,  \big | \,  d_2 = 0~ and~ d_4 = 0\}.
\end{align*}
Then 
$ \pmb{\mathscr{A}} : X \to  X^\prime_0 $  is invertible .  Moreover, we have the estimate:
\begin{equation} \label{eq:4.29}
\| \pmb{\mathscr{A}}^{-1}|_{X_0^{\prime}} \|_{X^\prime,   X}  \leq c_0 \frac{|s|^2}{\sigma\underline\sigma^3} ,
\end{equation}
where $c_0$ is a constant independent of $s$ and $\sigma:= \mathrm{Re}~s > 0$.  
\end{theorem}
\begin{proof}
It was pointed out in  (\ref{eq:4.19}) that 
\[
\gamma^+v=\jump{\gamma v} = \phi \in H^{1/2}(\Gamma), \quad \partial_n^+ v= \jump{\partial_n v} =\lambda \in H^{-1/2}(\Gamma). 
\]
Then we have the estimates (see, e.g. \cite{HSW:2013}):
\begin{equation} \label{eq:4.30}
\|\phi \|^2_{H^{1/2}(\Gamma)} = \|\gamma^+ v\|^2_{H^{1/2}(\Gamma)} \leq c_1 \triple{v}^2_{1, \Omega_+}
\leq c_1 \frac{1}{{\underline{\sigma}}^2} \triple{v}^2_{|s|,  \Omega_+}
\end{equation}
Similarly, an application of Bamberger and Ha-Duong's optimal lifting \cite{BaHa:1986a, BaHa:1986b} leads to the estimate
\[
\|\lambda\|_{H^{-1/2}(\Gamma)} =\|\partial^+_n v \|_{H^{-1/2}(\Gamma)} \leq  c_2 (|s|/\underline\sigma)^{1/2}\triple{v}_{|s|, \Omega_+}.
\]
A detailed proof of this inequality can be found in \cite[Proposition 2.5.2]{Sa:2015}. Thus, we can conclude that
\begin{equation} \label{eq:4.31}
\|\lambda\|^2_{H^{-1/2}(\Gamma)} \leq  c_2^2 \frac{|s|} {\underline\sigma} 
 \triple{v}^2_{|s|, \Omega_+}.
\end{equation}
From (\ref{eq:4.30}) and (\ref{eq:4.31}), we obtain the estimates
\begin{equation}\label{eq:4.32}
\frac{1}{2} \Big(\frac{\underline{\sigma}^2}{c_1}  \|\phi \|^2_{H^{1/2}(\Gamma)} + 
 \frac {\underline\sigma}{c_2^2|s|}  \|\lambda\|^2_{H^{-1/2}(\Gamma)}\Big)
 \leq \triple{v}^2_{|s|, \Omega_+}.
 \end{equation}
As a consequence of \eqref{eq:4.22},  it follows that
\[
\underline\sigma^2 \triple{\mathbf{u}}^2_{1, \Omega_-} + { \underline{\sigma}} \triple{\theta}^2_{1, \Omega_-}
+ c \Big(\underline{\sigma}^2  \|\phi \|^2_{H^{1/2}(\Gamma)} + 
 \frac {\underline\sigma}{|s|}  \|\lambda\|^2_{H^{-1/2}(\Gamma)}\Big)\leq 
\Big( c_0 \; \frac{|s|^{3/2} } {\sigma {\underline{\sigma}}^{3/2}}\|(d_1, 0, d_3,  0)\|_{X^{\prime}}\Big)^2.
\]
A simple manipulation shows that the left hand side of the previous expression is such that
\[
c \frac{\underline\sigma^3}{|s|} \Big(\triple{\mathbf{u}}^2_{1, \Omega_-} + \triple{\theta}^2_{1, \Omega_-}
+  \|\phi \|^2_{H^{1/2}(\Gamma)} + \|\lambda\|^2_{H^{-1/2}(\Gamma)}\Big) \leq LHS, 
\]
which  implies that 
\[
\Big( \triple{\mathbf{u}}^2_{1, \Omega_-} + \triple{\theta}^2_{1, \Omega_-} +
 \|\phi \|^2_{H^{1/2}(\Gamma)} +  \|\lambda\|^2_{H^{-1/2}(\Gamma)}\Big)^{1/2}
 \leq c_0 \frac{|s|^2}{\sigma\underline\sigma^3} \|(d_1,0, d_3, 0)\|_{X^\prime}, 
\]
with constant $c_0$ independent of $s$ and $\sigma$.
\end{proof}
We remark that in view of (\ref{eq:3.7}), we see that $\mathbf{u}, \theta$ and $ v$  are solutions of the system 
 \begin{equation}\label{eq:4.34}
\begin{pmatrix}
\mathbf{u} \\
\theta\\
v\\
\end{pmatrix}
= \left (
\begin{matrix}  % or pmatrix or bmatrix or Bmatrix or ...
      1& 0 & 0 & 0\\
      0 & 1& 0 & 0 \\
      0 & 0  &\mathcal{D}(s) & - \mathcal{S}(s)\\
    \end{matrix}
    \right )  \pmb{\mathscr{A}}^{-1} 
    \begin{pmatrix}
d_1 \\
0\\
d_3\\
0\\
\end{pmatrix}.
 \end{equation}
With the properties of solutions available  in the transformed domains,  we are now in a position to estimate  the corresponding properties of solutions in the time domain based on Lubich's Convolution Quadrature \cite{Lu:1994}, introduced in the early 90's for treating time-dependent boundary integral equations of convolution type. 

An essential feature of this approach is that estimates of properties of  solutions in the time domain are obtained  without the need for applying the inverse Laplace transform. Instead, the crucial result described on Proposition \ref{prop:4.3} below is employed to retrieve time domain estimates from those obtained in the Laplace domain.

Before presenting the aforementioned result we must introduce some notation.  For Banach spaces $X$ and $Y$,  let  $\mathcal{B}(X, Y)$ denote the set of bounded linear operators from $X$ to $Y$. We say that an analytic function $A :  \mathbb{C}_+ \rightarrow   \mathcal{B}(X, Y)$ is an element of the class of symbols $\mathbf{Sym} (\mu, \mathcal{B}(X, Y))$ if there exist $\mu \in \mathbb{R}$ and $m\geq 0$ such that  
$$ \|A(s)\|_{X,Y}  \le C_A(\mathrm{Re}~s) |s|^{\mu} \quad \mbox{for}\quad s \in \mathbb{C}_+ , $$
where $C_A : (0, \infty) \rightarrow (0, \infty) $ is a non-increasing function such that 
$$ C_A(\sigma) \le \frac{ c}{\sigma^m} , \quad \forall \, \sigma \in ( 0, 1].$$
In order to make the statement of the time-domain estimates more compact, we will make use of the regularity spaces
\begin{equation}
\label{eq:wk+}
W^k_+( \mathcal{H}):= \Big\{ w \in \mathcal{C}^{k-1}(\mathbb{R}; \mathcal{H}) : w ~\equiv 0~ in ~(-\infty,0), w^{(k)}  \in L^1 (\mathbb{R}; \mathcal{H})\Big \}  
\end{equation}
where $\mathcal{H}$ denotes a Banach space. 
\begin{proposition}[\cite{Sa:2015, LaSa:2009b}]
Let $A = \mathcal{L}\{a\} \in \mathbf{Sym}(\mu, \mathcal{B}(X, Y))$ with $\mu \ge 0$ and let 
$$ k:=\lfloor \mu +2 \rfloor \quad \varepsilon := k - (\mu +1) \in (0, 1].
$$ 
If $ g \in W_+^k(\mathbb R,X)$, then $a* g \in \mathcal{C}(\mathbb{R}, Y)$ is causal and 
$$ \| (a*g)(t) \|_Y  \le 2^{\mu+1} C_{\varepsilon} (t) C_A (t^{-1}) \int_0^1 \|(\mathcal{P}_kg)(\tau) \|_X \; d\tau,
$$
where 
$$ C_{\varepsilon} (t) := \frac{t^\varepsilon }{\pi \varepsilon },
\qquad\mbox{and}\qquad
(\mathcal{P}_kg) (t) = \sum_{\ell=0}^k \left ( \begin{array}{c} k \\ \ell \\
\end{array} \right ) g^{(\ell)}(t). $$  \label{prop:4.3}
\end{proposition}
\textbf{Remark:} The proof of this result can be found in full detail in \cite[Proposition 3.2.2]{Sa:2015}.\\

As an immediate consequence of this result, we see from Theorem \ref{thm:4.2} that  $ \pmb{\mathscr{A}}^{-1}|_{X_0^{\prime}} \in \mathbf{Sym}(2, \mathcal{B}(X^\prime, X))$.  Moreover, $ k= 4$ and $  \varepsilon = 1$ and we thus  have the estimate:
  %%%%%%%%%
  %%%%%%%
\begin{theorem} \label{thm:4.4} 
Let  ${\bf D}(t) := \mathcal{L}^{-1}\{ (d_1, 0, d_3, 0)^\top\}$  belong to $W_+^4(\mathbb R, X')$. Then
\[
(\mathbf{U}, \Theta, \mathcal{L}^{-1} \{\phi\} , \mathcal{L}^{-1}\{ \lambda \})^\top \in \mathcal{C}([0,T],X)
\]
and there exists $c>0$ depending only on the geometry such that
\begin{equation}  \label{eq: 4.35}
\| (\mathbf{U}, \Theta, \mathcal{L}^{-1} \{\phi\} , \mathcal{L}^{-1}\{ \lambda \})^\top(t) \|_X  \leq c\,t^2 \max \{1, t^{3}\}\!\! \int_0^t \| (\mathcal{P}_4\mathbf D) (\tau)\|_{X^{\prime}}\; d \tau. 
\end{equation}
\end{theorem}
Similarly, in view of Theorem \ref{th:4.1}, applying Proposition \ref{prop:4.3} with $\mu = 3/2 , k= 3,  \varepsilon = 1/2 $ to the elastic, thermal, and potential fields leads to:
\begin{theorem}\label{th:4.5} 
Let  $\pmb{\mathbb{H}} :=  {\mathbf H}^1(\Omega_-) \times H^1(\Omega_-) \times H^ 1(\Omega_+)$ and
\[
\mathbf{D}(t) := \mathcal{L}^{-1} \{ ( d_1, 0, d_3, 0)^{\top} \}(t)\in  W_+^3(\mathbb R,X').
\]
Then $(\mathbf{U}, \Theta, V)  $ belongs to $C( [0, T], \pmb{\mathbb{H}} ) $  and there holds the estimate
\begin{equation}
\label{eq:3.36}
\|(\mathbf{U}, \Theta, V)  (t)\|_{\pmb{\mathbb{H}} } \leq ~C_0~ t^{3/2} ~\max \{1, t^{5/2} \} \!\!\int_0^t
 \| ( \mathcal{P}_3\mathbf D) (\tau)\|_{X^{\prime} }\;  d \tau  
\end{equation}
for some constant $C_0 > 0$.
\end{theorem}
%
% ====================================================
\section{Semi-discrete error estimates}\label{sec:5}
% ====================================================
%
In this section, we discuss the results concerning the discretization of (\ref{eq:3.18}).  We begin with the Galerkin semi discretization in space of the system of equations. Let 
\[
{\mathbf V}_h \subset \mathbf H^1(\Omega_-), \quad W_h \subset H^1(\Omega_-), \quad Y_h \subset H^{1/2}(\Gamma),  \quad X_h \subset H^{-1/2} (\Gamma)
\]
be families of finite dimensional subspaces. We say $ (\mathbf u^h, \theta^h, \phi^h, \lambda^h) \in   
  {\mathbf V}_h \times W_h  \times Y_h \times X_h $ is a Galerkin solution of \eqref{eq:3.18} if it satisfies the  Galerkin equations: 
\begin{equation}\label{eq:5.1}
\Big(\pmb{\mathscr{A}} (\mathbf u^h, \theta^h, \phi^h, \lambda^h)^\top,(\mathbf v,\vartheta,\psi,\eta)^\top\Big) = \Big((d_1,d_2,d_3,d_4)^\top,(\mathbf v,\vartheta,\psi,\eta)^\top\Big)
\end{equation}
for all  $(\mathbf v, \vartheta, \psi, \eta) \in  {\mathbf V}_h \times W_h  \times Y_h \times X_h $.  Again, we multiply \eqref{eq:5.1} by the diagonal matrix $\pmb{\mathscr{Z}}(s)$ in \eqref{eq:4.7} and consider the Galerkin equations for  the modified equation \eqref{eq:4.12}:
\begin{equation}\label{eq:5.2}
\Big(\pmb{\mathscr{B}} (\mathbf u^h, \theta^h, \phi^h, \lambda^h)^\top,(\mathbf v,\vartheta,\psi,\eta)^\top\Big) = \Big((d_1,\frac{\zeta}{\eta}~|s|^{-1}~d_2,\rho_fd_3,\rho_fd_4)^\top,(\mathbf v,\vartheta,\psi,\eta)^\top\Big)
\end{equation}
for all  $(\mathbf v, \vartheta, \psi, \eta) \in  {\mathbf V}_h \times W_h  \times Y_h \times X_h $.   Solutions 
 of Galerkin equations of (\ref{eq:5.2}) can be established in the same manner as  the  exact solutions of the system \eqref{eq:4.21}.  We will not repeat the process and consider only the error estimates here. 
 
We note that if $({\mathbf u}^h, \theta^h, \phi^h, \lambda^h) \in {\mathbf V}_h \times W_h  \times Y_h \times X_h $ is a Galerkin solution of (\ref{eq:5.2}), then 
\begin{equation} \label{eq:5.3}
 v^h := \mathcal{D}(s) \phi^h - \mathcal{S}(s) \lambda^h \; \in H^1(\mathbb R^3\setminus\Gamma) 
\end{equation}
satisfies
\begin{alignat}{6}
\label{eq:5.4}
-\Delta v^h + (s/c)^2v^h = \; & 0 &\quad& \mbox{in } \quad \mathbb{R}^3 \setminus \Gamma,\\
\nonumber
\jump{\gamma v^h} = \;& \phi^h &\quad& \phi^h \in Y_h \subset H^{1/2}(\Gamma), \\
\nonumber
\jump{\partial v^h} = \;& \lambda^h &\quad& \lambda^h \in X_h \subset H^{-1/2}(\Gamma).
\end{alignat}
Now set 
\[
V_h := \{ w \in H^1(\mathbb{R}^3 \setminus \Gamma ) : \jump{\gamma w} \in Y_h,
 \gamma^-w  \in X^\circ_h \}.
\]
Here and in the sequel, the upper script $^\circ$ will be used to denote the annihilator (or polar set) of a given Banach space. In this particular case, 
\[
X_h^\circ := \{ w \in X^\prime_h : \langle v, w \rangle  = 0 \quad \forall v \in X_h \}.
\]
Applying Green's formula to (\ref{eq:5.4}), we obtain for  $w \in V_h$,
\begin{equation}\label{eq:5.5}
\int_{\Gamma} \partial_n^+ v^h~ \jump{\gamma w}\; d_{\Gamma} = 
- c_{\mathbb R^3\setminus\Gamma}\;(v^h, w\; ; s) - \int_{\Gamma}\jump{\partial_n v^h}\; \gamma^- w~ d_{\Gamma}. 
\end{equation}
As a consequence, we see that $ ( \mathbf u^h, \theta^h, v^h)  \in  {\mathbf V}_h \times W_h  \times V_h  $ satisfies the variational equations
\begin{alignat}{6} 
 a(\mathbf{u}^h, \mathbf v;  s )
  -  \zeta ( \theta^h, div \; \mathbf v )_{\Omega_-}\!  + s\rho_f\langle \jump{\gamma v^h}\mathbf n , \gamma^- \mathbf v \rangle_{\Gamma}  = \; &  (d_1, \mathbf v)_{\Omega_-} &\quad& \forall ~\mathbf v \in {\mathbf V}^h, \nonumber \\
\frac{s}{|s|}  \zeta ( div~ \mathbf u^h,  \vartheta)_{\Omega_-}\! +  \frac{\zeta}{\eta} \frac{1}{|s|} b(\theta^h, \vartheta; s)  =\;&  \frac{\zeta}{\eta} \frac{1}{|s|}(d_2 , \vartheta)_{\Omega_-}&\quad& \forall ~ \vartheta \in W_h, \label{eq:5.6}\\
- s\rho_f  \langle  \gamma^- \mathbf{u}^h  , \jump{\gamma w}\mathbf{n}\rangle_{\Gamma} + \rho_f  c_{\mathbb R^3\setminus \Gamma} ( v^h, w; s ) =\;&
\rho_f  \langle  d_3, \jump{\gamma w}\rangle_{\Gamma} &\quad& \forall w ~ \in V_h.
\nonumber
\end{alignat}
For the error estimate,  we need to compare $(\mathbf u^h, \theta^h, v^h)$ with  the exact solution $(\mathbf u, \theta, v)$ of the transmission problem.  The exact solution $(\mathbf u, \theta, u) \in \mathbf H^1(\Omega_-) \times H^1(\Omega_-) \times H^1(\mathbb R^3 \setminus \Gamma)$
satisfies the variational equation \eqref{eq:4.21}, which can be put in the same form as \eqref{eq:5.6} by choosing test functions in the discrete subspaces. This implies that 
\begin{alignat}{6} 
 a(\mathbf{u}^h\!- \mathbf u , \mathbf v;  s )
  -  \zeta( \theta^h \!- \theta, div \; \mathbf v )_{\Omega_-} \! + s\; \rho_f \langle\jump{\gamma (v^h \!-v)} \mathbf n , \gamma^- \mathbf v \rangle_{\Gamma}  &\; =  0 &\quad& \forall \mathbf v \in {\mathbf V}^h, \nonumber \\
\frac{s}{|s|}  \zeta( div (\mathbf u^h \!- \mathbf u),  \vartheta )_{\Omega_-}\! +  \frac{\zeta}{\eta} \frac{1}{|s|} b(\theta^h\!- \theta, \vartheta; s) &\; = 0 &\quad& \forall \vartheta \in W_h, \label{eq:5.7}\\
\nonumber
-  s\rho_f \langle  \gamma^- (\mathbf{u}^h\!- \mathbf u) , \jump{\gamma w}\mathbf{n}\rangle_{\Gamma} + \rho_f  c_{\mathbb R^3\setminus \Gamma} ( v^h\! - v, w; s ) \quad &  & & \\
\nonumber
+ \rho_f \langle \jump{\partial_n(v^h\! - v)}, \gamma^- w \rangle_{\Gamma} &\; = 0  &\quad& \forall w \in V_h. 
\nonumber
\end{alignat}
The significance  of \eqref{eq:5.7} is that it indicates that the Galerkin solutions are the best possible approximations of the exact solution in the finite dimensional subspaces with respect to the inner products defined by the underlying bilinear forms. However, it is worth emphasizing that the errors $ (\mathbf u^h  -\mathbf u )$ and $(\theta^h - \theta )$ do not belong to the discrete function spaces. In  order to justify \eqref{eq:5.7} as a proper variational formulation, we will make use of the spaces of constants and of  infinitesimal rigid motions
\begin{align}
\mathfrak{R}_{\mathbf u}  := \;&\{ \mathbf m \in \mathbf H^1(\Omega_-): ( \mathbf C \boldsymbol\varepsilon
(\mathbf m), \bm{\varepsilon}(\mathbf m))_{\Omega_-} = 0 \}, \\
{\mathfrak R}_{\theta}  := \;& \{ m \in H^1(\Omega_-): ( \nabla m, \nabla m)_{\Omega_-} = 0 \},
\end{align}
which in what follows will always be assumed to be contained on the discrete subspaces $\mathbf V_h$ and $W_h$ respectively. We now define the elliptic projection on displacement fields
\begin{alignat}{6} \label{eq:5.8}
\mathbf P_h: \mathbf H^1(\Omega_-) \longrightarrow \;&\mathbf V_h  \subset \mathbf H^1(\Omega_-) & \\
\nonumber
(\mathbf C\boldsymbol\varepsilon(\mathbf P_h \mathbf u),  \bm{\varepsilon} (\mathbf v^h))_{\Omega_-} = \;&  (\mathbf C\boldsymbol\varepsilon(\mathbf u), \bm{\varepsilon} (\mathbf v^h ) )_{\Omega_-} \qquad & \forall \mathbf v^h\in \mathbf V_h,\\
\nonumber
(\mathbf P_h \mathbf u, \mathbf m)_{\Omega_-} =\;& (\mathbf u, \mathbf m)_{\Omega_-}
\qquad & \forall\mathbf m\in \mathfrak R_{\mathbf u}.
\end{alignat}
This projection is well defined thanks to Korn's second inequality and can be alternatively introduced as the orthogonal projection onto $\mathbf V_h$ with respect to the non-standard (but equivalent) inner product in $\mathbf H^1(\Omega)$
\[
(\mathbf C\boldsymbol\varepsilon(\mathbf u),\boldsymbol\varepsilon(\mathbf v))_{\Omega_-}+(\mathbf P_{\mathfrak R_{\mathbf u}}\mathbf u,\mathbf P_{\mathfrak R_{\mathbf u}}\mathbf v)_{\Omega_-},
\]
where $\mathbf P_{\mathfrak R_{\mathbf u}}$ is the $\mathbf L^2(\Omega_-)$ orthogonal projection onto $\mathfrak R_{\mathbf u}$. This implies that the approximation error $\|\mathbf u-\mathbf P_h\mathbf u\|_{1,\Omega_-}$ is equivalent to the best approximation error in $\mathbf H^1(\Omega_-)$ by elements of $\mathbf V_h$. Similarly, we can introduce a projection on the discrete scalar fields
\begin{alignat}{6}  \label{eq:5.10}
Q_h: H^1(\Omega_-) \longrightarrow \;& W_h  \subset  H^1(\Omega_-) & \\
\nonumber
(\nabla (Q_h \theta), \nabla \vartheta^h)_{\Omega_-} = \;&  (\nabla \theta, \nabla\vartheta^h )_{\Omega_-}
\qquad  & \forall\vartheta^h\in W_h,\\
\nonumber
(Q_h \theta,  m)_{\Omega_-} = \;&(\theta, m)_{\Omega_-}\qquad &\forall m\in \mathfrak R_\theta.
\end{alignat}
Note that the reason to introduce these projections is to avoid having additional mass terms arising from the full Sobolev norm in the associated error equations.

In terms of  the elliptic projection $\mathbf P_h$, we can define
\[
\mathbf e_{\mathbf u}^h := \mathbf u^h - \mathbf P_h\mathbf u, \qquad \mathbf r_{\mathbf u}^h := \mathbf P_h \mathbf u - \mathbf u,
\]
so that the error function $\mathbf u^h  - \mathbf u$  can be decomposed as
\[
\mathbf u^h - \mathbf u = (\mathbf u^h - \mathbf P_h  \mathbf u ) + (\mathbf P_h \mathbf u - \mathbf u) = \mathbf e_{\mathbf u}^h + \mathbf r_{\mathbf u}^h.
\]
As a consequence
\[
a (\mathbf u^h - \mathbf u,  \mathbf v^h ; s) = a (\mathbf e_{\mathbf u}^h \;,  \mathbf v^h; s)
+ s^2 \rho_\Sigma (\mathbf r_{\mathbf u}^h, \mathbf v^h)_{\Omega_-}.    
\]
We may decompose the error $\theta^h - \theta $ in a similar manner by letting
\[
e^h_{\theta}:= \theta^h - Q_h\theta, \qquad  r^h_{\theta} := Q_h\theta - \theta,
\]
so that
\begin{eqnarray*}
\theta^h - \theta & = & (\theta^h - Q_h\theta) + (Q_h\theta - \theta)= e^h_{\theta} + r^h_{\theta}. \\
b (\theta^h -\theta, \vartheta^h; s ) & = & b (e^h_{\theta} , \vartheta^h; s) + (s/\kappa) (r^h_{\theta}, \vartheta^h)_{\Omega_-}.
\end{eqnarray*}
Finally, we define
\[
e_v^h:= \mathcal D(s) (\phi^h-\phi) - \mathcal S(s)(\lambda^h-\lambda) \qquad \hbox{ in }\; \mathbb R^3\setminus\Gamma.
\]
This leads to the variational formulation for the error functions $(\mathbf e_{\mathbf u}^h,  e^h_{\theta}, e^h_v) \in \mathbf V_h \times W_h \times H^1(\mathbb R^3 \setminus \Gamma)$.
%%%%%%%%
%%%%%%
\begin{theorem} \label{tn:5.1}
The error functions $ (\mathbf e^h_{\mathbf u},  e^h_{\theta}, e^h_v) \in  \mathbf V_h \times W_h \times H^1(\mathbb R^3 \setminus \Gamma) $ satisfy the variational formulation
\begin{alignat*}{6} 
(\gamma^- e^h_u,\;  \jump{\gamma e^h_u}\; +\phi,\; \jump{\partial_n e^h_v} + \lambda) \in \; & X^\circ_h \times Y_h\times X_h, &  \\
\mathcal{A}((\mathbf e^h_{\mathbf u}, e^h_{\theta}, e^h_v), ( \mathbf v,  \vartheta, w) ;s ) =\; & \ell((\mathbf v, \vartheta, w) ; s) & \quad\forall (\mathbf v, \vartheta, w) \in 
\mathbf V_h \times W_h \times V_h.  
\end{alignat*}
Where the bilinear form $\mathcal A$ is defined by
\begin{alignat}{6}
\nonumber
\mathcal{A}((\mathbf e^h_{\mathbf u}, e^h_{\theta}, e^h_v), ( \mathbf v,  \vartheta, w) ;s ) :=\;&  a (\mathbf e^h_{\mathbf u}, \mathbf v;  s ) + \frac{s}{|s|}  \zeta( div ~ \mathbf e^h_\mathbf u,  \vartheta)_{\Omega_-} -  s\rho_f\langle  \gamma^- \mathbf e^h_{\mathbf u} \; , \jump{\gamma w}\mathbf{n}\rangle_{\Gamma} \nonumber \\ 
 &   +  \frac{\zeta}{\eta}  \frac{1}{|s|}  b\; (e^h_{\theta}, \vartheta; s) - \zeta (e^h_{ \theta},  div \; \mathbf v )_{\Omega_-}   \nonumber \\
 & +\rho_f  c_{\mathbb R^3\setminus \Gamma} \; ( e^h_v, w; s )  + s \rho_f\langle \jump{\gamma e^h_v} \; \mathbf n , \gamma^- \mathbf v \rangle_{\Gamma}, \label{eq:5.14}
\end{alignat}
and the functional $\ell$ is given by
\begin{align}
\nonumber 
\ell ((\mathbf v, \vartheta, w ) ; s) :=\; & -  s^2 \rho_\Sigma(\mathbf r^h_{\mathbf u}, \mathbf v )_{\Omega_-} + \zeta(r^h_{\theta}, div~\mathbf v)_{\Omega_-}  -\frac{s}{|s|} \zeta ( div~\mathbf r^h_{\mathbf u}, \vartheta)_{\Omega_-}\\ 
\label{eq:5.15}
& -\frac{\zeta}{\eta} \frac{s}{|s|}\frac{1}{\kappa}( r^h_{\theta}, \vartheta)_{\Omega_-} +s \rho_f \langle \gamma^- \mathbf r^h_{ \mathbf u}, \jump{\gamma w} \mathbf n \rangle_{\Gamma} + \rho_f \langle \lambda, \gamma^-w \rangle_{\Gamma}.  
 \end{align}
\end{theorem} 
\begin{proof}
The bilinear form $\mathcal A$ follows easily from the left hand side of  \eqref{eq:5.6} replacing $(\mathbf u^h - \mathbf u, \theta_h - \theta,  v^h -  v) $ by  $( \mathbf e^h_{\mathbf u} +\mathbf r^h_\mathbf u ,  e^h_{\theta} + r^h_{\theta},  e^h_v) $ and taking special care of the term $e^h_{\theta}$.

From Green's formula \eqref{eq:5.5}, we have 
\[
\langle \partial_n^+ e^h_v~, \jump{\gamma w}\rangle = 
- c_{\mathbb R^3\setminus \Gamma}\;(e^h_v, w\; ; s) - \langle\jump{\partial_n e_v^h},  \gamma^- w~ \rangle. 
\]
But equations \eqref{eq:4.17} and  \eqref{eq:4.18} imply 
\[
-s  \mathbf n\cdot (e^h_{\mathbf u} + \mathbf r^h_{\mathbf u})  -\partial_n^+ e^h_u \in Y_h^\circ, \quad \hbox{and} \quad \gamma^- e^h \in X^\circ_h.
\]
Hence, 
\[
-  s\rho_f \langle  \gamma^- \mathbf e^h_{\mathbf u} , \jump{\gamma w}\mathbf{n} \rangle_{\Gamma} 
+ \rho_f  c_{\mathbb R^3\setminus \Gamma} \; ( e^h_v, w; s )= s \rho_f\langle \gamma^- \mathbf r^h_{ \mathbf u}, \jump{\gamma w} \mathbf n\rangle_{\Gamma}
-\rho_f\langle \jump{\partial e^h_v}, \gamma^- w \rangle_{\Gamma}.
\]
We can rewrite the last term on the right hand side as 
\[
-\rho_f \langle \jump{\partial e^h_v}, \gamma^- w \rangle_{\Gamma}= -\rho_f \langle \jump{\partial e^h_v}+ \lambda - \lambda, \gamma^- w \rangle_{\Gamma} = \rho_f \langle \lambda, \gamma^- w \rangle_{\Gamma}, 
\]
Where we have used that $\lambda^h = \jump{\partial_n e^h_v} + \lambda \in X_h $, and $ \gamma^- w \in X^\circ_h $ but $ \lambda \not \in X_h$. This completes the proof. 
\end{proof}
Following arguments similar as those employed in the proof of Theorem \ref{th:4.1}, we can obtain the  error estimate. In the following, for  simplicity, let
\[
\triple{(\mathbf e^h_{\mathbf u},\; e^h_{\theta}, \; e^h_v)}^2_{|s|} := \triple{ \mathbf e^h_{\mathbf u}}^2_{|s|, \Omega_-} + \triple{ e^h_{\theta} }^2_{|s|, \Omega_-} + \triple{ e^h_v }^2_{|s|, \mathbb{R}^3 \setminus \Gamma} .
\]
%%%%%%%%%%%%
%%%%%%%%%%%%%
%%%%%%
\begin{theorem}  \label{th:5.2}
For $(\mathbf u, \theta, \phi, \lambda) \in \mathbf H^1(\Omega_-) \times H^1(\Omega_-) \times H^{1/2}(\Gamma) \times H^{-1/2} (\Gamma) $, there holds the error estimate:
\begin{equation}\label{eq:5.18}
\triple{(\mathbf e^h_{\mathbf u}, e^h_{\theta}, e^h_v)}_{|s|} \leq  C \frac{|s|^2}{\sigma {\underline{\sigma}}^3} \Big(\| \lambda \|_{-1/2,\; \Gamma}^2 + \|s^2 \mathbf r^h_{\mathbf u} \|_{1, \Omega_-} + \|s \mathbf r^h_{\mathbf u} \|_{1, \Omega_-} + \|\mathbf r^h_{\mathbf u}\|_{1, \Omega_-} + \|r^h_{\theta}\|_{1, \Omega_-}\Big)
\end{equation}
where the constant $C$ depends only on the geometry, and physical parameters. 
\end{theorem} 
\begin{proof}
It is easy to see from the definition of the bilinear form $\mathcal{A}$ in \eqref{eq:5.14} that there is a constant C depending only on the geometry and physical parameters such that
\begin{equation}\label{eq:5.16}
|\mathcal{A} (e^h_{\mathbf u}, e^h_{\theta}, e^h_v), (\mathbf v, \vartheta, w); s)| \le C \frac{|s|}{\sigma \underline{\sigma}}
\triple{(e^h_{\mathbf u}, e^h_{\theta}, e^h_v)}_{1}~ \triple{(\mathbf v, \vartheta, w)}_{|s|} .
\end{equation}
We also need the estimate for the functional
\begin{equation}\label{eq:5.17}
|\ell((\mathbf v, \vartheta, w);s)| \le \frac{C}{\underline {\sigma}} \Big(\| \lambda \|_{-1/2,\; \Gamma} + \|s^2 \mathbf r^h_{\mathbf u} \|_{1, \Omega_-} + \|s \mathbf r^h_{\mathbf u} \|_{1, \Omega_-} + \|\mathbf r^h_{\mathbf u}\|_{1, \Omega_-}+ \|r^h_{\theta}\|_{1, \Omega_-}\Big)\triple{(\mathbf v, \vartheta, w)}_{|s|}.
\end{equation}
For $\phi \in H^{1/2} (\Gamma)$, we pick a lifting $v_{\phi}  \in H^1(\mathbb{R}^3 \setminus \Gamma)$ such that  $\gamma^+ v_{\phi}  = \phi, \gamma^-v_{\phi} =  0$. Thus,
\[
\| v_{\phi} \|_{1, \mathbb{R}^3 \setminus \Gamma} \le C \| \phi \|_{1/2, \Gamma}.
\]
Since  $(\mathbf e^h_{\mathbf u}, e^h_{\theta}, e^h_v + v_{\phi} ) \in \mathbf V_h \times W_h \times V_h$,  it follows from equations \eqref{eq:4.25}-\eqref{eq:4.27} that
\begin{align*}
\frac{\sigma \underline{\sigma} } {|s|^2}\triple{(\mathbf e^h_{\mathbf u}, e^h_{\theta}, e^h_v + v_{\phi} )}^2_{|s|} 
\le \;& 
\big |\mathcal{A}(( \mathbf e^h_{\mathbf u}, e^h_{\theta}, e^h_v + v_{\phi} ), (\mathbf e^h_{\mathbf u}, e^h_{\theta}, e^h_v + v_{\phi} ); s) \big |\\
 =\; & \big | \ell ( ( \mathbf e^h_{\mathbf u}, e^h_{\theta}, e^h_v + v_{\phi} ) ; s) + \mathcal{A} ((\mathbf 0, 0, v_{\phi}),  \mathbf e^h_{\mathbf u}, e^h_{\theta}, e^h_v + v_{\phi} ) ;s ) \big | \\
\le\; & \frac{C}{\underline\sigma}\triple{(\mathbf e^h_{\mathbf u},\; e^h_{\theta}, \; e^h_v+v_\phi)}_{|s|} \Big(\|s^2 \mathbf r^h_{\mathbf u} \|_{1, \Omega_-}\!\! + \|s \mathbf r^h_{\mathbf u} \|_{1, \Omega_-}\!\! + \|\mathbf r^h_{\mathbf u}\|_{1, \Omega_-}\!\!\\
&\;\; \qquad \qquad \qquad \qquad + \|r^h_{\theta}\|_{1, \Omega_-}\! +\| \lambda \|_{-1/2,\; \Gamma} + \!\frac{|s|}{\sigma}  \| v_{\phi} \|_{\mathbb{R}^3 \setminus \Gamma} \Big) \\
\le\;& \frac{C}{\underline{\sigma}^2}\triple{(\mathbf e^h_{\mathbf u},\; e^h_{\theta}, \; e^h_v + v_{\phi})}_{|s|} \Big(\|s^2 \mathbf r^h_{\mathbf u} \|_{1, \Omega_-}\!\! + \|s \mathbf r^h_{\mathbf u} \|_{1, \Omega_-}\!\! +  \|\mathbf r^h_{\mathbf u}\|_{1, \Omega_-}  \\ 
&\;\; \qquad \qquad \qquad \qquad +\|r^h_{\theta}\|_{1, \Omega_-} +   \| \lambda \|_{-1/2,\; \Gamma} + \|s \phi\|_{1/2, \Gamma}\Big).  
\end{align*}
And the result follows from this relation and the observation that
\[
\triple{(\mathbf 0,0,v_\phi)}\leq \frac{C}{\underline\sigma}\|s\phi\|_{1/2,\Gamma}
.\] 
\end{proof}
We are now in the position to establish the following result.
%%%%%%%
\begin{corollary}\label{co:5.3}
Let $(\mathbf u, \theta, \phi , \lambda) \in \mathbf H^1(\Omega_-) \times H^1(\Omega_-) \times H^{1/2}(\Gamma) \times H^{-1/2}(\Gamma) $ be the unique solution of the problem  {\em (\ref{eq:4.12})} and let  $(\mathbf u^h, \theta^h, \phi^h, \lambda^h) $ be their corresponding Galerkin solutions of 
{\rm (\ref{eq:5.6})}, then we  have the estimates 
\begin{subequations}%\label{eq:5.19}
\begin{align}
\nonumber
\triple{(\mathbf e^h_{\mathbf u}, e^h_{\theta},  e^h_v)}_1   +  \| \phi^h - \phi \|_{ 1/2, \Gamma} \le\;& C \frac{|s|^2}{\sigma \underline{\sigma}^4}
  \Big( \| s \phi \|_{1/2, \Gamma} +   \| \lambda \|_{-1/2,\; \Gamma}\\
  \label{eq:5.19a}
   & +  \| s^2 \mathbf r^h_{\mathbf u} \|_{1, \Omega_-} + \|s \mathbf r^h_{\mathbf u} \|_{1, \Omega_-} +  \|\mathbf r^h_{\mathbf u}\|_{1, \Omega_-} + \|r^h_{\theta}\|_{1, \Omega_-} \Big)\\
\nonumber
\|\lambda^h - \lambda \|_{-1/2, \Gamma} \le\; & C \frac{|s|^{5/2}}{\sigma\underline{\sigma}^{7/2}} \Big(\| s \phi \|_{1/2, \Gamma} +   \| \lambda \|_{-1/2,\; \Gamma} \\
\label{eq:5.19b}
& + \| s^2 \mathbf r^h_{\mathbf u} \|_{1, \Omega_-} + \|s \mathbf r^h_{\mathbf u} \|_{1, \Omega_-} +  \|\mathbf r^h_{\mathbf u}\|_{1, \Omega_-} + \|r^h_{\theta}\|_{1, \Omega_-} \Big).
 \end{align}
 \end{subequations}
\end{corollary}
Regarding the proof of this result, we would only like to point out that the estimate \eqref{eq:5.19a} follows from a combined application of \eqref{eq:5.18} in Theorem \ref{th:5.2}, and \eqref{eq:4.30} by making use of the jump condition $\phi^h -\phi =\jump{\gamma e^h_v}$. On the other hand, in order to establish the estimate \eqref{eq:5.19b}, one has to recall that  $\lambda -\lambda^h =\jump{\partial_{\nu} e^h_v}$ and therefore an application of  \eqref{eq:4.31} combined with \eqref{eq:5.18} yields the desired inequality.

Corollary \ref{co:5.3} has the awkward aspect of being an error estimate where part of the right-hand side (the terms $\|s\phi\|_{1/2,\Gamma}$ and $\|\lambda\|_{-1/2,\Gamma}$) does not converge to zero. We now clarify why this is not so. Consider the best approximation operators 
\[
\Pi_{X_h} : H^{-1/2} (\Gamma) \mapsto X_h, \quad \mbox{and} \quad \Pi_{Y_h} : H^{1/2} (\Gamma)\mapsto Y_h.
\]
If we create data for the problem so that the exact solution is $(\mathbf 0,0,\Pi_{Y_h}\phi,\Pi_{X_h}\lambda)$, then the associated numerical solution will be the exact solution and there will be no error in the method. Therefore, by linearity, we can use $(\mathbf u,\theta,\phi-\Pi_{Y_h}\phi,\lambda-\Pi_{X_h}\lambda)$ as exact solution in Corollary \ref{co:5.3} and the numerical solution will be $(\mathbf u^h,\theta^h,\phi^h-\Pi_{Y_h}\phi,\lambda^h-\Pi_{X_h}\lambda)$. Consequently, we can substitute $\|s\phi\|_{1/2,\Gamma}+\|\lambda\|_{-1/2,\Gamma}$ by $\| s(\phi-\Pi_{Y_h}\phi)\|_{1/2,\Gamma}+\|\lambda-\Pi_{X_h}\lambda\|_{-1/2,\Gamma}$ in the right-hand-side of \eqref{eq:5.19a}-\eqref{eq:5.19b}.

If we now apply Proposition \ref{prop:4.3} to Corollary \ref{co:5.3}, we may obtain the following estimates in the time domain.
\begin{corollary}
\label{co:5.4}
If the exact solution quadruple satisfies  
\[
\big(\mathbf U, \Theta, \mathcal{L}^{-1} \{\phi \}, \mathcal{L}^{-1} \{\lambda \}\big )   \in  
W_+^4(\mathbf H^1(\Omega_-)) \times W_+^4 (H^1(\Omega_-) )\times W_+^5 ( H^{1/2} (\Gamma) ) \times W_+^4 ( H^{-1/2} (\Gamma) )
\]
then
\[
\mathcal{L}^{-1}\big \{(\mathbf e^h_{\mathbf u}, e^h_{\theta},  e^h_v)\big \} \in \mathcal{C}(\mathbb{R}; \mathbf H^1(\Omega_-) \times H^1(\Omega_-) \times H^1(\mathbb{R}^3 \setminus \Gamma))
\]
is causal and for $t \ge 0$
\begin{align*}
\triple{\mathcal{L}^{-1}\big\{( \mathbf e^h_{\mathbf u}, e^h_{\theta},  e^h_v)\big\}}_1   +  \| \mathcal{L}^{-1} \{ \phi^h - \phi\} \|_{ 1/2, \Gamma} 
\leq \; & C  t^2\max \{1, t^4 \} ~g_h(t), \\
\|\mathcal{L}^{-1}\{\lambda^h - \lambda\} \|_{-1/2, \Gamma} \le\;  & C {t^{3/2}} \max \{1, t^{7/2} \}~g_h(t), 
\end{align*}
where 
\begin{align*}
g_h(t) :=\; & \int_0^t \!\big ( \| \mathcal{P}_4 ( \mathcal{L}^{-1} \{ \dot{\phi}- \Pi_{Y_h} \dot{\phi}\}) (\tau) \|_{1/2, \Gamma} +  \| \mathcal{P}_4  ( \mathcal{L}^{-1}  \{ \lambda -\Pi_{X_h} \lambda  \}) (\tau) \|_{-1/2, \Gamma}  \big )\, d \tau\\
&+ \int_0^t \! \big ( \|\mathcal{P}_4( \ddot{ \mathbf U} - \mathbf P_h\ddot{\mathbf U}) (\tau) \|_{1, \Omega_-}
+  \|\mathcal{P}_4 ( \dot{\mathbf U} - \mathbf P_h \dot{\mathbf U}) (\tau) \|_{1, \Omega_-} \big )\, d \tau \\
&+ \int_0^t \! \big(\|\mathcal{P}_4 ( \mathbf U - \mathbf P_h\mathbf U ) (\tau) \|_{1, \Omega_-} 
+ \| \mathcal{P}_4 ( \Theta - Q_h\Theta ) (\tau) \|_{1, \Omega_-} \big )\, d \tau .
\end{align*}
\end{corollary}
%
% ============================================
\section{Computational Aspects} \label{sec:6}
% ============================================
%
% =======================================
\subsection{Convolution Quadrature}
% =======================================
We present  a very brief description of the procedure used to obtain a full discretization using multistep-based Convolution Quadrature. The process was devised by  Christian Lubich \cite{Lu:1988a, Lub:1988}, and was designed employed originally for treating constitutional boundary integral equations \cite{Lu:1994, LuSc:1992} and in recent times has become a very powerful tool for the discretization of time domain problems. The present description is by no means comprehensive and is provided only for the sake of completeness, the interested reader is refereed to \cite{HaSa:2014}, where very detailed descriptions of the theory and implementation for both multi-step and multi-stage flavors of Convolution Quadrature are given.

Suppose that
\[
dim \,\mathbf{V}_h = N_1, \quad dim \,W_h = N_2, \quad dim\, Y_h = M_1, \quad \hbox{and}\quad dim \,X_h= M_2,
\]
and let 
\[
\{\bm{\mu}_j\}_{j = 1}^{N_1} ,\quad \{ \vartheta_j\}_{j=1}^{N_2},\quad  \{ \varphi_j \}_{j=1}^{M_1}  \; \hbox{, and }\quad  \{ \eta_j\}_{j=1}^{M_2}
\]
be the basis functions of the spaces $\mathbf{V}_h,\; W_h,\;  Y_h $, and $X_h $, respectively.  We choose a time-step $\Delta t > 0, $ and let us consider the uniform grid in time $t_n:= n \Delta t, $ for  $n\ge 0.$  We define $\mathbf A(s) \in \mathbb{C}^{(N_1+N_2+M_1+M_2)\times(N_1+N_2+M_1+M_2)}$ to be the stiffness matrix of equation (\ref{eq:5.2}). $\mathbf{A}(s)$ is a matrix valued function  of $ s \in \mathbb{C}_+$  whose structure is depicted in Figure \ref{fig:matrix}.
\begin{figure}{\centering\scalebox{.8}{
$\mathbf A(s) \sim \left(\begin{array}{c|c|c|c}
  & & & \\
 (\mathrm L(\boldsymbol\mu_i),\boldsymbol\mu_j)_{N_1\times N_1} & (\mathrm L(\boldsymbol\mu_i),\vartheta_j)_{N_1\times N_2} & (\mathrm L(\boldsymbol\mu_i),\varphi_j)_{N_1\times M_1} & \boldsymbol 0_{N_1\times M_2} \\ 
 & & & \\ \hline 
 & & & \\
 (\mathrm L(\boldsymbol\vartheta_i),\boldsymbol\mu_j)_{N_2\times N_1} & (\mathrm L(\vartheta_i),\vartheta_j)_{N_2\times N_2} & \boldsymbol 0_{N_2\times M_1} & \boldsymbol 0_{N_2\times M_2} \\ 
 & & & \\ \hline 
 & & & \\
(\mathrm L(\varphi_i),\boldsymbol\mu_j)_{M_1\times N_1} &  \boldsymbol 0_{M_1\times N_2} & (\mathrm L(\varphi_i),\varphi_j)_{M_1\times M_1} & (\mathrm L(\varphi_i),\eta_j)_{M_1\times M_2} \\
 & & & \\ \hline
 & & & \\
 \boldsymbol 0_{M_2\times N_1} & \boldsymbol 0_{M_2\times N_2} & (\mathrm L(\eta_i),\varphi_j)_{M_2\times M_1} & (\mathrm L(\eta_i),\eta_j)_{M_2\times M_2} \\
 & & &
\end{array}\right)$
}
\caption{The linear system associated to the discretization has the block structure of the schematic. The elastic and thermal unknowns and the acoustic unknowns are weakly coupled, reflecting the physical fact that the systems communicate only through boundary interactions between the acoustic and elastic variables.}\label{fig:matrix}}
\end{figure}
The data are sampled in time and tested to define vectors $\mathbf f_n \in \mathbb{R}^{(N_1+N_2+M_1+M_2)}$:
\begin{alignat*}{6}
f_{n, i} :=\;& (D_1(t_n), \bm{\mu}_i )_{\Omega_-}, \quad && i = 1, \cdots, N_1\\
f_{n, i} :=\;& (D_2(t_n), \vartheta_i)_{\Omega_-}, \quad &&  i =   N_1+1, \cdots, N_1+N_2\\
f_{n, i} :=\;& \langle D_3(t_n), \varphi_i\rangle_{\Gamma}, \quad && i = N_1+N_2 +1, \cdots, N_1+N_2 +M_1\\
f_{n, i} :=\; & \langle D_4(t_n), \eta_i, \rangle_{\Gamma}, \quad && i = N_1+N_2+M_1+1, \cdots, N_1+N_2+M_1+M_2,
\end{alignat*}
where $ D_i(t) =  \mathcal{L}\{ d_i \}, i =1, \cdots, 4 $  in Theorem \ref{thm:4.4}. The CQ discretization of (\ref{eq:5.2}) starts with the Taylor expansion
\begin{equation}\label{eq:6.1}
\mathbf{A}\left( \frac{ \gamma(z)}{\Delta t}\right) = \sum_{n=0}^\infty \mathbf A_n (\Delta t) z^n , \quad \gamma (z) = \frac{\alpha_0 + \cdots +\alpha_k z^{-k}} {\beta_0 + \cdots +\beta_k z^{-k}},
\end{equation}
where  $\gamma(z)$ characterizes the underlying k- multistep method, and is, therefore, usually referred to as the characteristic function of the linear multistep method.

For the discretization of (\ref{eq:5.2}), we seek the sequence of vectors ${\bf b}_n \in \mathbb{R}^{(N_1+N_2+M_1+M_2)}$ given by the recurrence:
\begin{equation}
\label{eq:6.2}
\mathbf A_0 (\Delta t) ~\mathbf b_n  = ~\mathbf {f}_n - \sum_{m=1}^n  \mathbf A_m (\Delta t)~ \mathbf b_{n-m}, \quad n \ge 0.
 \end{equation}
The coefficients $A_n(\Delta t)$ can be computed by means of Cauchy's integral formula
\[
\mathbf A_m(\Delta t) = \frac{1}{m!}\frac{d^{(m)}}{dz^{(m)}}\left(\mathbf A(\gamma(z)/\Delta t)\right)|_{z=0} = \frac{1}{2\pi i}\oint_{C} \zeta^{-m-1}\mathbf A(\gamma(\zeta)/\Delta t)\,d\zeta.
\]
For implementation purposes, the integration contour $C$ is taken to be a circle with radius $R_C$ dependent on the number of terms in the expansion and the specific value of the computer's machine epsilon \cite{HaSa:2014}. This choice of contour allows for fast and accurate computation of the coefficients exploiting the properties of the trapezoidal rule and the Fast Fourier Transform. 

If the solution of (\ref{eq:6.2}) assumes the form 
${\bf{b}}_n = (b_{n,1}, \cdots, b_{n, (N_1+N_2+M_1+M_2)})$, then the Galerkin solutions of (\ref{eq:5.2}) 
at $t_n$ are given by
\begin{alignat*}{6}
\mathbf u^h_n = \;& \sum_{j=1}^{N_1}  b_{n,j} \bm{\mu}_j,  & \qquad
\theta_n^h =\;& \sum_{j=N_1+1}^{N_1+N_2}  b_{n,j} \vartheta_j ,\\
\nonumber & & & &\\
\phi^h_n =\;&  \sum_{j=N_1+N_2+1}^{N_1+N_2+M_1}  b_{n,j} \varphi_j,  &\qquad
\lambda^h_n =\;& \sum_{j=N_1+N_2+M_1+1}^{N_1+N_2+M_1+M_2} b_{n, j} \eta_j.
\end{alignat*}
%
% ===============================================
\subsection{A combined approach for time evolution}
% ===============================================
The linear system arising from the discretization and depicted in Figure \ref{fig:matrix} can be thought of as having the block structure
\[
\left[\begin{array}{cc} \mathrm{\mathbf{FEM}}(s) & s\rho_f(\mathrm N\Gamma)_h^t \\ -s\rho_f(\mathrm N\Gamma)_h &\mathrm{\mathbf{BEM}}(s) \end{array}\right]
\left[\begin{array}{c} \left[\begin{array}{c}\mathbf u^h \\ \theta^h \end{array}\right] \\[2.5ex] \left[\begin{array}{c} \lambda^h \\ \phi^h 
\end{array}\right] \end{array}\right] = 
\left[\begin{array}{c} 
\left[\begin{array}{c}
	\!\!\!-s\rho_f\Gamma_h^t\beta^h \!\!\!\\ \eta^h 
\end{array}\right] \\[2.5ex]
\left[\begin{array}{c}
	 0 \\ \rho_f\alpha^h
\end{array}\right]
 \end{array}\right],
\] 
where the sparse Finite Element block
\begin{align*}
\mathbf{FEM}(s):=\;&\; s^2\left[\begin{array}{cc}(\rho_\Sigma \mathbf u_j,\boldsymbol v_i)_{\Omega_-} & 0 \\ 0 & 0 \end{array}\right] + s\left[\begin{array}{cc} 0 & 0 \\ -(\boldsymbol\eta\mathbf u_j, \nabla v_i)_{\Omega_-} & (\theta_j,v_i)_{\Omega_-}\end{array}\right] \\
 		&+\left[\begin{array}{cc} (\mathbf C\boldsymbol\varepsilon(\mathbf u_j),\boldsymbol\varepsilon(\boldsymbol v_i))_{\Omega_-} & -(\boldsymbol\zeta\theta_j,\boldsymbol\varepsilon(\boldsymbol v_i))_{\Omega_-} \\ 0 & (\boldsymbol\kappa\nabla\theta_j,\nabla v_i)_{\Omega_-}\end{array}\right]
\end{align*}
contains mass and stiffness matrices as well as first order terms related to the elastic and thermal unknowns. The boundary element block $\mathbf{BEM}(s)$ contains the Galerkin discretization of the operators of the acoustic Calder\'on calculus and the coupling trace matrix $(\mathrm N\Gamma)_h$ is the discretization of the bilinear form arising from the duality pairing $\langle\mathbf u^h\cdot\boldsymbol\nu,\chi^h\rangle_\Gamma$.  

\begin{figure}[b]{\centering
\includegraphics[scale=0.5]{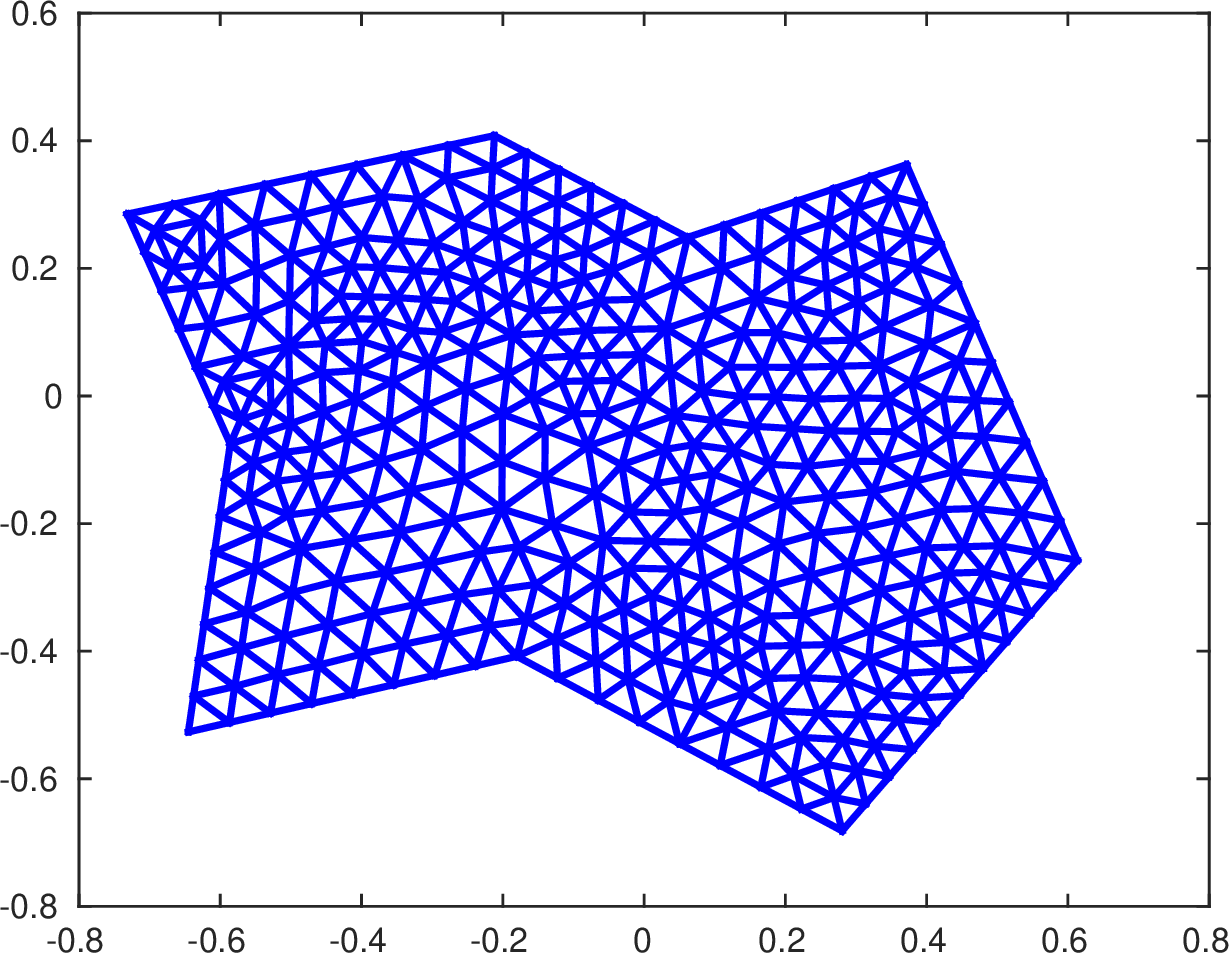}
\caption[Thermoelastic computational domain]{{\footnotesize Interior geometry used in the numerical experiments for both frequency-domain and time-domain studies. The domain was generated and meshed using Matlab's pdetool and refined uniformly using pde tool's refinement capabilities.}}\label{fig:c5:2}
}
\end{figure}
Even if a CQ approach can be applied to the entire system using the expansion \eqref{eq:6.1} on  the global matrix and solving the resulting linear system \eqref{eq:5.18},   it is a common practice to decouple the computations of the finite element and boundary element unknowns via a Schur complement strategy. The decoupled  Boundary Element unknowns are then evolved in time using CQ (via \eqref{eq:6.1} and \eqref{eq:5.18}), while the same underlying multistep scheme is used for the Finite Element unknowns. This process was first described in \cite{BaSa:2009} and is explained in detail  in \cite{HaSa:2015}, where it is used in the context of  for purely acoustic waves.
%
% ===========================================
\subsection{Numerical Experiments}
% ===========================================
%
In order to test numerically the formulations of the previous sections, computational convergence studies were performed in both frequency and time domains for 2D test problems. Moreover, we explore numerically the case when the Lam\'e parameters or the thermal diffusivity and expansion are non constant tensors. We emphasize that the goal of the computations presented here is mainly to provide a proof of concept of the suggested discretization and to highlight the fact that the discretization can be readily implemented with only minor additions to existing code. The previous analysis remains valid in 3D and the implementation of the  discretization in that case can be done following completely analogous steps. 

% =============================
%\textbf{The computational domain}
% =============================
\textbf{The computational domain.} For the convergence studies, the interior domain $\Omega_-$ where the thermoelastic equations were imposed was the polygon depicted in Figure \ref{fig:c5:2}. The domain was generated and meshed using Matlab's pdetool. All the mesh refinements were done using the refinement capabilities of the pde toolbox.

% ============================
%\texbf{Approximation errors}
% ============================
\textbf{Approximation errors.} As a measurement of the accuracy of the approximations, the difference between the manufactured solution and the approximate Finite Element solutions was measured in the $L^2(\Omega_-)$ and $H^1(\Omega_-)$ norms for the elastic and thermal variables $\mathbf u^h$ and $\theta^h$. For the acoustic unknown $v^h$, the approximate solution was sampled in 25 randomly placed points in $\Omega_+$ and the maximum absolute difference with the exact solution was taken as a measure of the error. In the time domain experiments this measurements were done for the a final time $t=1.5$. 

% ===============================
% \textbf{Physical parameters.}
% ===============================

\textbf{Physical parameters.} The following values of the physical parameters are functions only of space and were used equally for both series of experiments. They are chosen for validation and expository purposes only and do not correspond with any relevant physical material. For the entries of the tensors we make use of the symmetries and of Voigt's notation \cite{Gu:1972} to shorten the subscripts.  
\begin{enumerate}
\item Density  of the elastic solid and Lam\'e parameters
\begin{equation}\label{eq:c5:num:1}
\rho_\Sigma = 5 + \sin{(x)}\sin{(y)},\qquad \lambda = 2, \qquad \mu =3.
\end{equation}
\item Thermal expansion tensor $\boldsymbol\zeta$:
\begin{equation}\label{eq:c5:num:3} 
 \boldsymbol\zeta_1\leftrightarrow\, \boldsymbol\zeta_{11} \!=\! \sin{(x)}+\cos{(y)}, \quad \boldsymbol\zeta_2 \leftrightarrow \,\boldsymbol\zeta_{22} \!=\! -\sin{(y)}, \quad \boldsymbol\zeta_3 \leftrightarrow\, \boldsymbol\zeta_{12} \!=\! \boldsymbol\zeta_{21} \!=\! \cos{(x)}.
\end{equation} 
\item Thermal diffusivity tensor $\boldsymbol\kappa$:
\begin{equation}\label{eq:c5:num:4} 
 \boldsymbol\kappa_1\leftrightarrow\, \boldsymbol\kappa_{11}  = 10 + x^2, \quad \boldsymbol\kappa_2 \leftrightarrow \,\boldsymbol\kappa_{22}  = 10+ y, \quad \boldsymbol\kappa_3 \leftrightarrow\, \boldsymbol\kappa_{12}= \boldsymbol\kappa_{21} = 0.
\end{equation}
\item The components of the tensor $\boldsymbol\eta$ were chose to be:
\begin{equation}\label{eq:c5:num:5} 
 \boldsymbol\eta_1\leftrightarrow\, \boldsymbol\eta_{11}  = 1, \quad \boldsymbol\eta_2 \leftrightarrow \,\boldsymbol\eta_{22}  = x + y, \quad \boldsymbol\eta_3 \leftrightarrow\, \boldsymbol\eta_{12}= \boldsymbol\eta_{21} = 5+x+y.
\end{equation}
\end{enumerate}
% =======================================================
% \paragraph*{Convergence studies in the frequency domain}
% =======================================================
\textbf{Convergence studies in the frequency domain.} We first verify the results in the frequency domain. We proceed by the method of manufactured solutions using the functions
\begin{subequations}
\begin{alignat*}{6}
\mathbf u :=&\, (x^3+xy+y^3,\sin{(x)}\cos{(y)}), \qquad&& \theta := \sin^2{(\pi x)}\sin^2{(y)}, \\
v :=&\, \tfrac{i}{4}H_0^{(1)}(isr), \quad&& r = \sqrt{x^2+y^2},
\end{alignat*}
\end{subequations}
together with the parameters defined in \eqref{eq:c5:num:1} through \eqref{eq:c5:num:5}. Right-hand side load vectors and boundary conditions were constructed accordingly. 

For the numerical experiments, Lagrangian $\mathcal P_k$ finite elements were used for the elastic and thermal unknowns, while Galerkin $\mathcal P_k/\mathcal P_{k-1}$ continuous/discontinuous Boundary Elements were used for the acoustic potential $v$. Convergence studies for spatial refinements with a fixed polynomial degree (\textit{h-convergence}) and increasing degree of polynomial approximation with a fixed mesh size (\textit{p-convergence}) were performed for $s=2.8i$. The results of the mesh-refinement experiments are shown in Tables \ref{tab:c5:1}, \ref{tab:c5:2}, and \ref{tab:c5:3}. The table \ref{tab:pconvergence} contains the results for a fixed mesh with increasing polynomial degree for the basis functions. The convergence plots for all the simulations are displayed in Figure \ref{fig:c5:3}.

% 
% ===============================================
% \textbf{Convergence studies in the time domain}
% ===============================================
\textbf{Convergence studies in the time domain.} In a way analogous to the previous section, the numerical experiments were carried out using the physical parameters and coefficients given in \eqref{eq:c5:num:1} through \eqref{eq:c5:num:5} and with manufactured solutions using the functions
\begin{subequations}
\begin{alignat*}{6}
\mathbf u :=&\, \mathrm T(t) (x^3+xy+y^3,\sin{(x)}\cos{(y)}), \qquad&& \theta :=  \mathrm T(t) \sin^2{(\pi x)}\sin^2{(y)}, \\
v :=&\,\mathcal{L}^{-1}\left\{iH^{(1)}_0(i s r)\,\mathcal{L}\{\mathcal{H}(t)\sin(3t)\} \right\}, \qquad&& r := \sqrt{x^2+y^2},
\end{alignat*}
where $\mathcal L\{\cdot\}$ is the Laplace transform, the time factor $\mathrm T(t)$ is given by
\begin{equation}
\mathrm T := \mathcal H(t)(t^2+2t),
\end{equation}
\end{subequations}
and $\mathcal H(t)$ is the $\mathcal C^5$ approximation to Heaviside's step function
\begin{equation*}\label{eq:c3:HH}
\mathcal H(t)\!:= \!t^5(1-5(t-1)\!+\!15(t-1)^2\!-\!35(t-1)^3\!+\!70(t-1)^4\!-\!126(t-1)^5)\chi_{[0,1]}(t)\!+\!\chi_{[1,\infty)}(t).
\end{equation*}
Two kinds of experiments were carried out using the same geometry as in the frequency domain. On the one hand, for a spatial discretization with fixed polynomial degree, successive dyadic refinements in both mesh size $h$ and time step $\Delta t$  were carried out (\textit{h-refinement}). The experiment was repeated for polynomial degrees $k=1,2$, and 3 starting with with a spatial mesh with parameter $h=1 \times 10^{-1}$ and time step $\Delta t = 3.75 \times 10^{-2}$. 

The second experiment corresponds to \textit{p-refinements} in space and consisted in using a fixed spatial mesh, starting with a polynomial discretization of degree $k=1$ in space and a time step $\Delta t = 3.75 \times 10^{-2}$. With every successive dyadic refinement of $\Delta t$, the degree of the polynomial interpolant was increased by one.  The initial mesh size of $h=5.016 \times 10^{-2}$ corresponds to the second level of refinement used for the \textit{h-refinement} experiments. This space-time refinement strategies highlight the global order of convergence of the method, which is expected to be asymptotically limited by the order of the multi-step scheme used for time discretization.

Both strategies were tried for BDF2 and Trapezoidal Rule based time discretizations. The results for the studies using BDF2 are shown in Tables \ref{tab:BDF2k1}, \ref{tab:BDF2k2}, \ref{tab:BDF2k3} (h-refinement), and \ref{tab:BDF2pref} (p-refinement). These results are summarized in the convergence plot of Figure \ref{fig:BDF2convergence}. Similarly, the results for the experiments using Trapezoidal Rule time stepping are shown in Tables \ref{tab:TRk1}, \ref{tab:TRk2}, \ref{tab:TRk3} (h-refinement), and \ref{tab:TRpref} (p-refinement), all condensed on the convergence plots shown in Figure \ref{fig:TRconvergence}.

Depending on the refinement strategy, the number of degrees of freedom required to approximate to the system increases quickly, especially for h-refinements with higher order polynomial basis. Table \ref{tab:ndof} shows the number of unknowns associated to a single scalar FEM function represented in the grid shown in Figure \ref{fig:c5:2}. The increase in computational requirements imposed by h-refinement makes some asymptotic properties of the scheme difficult to observe following this strategy. 

In particular, the smoothing properties of the parabolic part of the system introduce superconvergent behavior on the thermal unknowns during the pre-asymptotic regime. As can be seen in the p-refinement experiments (c.f. Figures \ref{fig:BDF2convergence} and  \ref{fig:TRconvergence}, bottom right) the convergence stabilizes to the predicted rate for relatively small time step, after five refinement levels. The number of spatial degrees of freedom required to achieve such a discretization level by h-refinements causes the true convergence rate to be observable only using a p-refinement strategy.
\begin{table}[tbp]
\centering \scalebox{0.9}{\!
\begin{tabular}{|c|cccccc|}
\hline
\textbf{Growth in FEM DOF} & \multicolumn{6}{c|}{Refinement Level}  \\ \hline
Refinement Strategy    & \multicolumn{1}{c|}{1} & \multicolumn{1}{c|}{2} & \multicolumn{1}{c|}{3} & \multicolumn{1}{c|}{4} & \multicolumn{1}{c|}{5} & \multicolumn{1}{c|}{6}     \\ \hline
h-refinement, $k=1$ & 108  & 394   & 1503   & 5869   & 23193  & 92209  \\ \cline{1-7}
h-refinement, $k=2$ & 394  & 1503  & 5869   & 23193  & 92209  & 161225 \\ \cline{1-7}
h-refinement, $k=3$ & 859  & 3328  & 13099  & 51973  & 207049 & 413537 \\ \cline{1-7}
p-refinement        & 108  & 394   & 859    & 1503   & 2326   & 3328   \\ \hline
\end{tabular}
}
\caption{{\footnotesize Number of Finite Element degrees of freedom required to represent a single scalar function defined on the mesh depicted in Figure \ref{fig:c5:2}. The number increases depending on the chosen refinement strategy, with the slowest growth rate being the one associated to p-refinements. }}
\label{tab:ndof}
\end{table}

\textbf{Examples.} We now present a couple of illustrative examples in 2D. The first example shows the interaction between the plane wave
\[
v^{inc}= 3\chi_{[0,0.3]}(88\tau)\sin{(88\tau)},\quad \tau:= t-\mathbf r\cdot\mathbf d,\quad \mathbf r:=(x,y),\quad \mathbf d := (1,5)/\sqrt{26},
\]
and a pentagonal scatterer with mass density given by
\[
\rho_\Sigma = 15 + 40e^{-49\, r^2}, \qquad r := \sqrt{x^2+y^2}.
\]
The values of the elastic parameters, thermic diffusivity $\boldsymbol\kappa$, thermoelastic expansion tensors $\boldsymbol\zeta$  and $\boldsymbol\eta$ were the same as those used for the convergence experiments in the previous paragraphs and given in equations \eqref{eq:c5:num:1}-\eqref{eq:c5:num:5}. The simulation used $\mathcal P_2$ Lagrangian finite elements on a grid with mesh parameter $h=7 \times 10^{-3}$ and 36096 elements. The inherited boundary element grid had 496 panels and a grid parameter of $h=9.1 \times 10^{-3}$, and $\mathcal P_2/\mathcal P_1$ continuous/discontinuous Galerkin boundary elements were used. Trapezoidal rule-based discretization was applied in time with a time step $\Delta t=1\times 10^{-2}$. Some snapshots of the simulation are shown in Figures \ref{fig:c5:5}-\ref{fig:c5:7}.

The second example is a trapping geometry with density $\rho_\Sigma = 20 + |x|+|y|$ and physical parameters given by  \eqref{eq:c5:num:1}-\eqref{eq:c5:num:5}. For this example $\mathcal P_5$ Lagrangian elements were used on a grid with 2992 elements and mesh parameter $h=1.72 \times 10^{-2}$, the acoustic equations were discretized with  $\mathcal P_5/\mathcal P_4$ continuous/discontinuous Galerkin boundary elements on a mesh with 236 panels  and mesh parameter $h=2.5 \times 10^{-2}$. For time discretization trapezoidal rule-based CQ was used with a time step size of $\Delta t=2 \times \times 10^{-3}$. Figures \ref{fig:c5:8}-\ref{fig:c5:10} show snapshots of the acoustic, elastic and temperature fields.

% =========================
\paragraph{Acknowledgements.\\}
% =========================
The authors would like to thank the referees for their detailed comments and suggestions, which greatly improved the quality of this communication.

\textit{In respectful memory of Prof. Richard Weinacht.}

\newpage
\begin{table}\centering
\scalebox{0.785}{\!
\begin{tabular}{ccccccccccccc}
\hline
\multicolumn{3}{|c|}{$k=1$} & \multicolumn{4}{|c|}{$L^2(\Omega_-)$} & \multicolumn{4}{|c|}{$H^1(\Omega_-)$} \\
\hline
\multicolumn{1}{|c|}{$h$} & \multicolumn{1}{c|}{$E^{v}_{h}$} & \multicolumn{1}{c|}{e.c.r.} & \multicolumn{1}{c|}{$E^{\mathbf u}_{h}$} & \multicolumn{1}{c|}{e.c.r.} & \multicolumn{1}{c|}{$E^{\theta}_{h,\kappa}$} & \multicolumn{1}{c|}{e.c.r.} & \multicolumn{1}{c|}{$E^{\mathbf u}_{h}$} & \multicolumn{1}{c|}{e.c.r.} & \multicolumn{1}{c|}{$E^{\theta}_{h}$} & \multicolumn{1}{c|}{e.c.r.} \\ \hline
 1 E-1  & 1.787 E-2  & ---   & 3.999 E-2  & ---   & 2.015 E-3  & ---   & 2.011 E-1 & ---   & 7.430 E-2     & ---    \\ \hline
5.016 E-2 & 7.292 E-3  & 1.293 & 1.675 E-2  & 1.255 & 6.397 E-4  & 1.656  & 8.733 E-2 & 1.203 & 3.746 E-2 & 0.988  \\ \hline
2.508 E-2 & 2.272 E-3  & 1.683 & 5.344 E-3  & 1.648 & 1.837 E-4 & 1.799 & 3.297 E-2 & 1.405 & 1.876 E-2 & 0.976  \\ \hline 
1.254 E-2 & 6.099 E-4  & 1.897 & 1.447 E-3  & 1.885 & 4.824 E-5 & 1.929 & 1.314 E-2 & 1.327 & 9.383 E-3 & 0.996  \\ \hline
6.27 E-3 & 1.556 E-4  & 1.971 & 3.703 E-4  & 1.966 & 1.223 E-4 & 1.980 & 5.961 E-3 & 1.141 & 4.692 E-3 & 1.000  \\ \hline  
\end{tabular}
}
\caption{{\footnotesize The experiments were run using $\mathcal P_k$ Lagrangian finite elements and $\mathcal P_k/\mathcal P_{k-1}$ boundary elements. This table shows the relative errors and estimated convergence rates in the frequency domain for $k=1$. The maximum length of the panels used to discretize the boundary is denoted by $h$.}}\label{tab:c5:1}
\end{table}
\begin{table}\centering
\scalebox{0.785}{\!
\begin{tabular}{ccccccccccccc}
\hline
\multicolumn{3}{|c|}{$k=2$} & \multicolumn{4}{|c|}{$L^2(\Omega_-)$} & \multicolumn{4}{|c|}{$H^1(\Omega_-)$} \\
\hline
\multicolumn{1}{|c|}{$h$} & \multicolumn{1}{c|}{$E^{v}_{h}$} & \multicolumn{1}{c|}{e.c.r.} & \multicolumn{1}{c|}{$E^{\mathbf u}_{h}$} & \multicolumn{1}{c|}{e.c.r.} & \multicolumn{1}{c|}{$E^{\theta}_{h,\kappa}$} & \multicolumn{1}{c|}{e.c.r.} & \multicolumn{1}{c|}{$E^{\mathbf u}_{h}$} & \multicolumn{1}{c|}{e.c.r.} & \multicolumn{1}{c|}{$E^{\theta}_{h}$} & \multicolumn{1}{c|}{e.c.r.} \\ \hline
 1 E-1  & 7.926 E-5  & ---   & 1.284 E-4  & ---   & 9.742 E-5  & ---   & 3.514 E-3 & ---   & 6.446 E-3  & ---    \\ \hline
5.016 E-2 & 6.676 E-6  & 3.570 & 1.181 E-5  & 3.442 & 1.214 E-5  & 3.004  & 8.708 E-4 & 2.013 & 1.630 E-3 & 1.983  \\ \hline
2.508 E-2 & 5.590 E-7  & 3.578 & 1.207 E-6  & 3.290 & 1.517 E-6 & 3.000 & 2.172 E-4 & 2.003 & 4.093 E-4 & 1.993  \\ \hline 
1.254 E-2 & 4.630 E-8  & 3.594 & 1.331 E-7  & 3.181 & 1.897 E-7 & 2.999 & 5.426 E-5 & 2.001 & 5.426 E-5 & 1.997  \\ \hline
6.27 E-3 & 3.793 E-9  & 3.609 & 1.550 E-8  & 3.103 & 2.373 E-8 & 2.999 & 1.356 E-8 & 2.001 & 2.566 E-5 & 1.999  \\ \hline  
\end{tabular}
}
\caption{{\footnotesize The experiments were run using $\mathcal P_k$ Lagrangian finite elements and $\mathcal P_k/\mathcal P_{k-1}$ boundary elements. This table shows the relative errors and estimated convergence rates in the frequency domain for $k=2$. The maximum length of the panels used to discretize the boundary is denoted by $h$.}}\label{tab:c5:2}
\end{table}
\begin{table}\centering
\scalebox{0.76}{\!
\begin{tabular}{ccccccccccccc}
\hline
\multicolumn{3}{|c|}{$k=3$} & \multicolumn{4}{|c|}{$L^2(\Omega_-)$} & \multicolumn{4}{|c|}{$H^1(\Omega_-)$} \\
\hline
\multicolumn{1}{|c|}{$h$} & \multicolumn{1}{c|}{$E^{v}_{h}$} & \multicolumn{1}{c|}{e.c.r.} & \multicolumn{1}{c|}{$E^{\mathbf u}_{h}$} & \multicolumn{1}{c|}{e.c.r.} & \multicolumn{1}{c|}{$E^{\theta}_{h,\kappa}$} & \multicolumn{1}{c|}{e.c.r.} & \multicolumn{1}{c|}{$E^{\mathbf u}_{h}$} & \multicolumn{1}{c|}{e.c.r.} & \multicolumn{1}{c|}{$E^{\theta}_{h}$} & \multicolumn{1}{c|}{e.c.r.} \\ \hline
 1 E-1  & 6.847 E-7  & ---   & 1.726 E-6  & ---   & 4.564 E-6  & ---   & 9.540 E-6 & ---   & 4.018 E-4  & ---    \\ \hline
5.016 E-2 & 3.869 E-8 & 4.145 & 9.804 E-8  & 4.138 & 2.886 E-7 & 3.983  & 7.701 E-7 & 3.631 & 5.044 E-5 & 2.994  \\ \hline
2.508 E-2 & 2.279 E-9 & 4.085 & 5.794 E-9  & 4.081 & 1.810 E-8 & 3.995 & 7.600 E-8 & 3.341 & 6.312 E-6 & 2.998  \\ \hline 
1.254 E-2 & 1.375 E-10 & 4.051 & 3.502 E-10  & 4.048 & 1.132 E-9 & 3.999 & 8.504 E-9 & 3.160 & 7.892 E-7 & 3.000   \\ \hline
6.27 E-3  & 8.468 E-12  & 4.021 & 2.141 E-11  & 4.032 & 7.076 E-11 & 4.000 & 1.011 E-9 & 3.072 & 9.866 E-8 & 3.000  \\ \hline  
\end{tabular}
}
\caption{{\footnotesize The experiments were run using $\mathcal P_k$ Lagrangian finite elements and $\mathcal P_k/\mathcal P_{k-1}$ boundary elements. This table shows the relative errors and estimated convergence rates in the frequency domain for $k=3$. The maximum length of the panels used to discretize the boundary is denoted by $h$.}}\label{tab:c5:3}
\end{table}
\begin{table}\centering
\scalebox{.78}{\!
\begin{tabular}{ccccccccccccc}
\hline
\multicolumn{3}{|c|}{$h=0.1$} & \multicolumn{4}{|c|}{$L^2(\Omega_-)$} & \multicolumn{4}{|c|}{$H^1(\Omega_-)$} \\
\hline
\multicolumn{1}{|c|}{Ndof (degree)} & \multicolumn{1}{c|}{$E^{v}_{h}$} & \multicolumn{1}{c|}{e.c.r.} & \multicolumn{1}{c|}{$E^{\mathbf u}_{h}$} & \multicolumn{1}{c|}{e.c.r.} & \multicolumn{1}{c|}{$E^{\theta}_{h,\kappa}$} & \multicolumn{1}{c|}{e.c.r.} & \multicolumn{1}{c|}{$E^{\mathbf u}_{h}$} & \multicolumn{1}{c|}{e.c.r.} & \multicolumn{1}{c|}{$E^{\theta}_{h}$} & \multicolumn{1}{c|}{e.c.r.} \\ \hline
108 (1) & 1.787 E-2 &--- & 3.999 E-2 &--- & 2.015 E-3 &--- & 2.011 E-1 &--- & 7.430 E-2 & ---\\ \hline
394  (2)& 7.926 E-5 &--- & 1.284 E-4 &--- & 9.742 E-5 &--- & 3.514 E-3 &--- & 6.446 E-3 & ---\\ \hline
859  (3)& 6.848 E-7 &--- & 1.726 E-6 &--- & 4.564 E-6 &--- & 9.540 E-6 &--- & 4.018 E-4 & ---\\ \hline
1503  (4)& 5.185 E-9 &--- & 1.154 E-8 &--- & 1.503 E-7 &--- & 2.861 E-7 &--- & 2.042 E-5 & ---\\ \hline
2326  (5)& 1.241 E-10 &--- & 3.533 E-10 &--- & 5.814 E-9 &--- & 9.008 E-9 &--- & 8.133 E-7 & ---\\ \hline
\end{tabular}
}
\caption{{\footnotesize Frequency-domain p-convergence studies. The experiments were run on a fixed mesh with parameter $h=0.1$ using $\mathcal P_k$ Lagrangian finite elements and $\mathcal P_k/\mathcal P_{k-1}$ boundary elements with $k=1,\ldots,5$. The relative error is shown as a function of the number of degrees of freedom (Ndof).}}\label{tab:pconvergence}
\end{table}
\begin{figure}{\centering
\includegraphics[width = .49\linewidth]{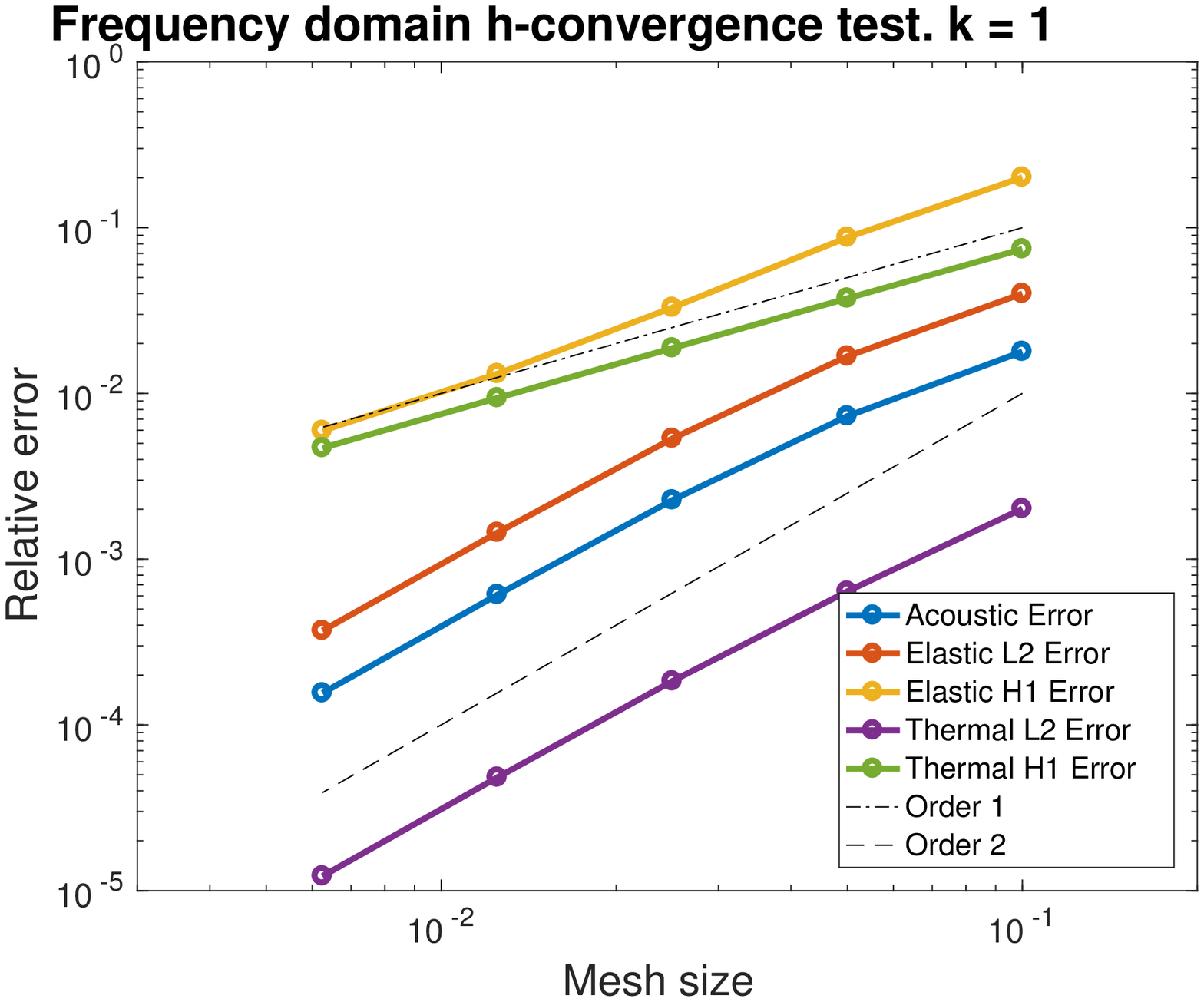}\;
\includegraphics[width = .49\linewidth]{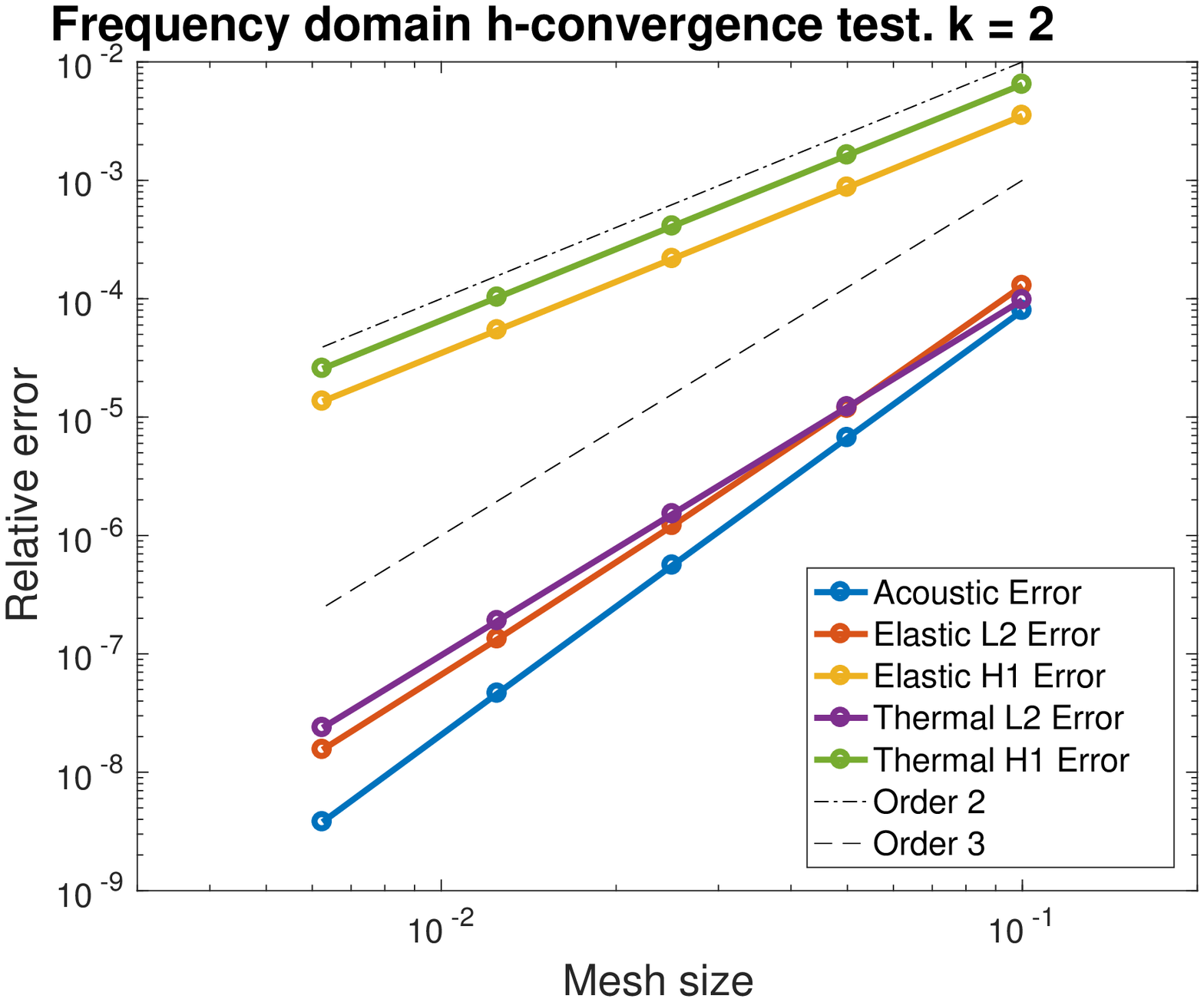}\\
\includegraphics[width = .49\linewidth]{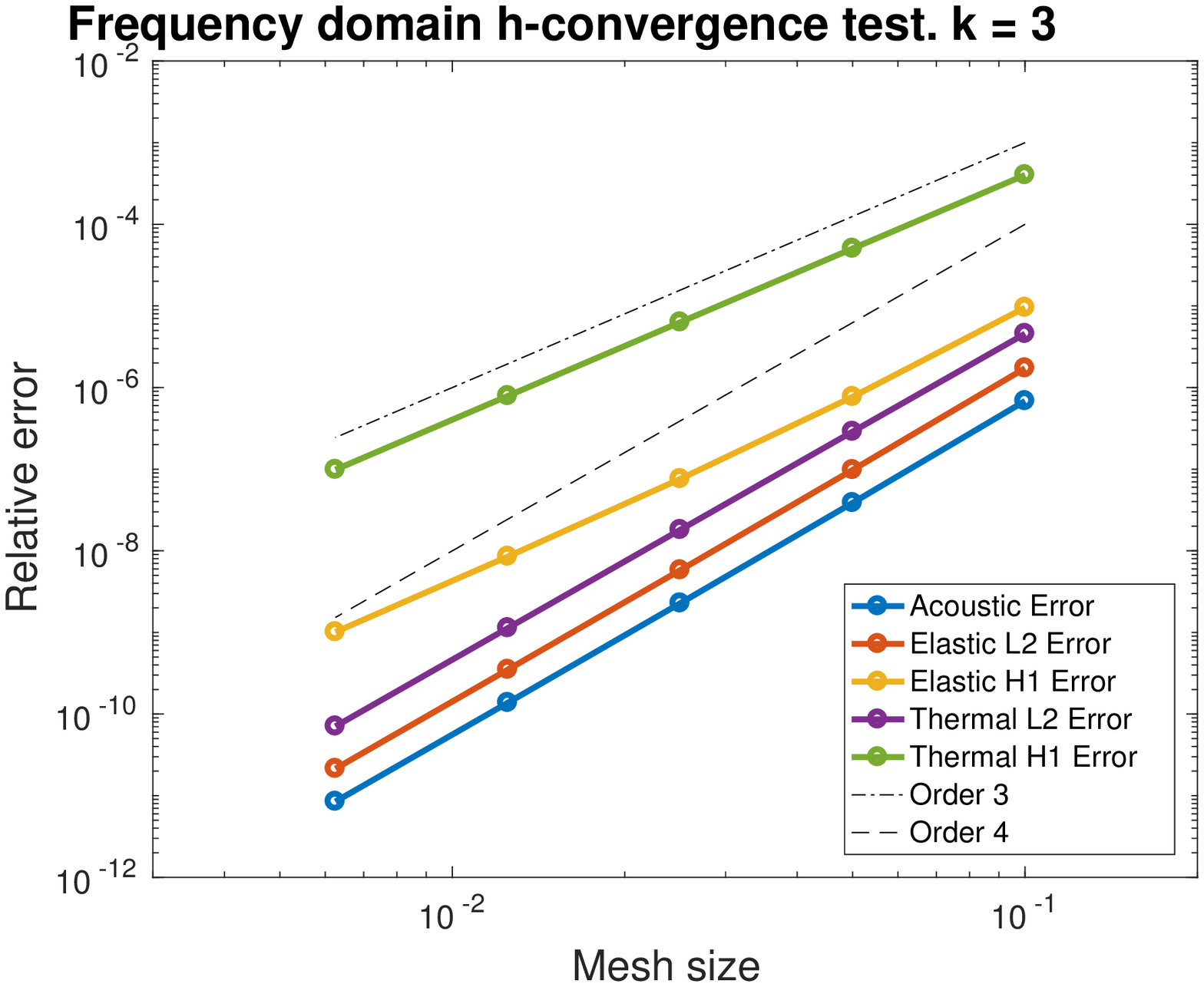}\;
\includegraphics[width = .49\linewidth]{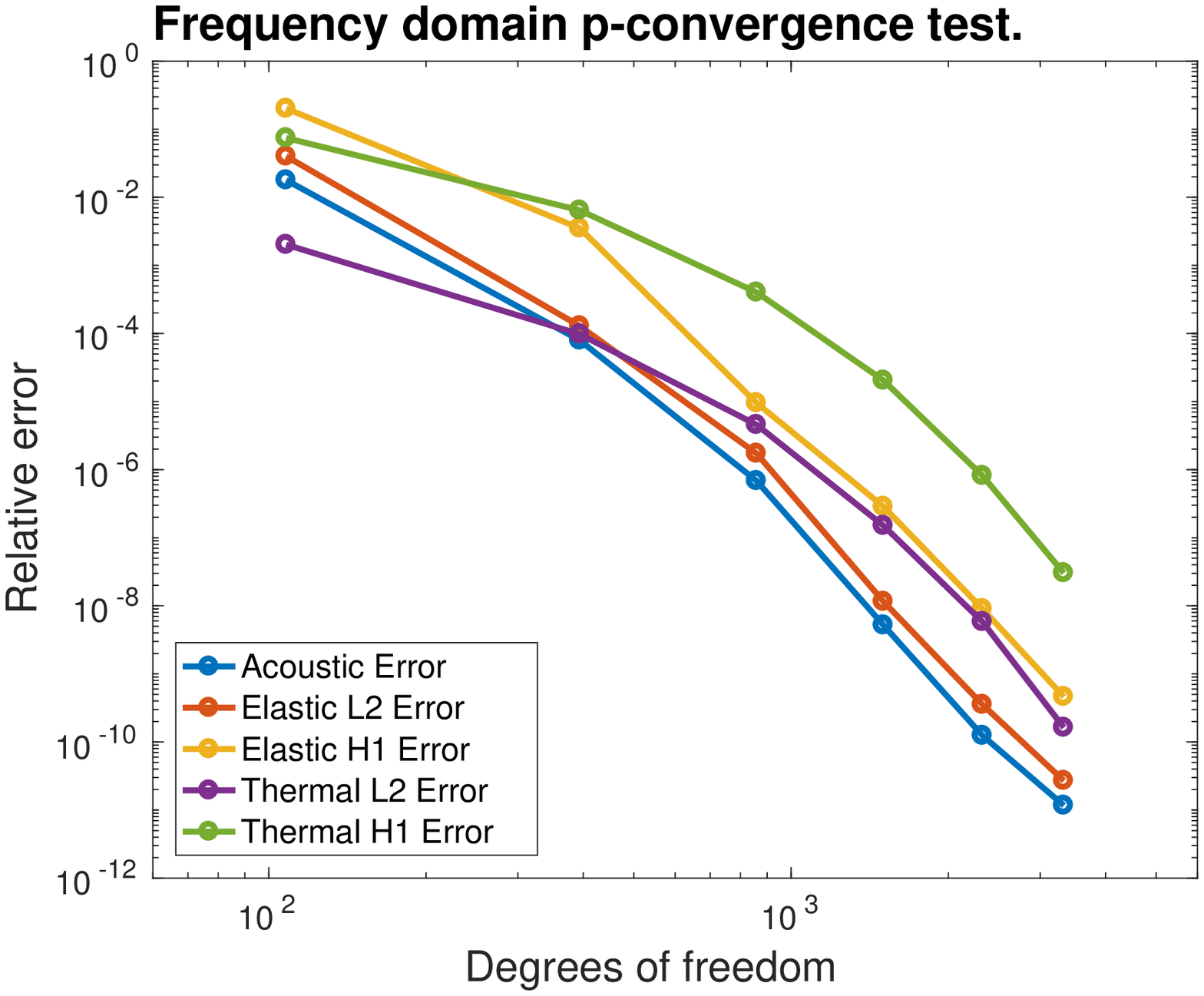}
\caption{{\footnotesize Convergence studies in the frequency domain. Top row and bottom left: Successive mesh refinements were carried out for basis functions with polynomial degrees $k=1,2,$ and $3$. Bottom right: Approximation errors as a function of the degrees of freedom for basis functions of increasing order over a fixed mesh with parameter $h=0.1$. For the color code we refer the reader to the electronic version of the manuscript.}}\label{fig:c5:3}
}
\end{figure}
\clearpage
\begin{table}\centering
\scalebox{0.74}{\!
\begin{tabular}{ccccccccccccc}
\hline
\multicolumn{3}{|c|}{BDF2. k=1} & \multicolumn{4}{|c|}{$L^2(\Omega_-)$} & \multicolumn{4}{|c|}{$H^1(\Omega_-)$} \\
\hline
\multicolumn{1}{|c|}{$\Delta t$   ($h$)} & \multicolumn{1}{c|}{$E^{v}_{h}$} & \multicolumn{1}{c|}{e.c.r.} & \multicolumn{1}{c|}{$E^{\mathbf u}_{h}$} & \multicolumn{1}{c|}{e.c.r.}   & \multicolumn{1}{c|}{$E^{\theta}_{h,\kappa}$} & \multicolumn{1}{c|}{e.c.r.} & \multicolumn{1}{c|}{$E^{\mathbf u}_{h}$} & \multicolumn{1}{c|}{e.c.r.}   & \multicolumn{1}{c|}{$E^{\theta}_{h}$} & \multicolumn{1}{c|}{e.c.r.} \\ \hline
3.75  E-2 (1.000 E-1) & 4.486 E-3 & --- & 1.217 E-2 & --- & 5.381 E-3 & --- & 2.975 E-1 & --- & 2.221 E-1 & --- \\ \hline  
1.875 E-2 (5.016 E-2)& 1.130 E-3 & 1.989 & 3.337 E-3 & 1.867 & 1.358 E-3 & 1.987 & 1.434 E-1 & 1.053 & 1.121 E-1 & 0.986 \\ \hline 
9.375 E-3 (2.508 E-2)& 1.130 E-3 & 1.989 & 3.337 E-3 & 1.867 & 1.358 E-3 & 1.987 & 1.434 E-1 & 1.053 & 1.121 E-1 & 0.986 \\ \hline 
4.697 E-3 (1.254 E-2) & 7.293 E-5 & 1.990 & 2.214 E-4 & 1.975 & 8.528 E-5 & 1.998 & 3.458 E-2 & 1.017 & 2.814 E-2 & 0.999 \\ \hline 
2.344 E-3 (6.270 E-3)& 1.827 E-5 & 1.997 & 5.569 E-5 & 1.991 & 2.133 E-5 & 1.999 & 1.721 E-2 & 1.006 & 1.407 E-2 & 1.000 \\ \hline
\end{tabular}
}
\caption{{\footnotesize Time domain convergence results for BDF2-based CQ with combined $h$ and $\Delta t$ refinements. In every successive refinement level the size of the time step and the mesh parameter were halved. The table shows the relative errors and estimated convergence rates measured for a final time $t=1.5$ and polynomial degree $k=1$.}}\label{tab:BDF2k1}
\end{table}
\begin{table}\centering
\scalebox{0.74}{\!
\begin{tabular}{ccccccccccccc}
\hline
\multicolumn{3}{|c|}{BDF2. k=2} & \multicolumn{4}{|c|}{$L^2(\Omega_-)$} & \multicolumn{4}{|c|}{$H^1(\Omega_-)$} \\
\hline
\multicolumn{1}{|c|}{$\Delta t$   ($h$)} & \multicolumn{1}{c|}{$E^{v}_{h}$} & \multicolumn{1}{c|}{e.c.r.} & \multicolumn{1}{c|}{$E^{\mathbf u}_{h}$} & \multicolumn{1}{c|}{e.c.r.}   & \multicolumn{1}{c|}{$E^{\theta}_{h,\kappa}$} & \multicolumn{1}{c|}{e.c.r.} & \multicolumn{1}{c|}{$E^{\mathbf u}_{h}$} & \multicolumn{1}{c|}{e.c.r.}   & \multicolumn{1}{c|}{$E^{\theta}_{h}$} & \multicolumn{1}{c|}{e.c.r.} \\ \hline
3.75  E-2 (1.000 E-1) & 3.624 E-3 & --- & 7.798 E-4 & --- & 2.917 E-4 & --- & 1.297 E-2 & --- & 1.934 E-2 & --- \\ \hline  
1.875 E-2 (5.016 E-2) & 6.471 E-4 & 2.486 & 2.004 E-4 & 1.960 & 3.640 E-5 & 3.002 & 3.370 E-3 & 1.944 & 4.890 E-3 & 1.983 \\ \hline
9.375 E-3 (2.508 E-2) & 1.571 E-4 & 2.043 & 5.062 E-5 & 1.985 & 4.550 E-6 & 3.000 & 8.524 E-4 & 1.983 & 1.228 E-3 & 1.993 \\ \hline  
4.697 E-3 (1.254 E-2)& 3.891 E-5 & 2.013 & 1.271 E-5 & 1.994 & 5.692 E-7 & 2.999 & 2.137 E-4 & 1.996 & 3.076 E-4 & 1.997 \\ \hline
2.344 E-3 (6.270 E-3) & 9.701 E-6 & 2.004 & 3.182 E-6 & 1.998 & 7.119 E-8 & 2.999 & 5.347 E-5 & 1.999 & 7.697 E-5 & 1.999 \\ \hline
\end{tabular}
}
\caption{{\footnotesize Time domain convergence results for BDF2-based CQ with combined $h$ and $\Delta t$ refinements. In every successive refinement level the size of the time step and the mesh parameter were halved. The table shows the relative errors and estimated convergence rates measured for a final time $t=1.5$ and polynomial degree $k=2$.}}\label{tab:BDF2k2}
\end{table}
\begin{table}\centering
\scalebox{0.74}{\!
\begin{tabular}{ccccccccccccc}
\hline
\multicolumn{3}{|c|}{BDF2. k=3} & \multicolumn{4}{|c|}{$L^2(\Omega_-)$} & \multicolumn{4}{|c|}{$H^1(\Omega_-)$} \\
\hline
\multicolumn{1}{|c|}{$\Delta t$   ($h$)} & \multicolumn{1}{c|}{$E^{v}_{h}$} & \multicolumn{1}{c|}{e.c.r.} & \multicolumn{1}{c|}{$E^{\mathbf u}_{h}$} & \multicolumn{1}{c|}{e.c.r.}   & \multicolumn{1}{c|}{$E^{\theta}_{h,\kappa}$} & \multicolumn{1}{c|}{e.c.r.} & \multicolumn{1}{c|}{$E^{\mathbf u}_{h}$} & \multicolumn{1}{c|}{e.c.r.}   & \multicolumn{1}{c|}{$E^{\theta}_{h}$} & \multicolumn{1}{c|}{e.c.r.} \\ \hline
3.75  E-2 (1.000 E-1)& 3.631 E-3 & --- & 7.616 E-4 & --- & 1.368 E-5 & --- & 7.737 E-3 & --- & 1.205 E-3 & --- \\ \hline
1.875 E-2 (5.016 E-2)& 6.480 E-4 & 2.486 & 1.995 E-4 & 1.933 & 8.649 E-7 & 3.983 & 2.140 E-3 & 1.854 & 1.513 E-4 & 2.994 \\ \hline  
9.375 E-3 (2.508 E-2)& 1.571 E-4 & 2.044 & 5.059 E-5 & 1.980 & 5.423 E-8 & 3.995 & 5.506 E-4 & 1.959 & 1.894 E-5 & 2.998 \\ \hline
4.697 E-3 (1.254 E-2)& 3.892 E-5 & 2.013 & 1.270 E-5 & 1.993 & 3.392 E-9 & 3.999 & 1.386 E-4 & 1.990 & 2.368 E-6 & 3.000 \\ \hline 
\end{tabular}
}
\caption{{\footnotesize Time domain convergence results for BDF2-based CQ with combined $h$ and $\Delta t$ refinements. In every successive refinement level the size of the time step and the mesh parameter were halved. The table shows the relative errors and estimated convergence rates measured for a final time $t=1.5$ and polynomial degree $k=3$.}}\label{tab:BDF2k3}
\end{table}
\begin{table}\centering
\scalebox{0.74}{\!
\begin{tabular}{ccccccccccccc}
\hline
\multicolumn{3}{|c|}{BDF2} & \multicolumn{4}{|c|}{$L^2(\Omega_-)$} & \multicolumn{4}{|c|}{$H^1(\Omega_-)$} \\
\hline
\multicolumn{1}{|c|}{$\Delta t$ (Ndof)} & \multicolumn{1}{c|}{$E^{v}_{h}$} & \multicolumn{1}{c|}{e.c.r.} & \multicolumn{1}{c|}{$E^{\mathbf u}_{h}$} & \multicolumn{1}{c|}{e.c.r.}   & \multicolumn{1}{c|}{$E^{\theta}_{h,\kappa}$} & \multicolumn{1}{c|}{e.c.r.} & \multicolumn{1}{c|}{$E^{\mathbf u}_{h}$} & \multicolumn{1}{c|}{e.c.r.}   & \multicolumn{1}{c|}{$E^{\theta}_{h}$} & \multicolumn{1}{c|}{e.c.r.} \\ \hline
 3.75 E-2 (108)  & 7.793 E-3  & ---   & 1.231 E-2  & ---   & 5.184 E-3  & ---   & 2.975 E-1 & ---   & 2.222 E-1 & ---    \\ \hline
1.875 E-2  (394)  & 2.775 E-3  & 1.489 & 7.725 E-4  & 3.994 & 3.275 E-4  & 3.984 & 1.258 E-2 & 4.563 & 1.940 E-2 & 3.518  \\ \hline
9.375 E-3  (859)  & 7.955 E-4  & 1.803 & 1.980 E-4  & 1.964 & 4.061 E-5  & 3.012 & 1.916 E-3 & 2.715 & 1.265 E-3 & 3.938  \\ \hline 
4.687 E-3  (1503) & 2.072 E-4  & 1.941 & 5.035 E-5  & 1.975 & 9.408 E-6  & 2.110 & 4.905 E-4 & 1.966 & 1.125 E-4 & 3.489  \\ \hline
2.344 E-3  (2326) & 5.258 E-5  & 1.978 & 1.267 E-5  & 1.991 & 2.329 E-6  & 2.014 & 1.236 E-4 & 1.988 & 2.355 E-5 & 2.259  \\ \hline 
1.172 E-3  (3328) & 1.323 E-5  & 1.991 & 3.175 E-6  & 1.996 & 5.795 E-7  & 2.007 & 3.100 E-5 & 1.995 & 5.825 E-6 & 2.015  \\ \hline 
\end{tabular}
}
\caption{{\footnotesize Time domain convergence results for BDF2-based CQ. The experiments were run with a fixed mesh using $\mathcal P_k$ Lagrangian finite elements and $\mathcal P_k/\mathcal P_{k-1}$ boundary elements. In every successive refinement level the size of the time step was halved and the polynomial degree of the space refinement increased by one. The table shows the relative errors and estimated convergence rates measured for a final time $t=1.5$ as a function of the time step $\Delta t$ and the number of degrees of freedom used in the spatial discretization.}}\label{tab:BDF2pref}
\end{table}
\begin{figure}{\centering
\includegraphics[width=.49\linewidth]{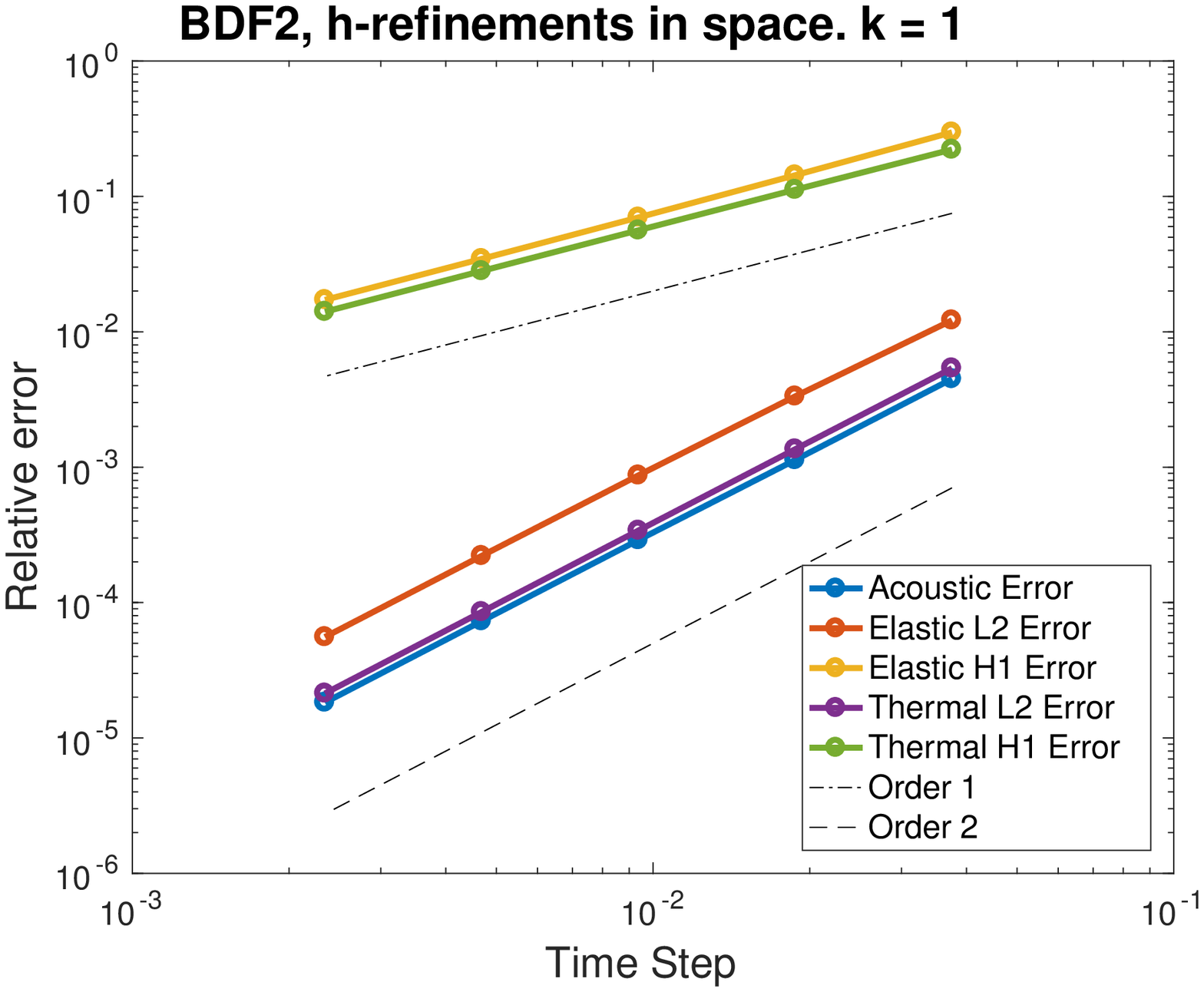}\;
\includegraphics[width=.49\linewidth]{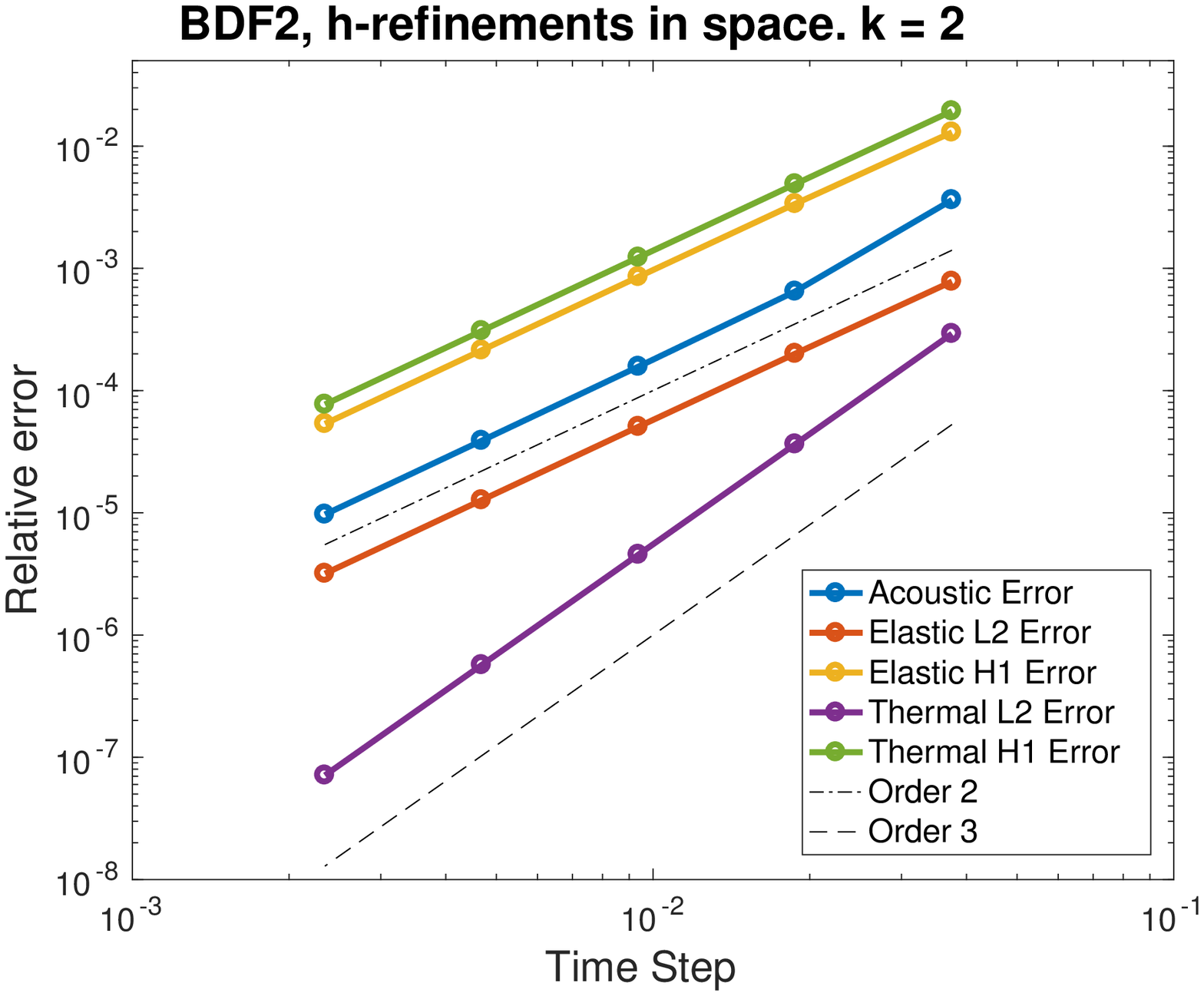} \\
\includegraphics[width=.49\linewidth]{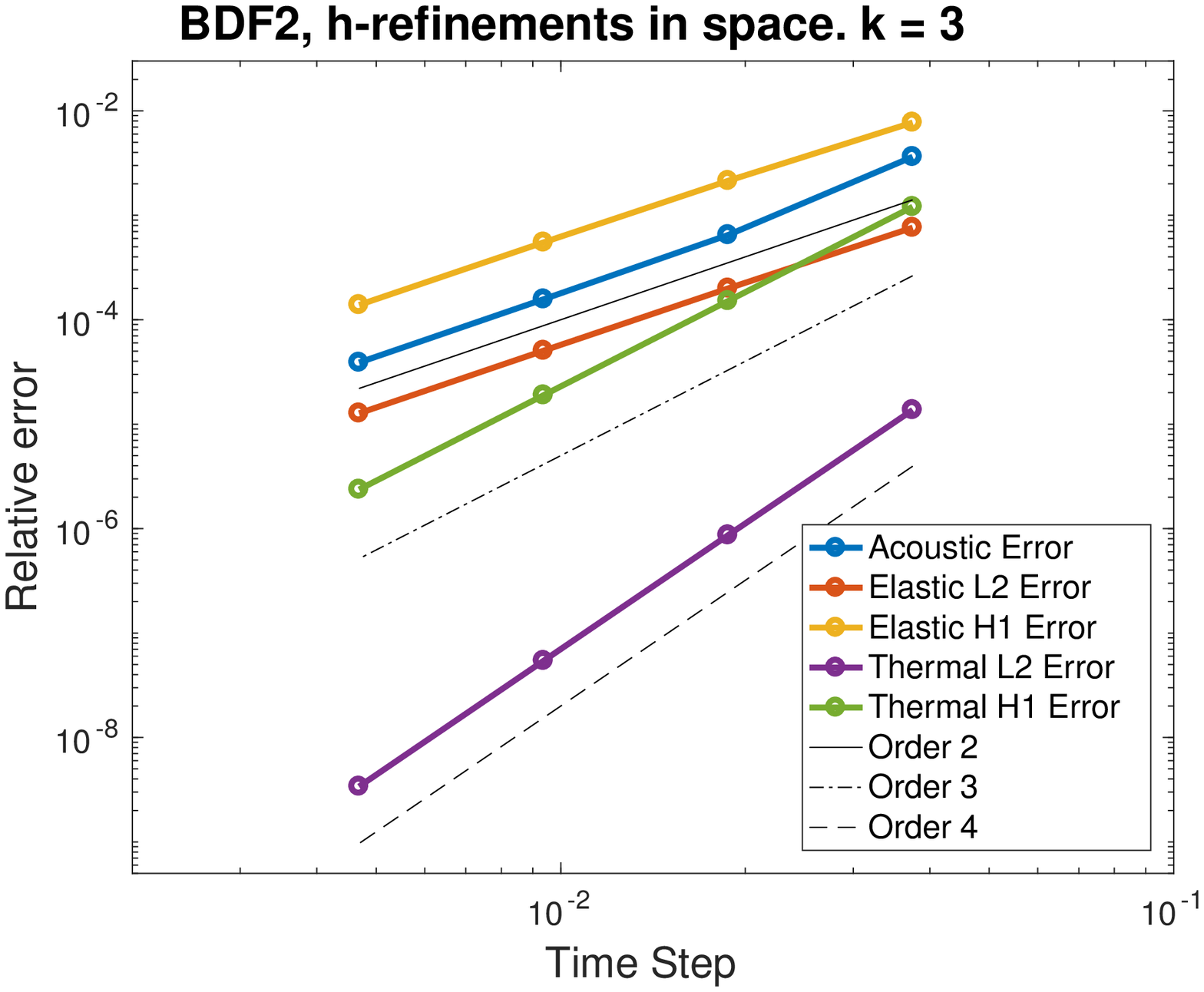}\;
\includegraphics[width=.49\linewidth]{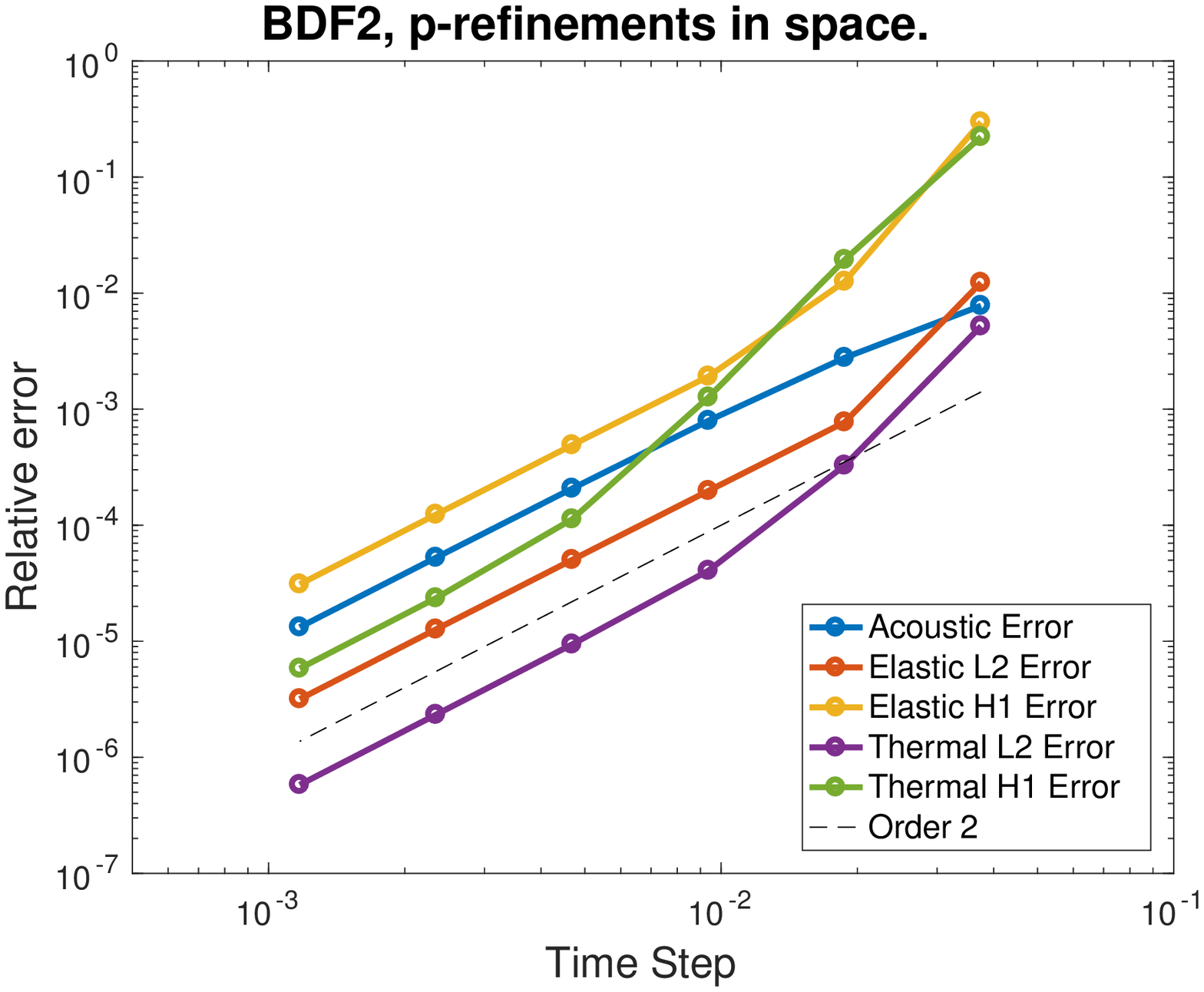}
\caption{{\footnotesize Time-domain convergence studies for the BDF2-based time stepping scheme. For the color code we refer the reader to the electronic version of the manuscript. }}\label{fig:BDF2convergence}
}
\end{figure} 
\clearpage
\begin{table}\centering
\scalebox{0.74}{\!
\begin{tabular}{ccccccccccccc}
\hline
\multicolumn{3}{|c|}{Trapezoidal Rule. k=1} & \multicolumn{4}{|c|}{$L^2(\Omega_-)$} & \multicolumn{4}{|c|}{$H^1(\Omega_-)$} \\
\hline
\multicolumn{1}{|c|}{$\Delta t$   ($h$)} & \multicolumn{1}{c|}{$E^{v}_{h}$} & \multicolumn{1}{c|}{e.c.r.} & \multicolumn{1}{c|}{$E^{\mathbf u}_{h}$} & \multicolumn{1}{c|}{e.c.r.}   & \multicolumn{1}{c|}{$E^{\theta}_{h,\kappa}$} & \multicolumn{1}{c|}{e.c.r.} & \multicolumn{1}{c|}{$E^{\mathbf u}_{h}$} & \multicolumn{1}{c|}{e.c.r.}   & \multicolumn{1}{c|}{$E^{\theta}_{h}$} & \multicolumn{1}{c|}{e.c.r.} \\ \hline
3.75  E-2 (1.000 E-1) & 6.497 E-3 & --- & 1.222 E-2 & --- & 5.381 E-3 & --- & 2.977 E-1 & --- & 2.221 E-1 & --- \\ \hline
1.875 E-2 (5.016 E-2)& 1.648 E-3 & 1.979 & 3.347 E-3 & 1.868 & 1.358 E-3 & 1.987 & 1.434 E-1 & 1.054 & 1.121 E-1 & 0.986 \\ \hline
9.375 E-3 (2.508 E-2)& 4.145 E-4 & 1.991 & 8.731 E-4 & 1.939 & 3.407 E-4 & 1.995 & 6.996 E-2 & 1.036 & 5.624 E-2 & 0.996 \\ \hline
4.697 E-3 (1.254 E-2)& 1.038 E-4 & 1.997 & 2.220 E-4 & 1.975 & 8.528 E-5 & 1.998 & 3.458 E-2 & 1.017 & 2.814 E-2 & 0.999 \\ \hline
2.344 E-3 (6.270 E-3)& 2.596 E-5 & 1.999 & 5.585 E-5 & 1.991 & 2.133 E-5 & 1.999 & 1.721 E-2 & 1.006 & 1.407 E-2 & 1.000 \\ \hline
\end{tabular}
}
\caption{{\footnotesize Time domain convergence results for Trapezoidal Rule-based CQ with combined $h$ and $\Delta t$ refinements. In every successive refinement level the size of the time step and the mesh parameter were halved. The table shows the relative errors and estimated convergence rates measured for a final time $t=1.5$ and polynomial degree $k=1$.}}\label{tab:TRk1}
\end{table}
\begin{table}\centering
\scalebox{0.74}{\!
\begin{tabular}{ccccccccccccc}
\hline
\multicolumn{3}{|c|}{Trapezoidal Rule. k=2} & \multicolumn{4}{|c|}{$L^2(\Omega_-)$} & \multicolumn{4}{|c|}{$H^1(\Omega_-)$} \\
\hline
\multicolumn{1}{|c|}{$\Delta t$   ($h$)} & \multicolumn{1}{c|}{$E^{v}_{h}$} & \multicolumn{1}{c|}{e.c.r.} & \multicolumn{1}{c|}{$E^{\mathbf u}_{h}$} & \multicolumn{1}{c|}{e.c.r.}   & \multicolumn{1}{c|}{$E^{\theta}_{h,\kappa}$} & \multicolumn{1}{c|}{e.c.r.} & \multicolumn{1}{c|}{$E^{\mathbf u}_{h}$} & \multicolumn{1}{c|}{e.c.r.}   & \multicolumn{1}{c|}{$E^{\theta}_{h}$} & \multicolumn{1}{c|}{e.c.r.} \\ \hline
3.75  E-2 (1.000 E-1) & 6.150 E-4 &--- & 2.709 E-4 & --- & 2.917 E-4 & --- & 1.069E-2 & --- & 1.934E-2 & --- \\ \hline
1.875 E-2 (5.016 E-2) & 1.594 E-4 & 1.948 & 5.531 E-5 & 2.292 & 3.640 E-5 & 3.002 & 2.668 E-3 & 2.002 & 4.890 E-3 & 1.983 \\ \hline 
9.375 E-3 (2.508 E-2)& 4.114 E-5 & 1.954 & 1.301 E-5 & 2.088 & 4.550 E-6 & 3.000 & 6.662 E-4 & 2.002 & 1.228 E-3 & 1.993 \\ \hline  
4.697 E-3 (1.254 E-2)& 1.036 E-5 & 1.989 & 3.202 E-6 & 2.022 & 5.692 E-7 & 2.999 & 1.664 E-4 & 2.001 & 3.076 E-4 & 1.997 \\ \hline
2.344 E-3 (6.270 E-3) & 2.596 E-6 & 1.997 & 7.974 E-7 & 2.005 & 7.119 E-8 & 2.999 & 4.159 E-5 & 2.001 & 7.697 E-5 & 1.999 \\ \hline
\end{tabular}
}
\caption{{\footnotesize Time domain convergence results for Trapezoidal Rule-based CQ with combined $h$ and $\Delta t$ refinements. In every successive refinement level the size of the time step and the mesh parameter were halved. The table shows the relative errors and estimated convergence rates measured for a final time $t=1.5$ and polynomial degree $k=2$.}}\label{tab:TRk2}
\end{table}
\begin{table}\centering
\scalebox{0.74}{\!
\begin{tabular}{ccccccccccccc}
\hline
\multicolumn{3}{|c|}{Trapezoidal Rule. k=3} & \multicolumn{4}{|c|}{$L^2(\Omega_-)$} & \multicolumn{4}{|c|}{$H^1(\Omega_-)$} \\
\hline
\multicolumn{1}{|c|}{$\Delta t$   ($h$)} & \multicolumn{1}{c|}{$E^{v}_{h}$} & \multicolumn{1}{c|}{e.c.r.} & \multicolumn{1}{c|}{$E^{\mathbf u}_{h}$} & \multicolumn{1}{c|}{e.c.r.}   & \multicolumn{1}{c|}{$E^{\theta}_{h,\kappa}$} & \multicolumn{1}{c|}{e.c.r.} & \multicolumn{1}{c|}{$E^{\mathbf u}_{h}$} & \multicolumn{1}{c|}{e.c.r.}   & \multicolumn{1}{c|}{$E^{\theta}_{h}$} & \multicolumn{1}{c|}{e.c.r.} \\ \hline
3.75  E-2 (1.000 E-1)& 6.108 E-4 & --- & 2.027 E-4 & --- & 1.368 E-5 & --- & 2.204 E-3 & --- & 1.205 E-3 & --- \\ \hline
1.875 E-2 (5.016 E-2)& 1.601 E-4 & 1.932 & 5.090 E-5 & 1.994 & 8.650 E-7 & 3.983 & 5.550 E-4 & 1.990 & 1.513 E-4 & 2.994 \\ \hline  
9.375 E-3 (2.508 E-2)& 4.122 E-5 & 1.958 & 1.274 E-5 & 1.998 & 5.424 E-8 & 3.995 & 1.390 E-4 & 1.998 & 1.894 E-5 & 2.998 \\ \hline
4.697 E-3 (1.254 E-2)& 1.037 E-5 & 1.991 & 3.186 E-6 & 2.000 & 3.392 E-9 & 3.999 & 3.475 E-5 & 1.999 & 2.368 E-6 & 3.000 \\ \hline 
\end{tabular}
}
\caption{{\footnotesize Time domain convergence results for Trapezoidal Rule-based CQ with combined $h$ and $\Delta t$ refinements. In every successive refinement level the size of the time step and the mesh parameter were halved. The table shows the relative errors and estimated convergence rates measured for a final time $t=1.5$ and polynomial degree $k=3$.}}\label{tab:TRk3}
\end{table}
\begin{table}\centering
\scalebox{0.74}{\!
\begin{tabular}{ccccccccccccc}
\hline
\multicolumn{3}{|c|}{Trapezoidal Rule} & \multicolumn{4}{|c|}{$L^2(\Omega_-)$} & \multicolumn{4}{|c|}{$H^1(\Omega_-)$} \\
\hline
\multicolumn{1}{|c|}{$\Delta t$ (Ndof)} & \multicolumn{1}{c|}{$E^{v}_{h}$} & \multicolumn{1}{c|}{e.c.r.} & \multicolumn{1}{c|}{$E^{\mathbf u}_{h}$} & \multicolumn{1}{c|}{e.c.r.}   & \multicolumn{1}{c|}{$E^{\theta}_{h,\kappa}$} & \multicolumn{1}{c|}{e.c.r.} & \multicolumn{1}{c|}{$E^{\mathbf u}_{h}$} & \multicolumn{1}{c|}{e.c.r.}   & \multicolumn{1}{c|}{$E^{\theta}_{h}$} & \multicolumn{1}{c|}{e.c.r.} \\ \hline
 3.75 E-2 (108)  & 5.620 E-3  & ---   & 1.218 E-2  & ---   & 5.213 E-3  & ---   & 2.976 E-1 & ---   & 2.221 E-1 & ---    \\ \hline
1.875 E-2 (394)  & 8.283 E-4  & 2.762 & 2.713 E-4  & 5.489 & 2.934 E-4  & 4.151 & 1.064 E-2 & 4.805 & 1.934 E-2 & 3.522  \\ \hline
9.375 E-3 (859)  & 2.107 E-4  & 1.975 & 5.085 E-5  & 2.416 & 1.660 E-5  & 4.144 & 4.958 E-4 & 4.424 & 1.209 E-3 & 4.000  \\ \hline 
4.687 E-3 (1503) & 5.278 E-5  & 1.997 & 1.272 E-5  & 1.999 & 2.349 E-6  & 2.821 & 1.242 E-4 & 1.997 & 6.549 E-5 & 4.206  \\ \hline
2.344 E-3 (2326) & 1.320 E-5  & 1.996 & 3.184 E-6  & 1.999 & 5.770 E-7  & 2.026 & 3.107 E-5 & 1.999 & 6.286 E-6 & 3.381  \\ \hline 
1.172 E-3 (3328) & 3.300 E-6  & 2.000 & 7.956 E-7  & 2.000 & 1.442 E-7  & 2.001 & 7.770 E-6 & 2.000 & 1.451 E-6 & 2.115  \\ \hline  
\end{tabular}
}
\caption{{\footnotesize Time domain convergence results for Trapezoidal Rule-based CQ. The experiments were run with a fixed mesh using $\mathcal P_k$ Lagrangian finite elements and $\mathcal P_k/\mathcal P_{k-1}$ boundary elements. In every successive refinement level the size of the time step was halved and the polynomial degree of the space refinement increased by one. The table shows the relative errors and estimated convergence rates measured for a final time $t=1.5$ as a function of the time step $\Delta t$ and the number of degrees of freedom used in the spatial discretization.}}\label{tab:TRpref}
\end{table}
\begin{figure}{\centering
\includegraphics[width=.49\linewidth]{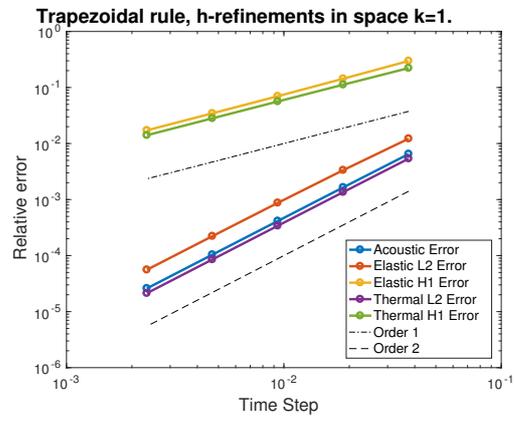}\;
\includegraphics[width=.49\linewidth]{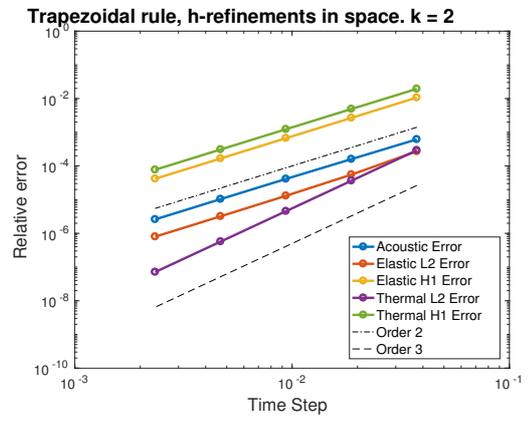} \\
\includegraphics[width=.49\linewidth]{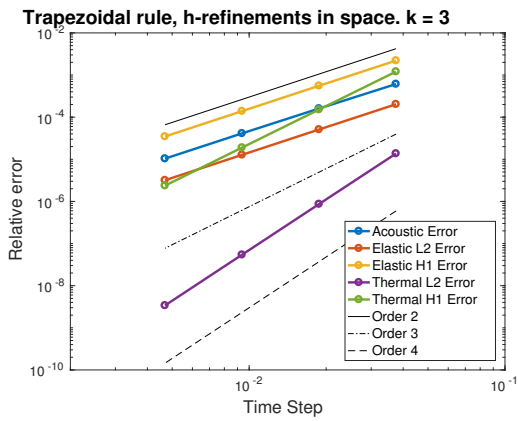}\;
\includegraphics[width=.49\linewidth]{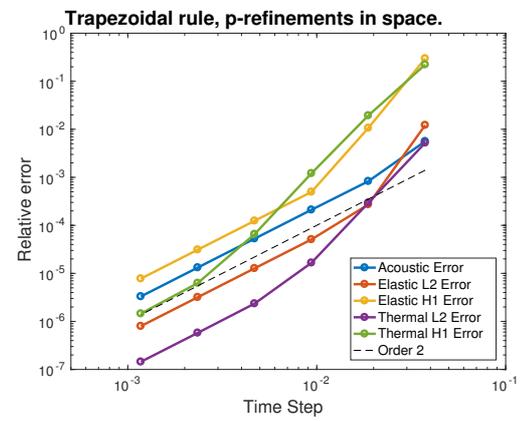}
\caption{{\footnotesize Time-domain convergence studies for the Trapezoidal Rule-based time stepping scheme. For the color code we refer the reader to the electronic version of the manuscript. }}\label{fig:TRconvergence}
}
\end{figure} 
\newpage
% ====================
% \paragraph*{Examples}
% ====================
\begin{figure}{\centering
\includegraphics[height =7cm]{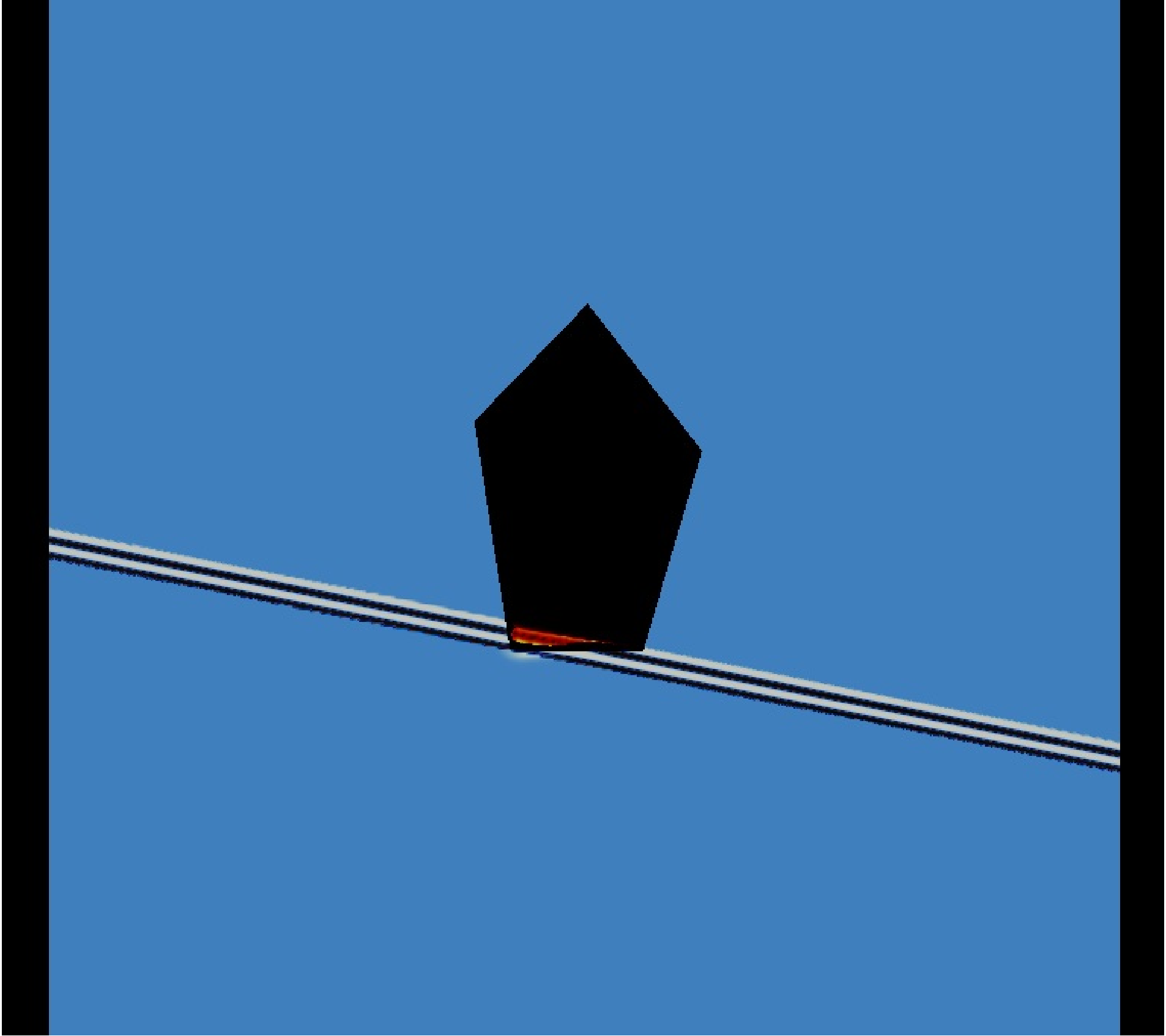}
\includegraphics[height =7cm]{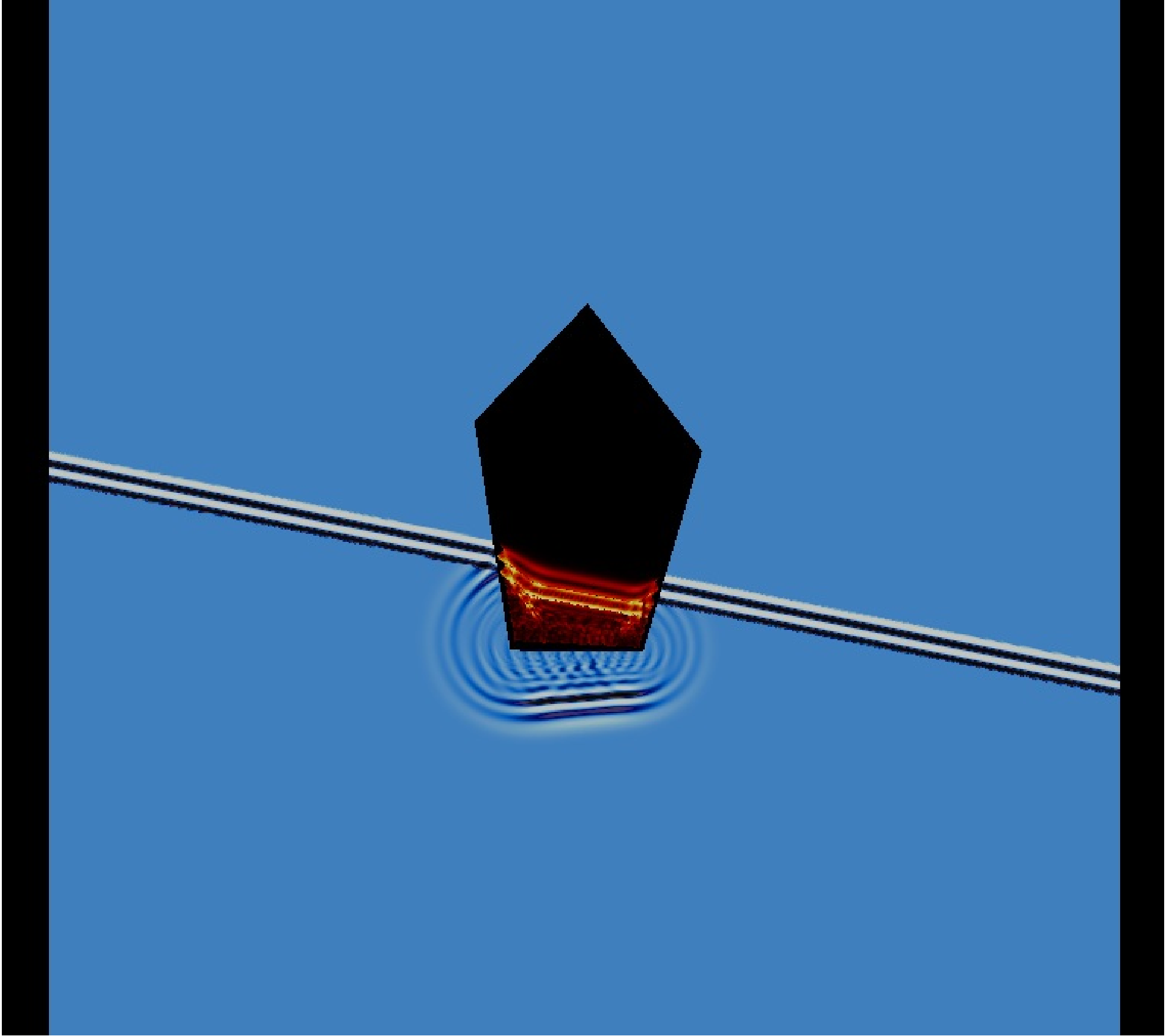}
\includegraphics[height =7cm]{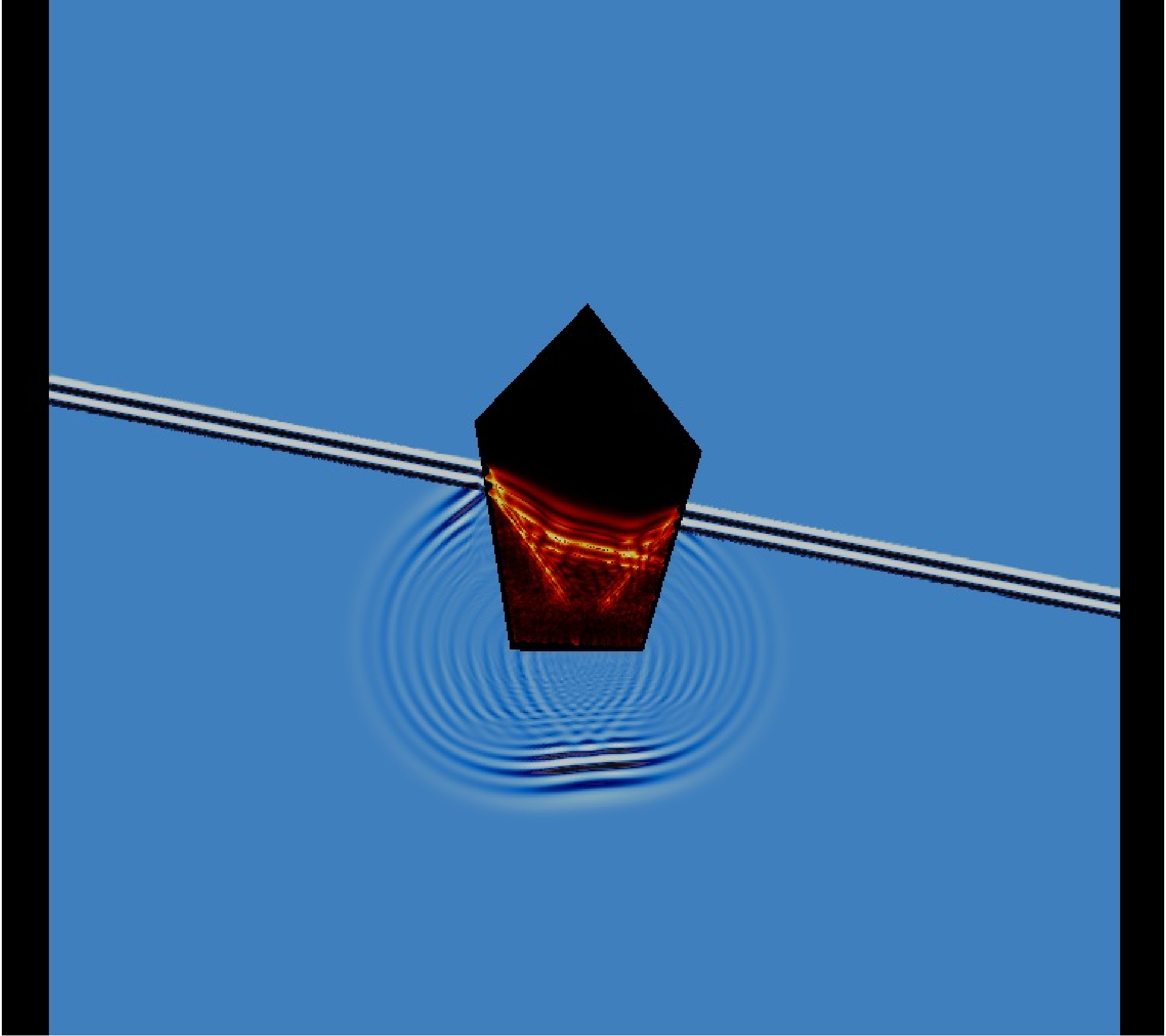}
\includegraphics[height =7cm]{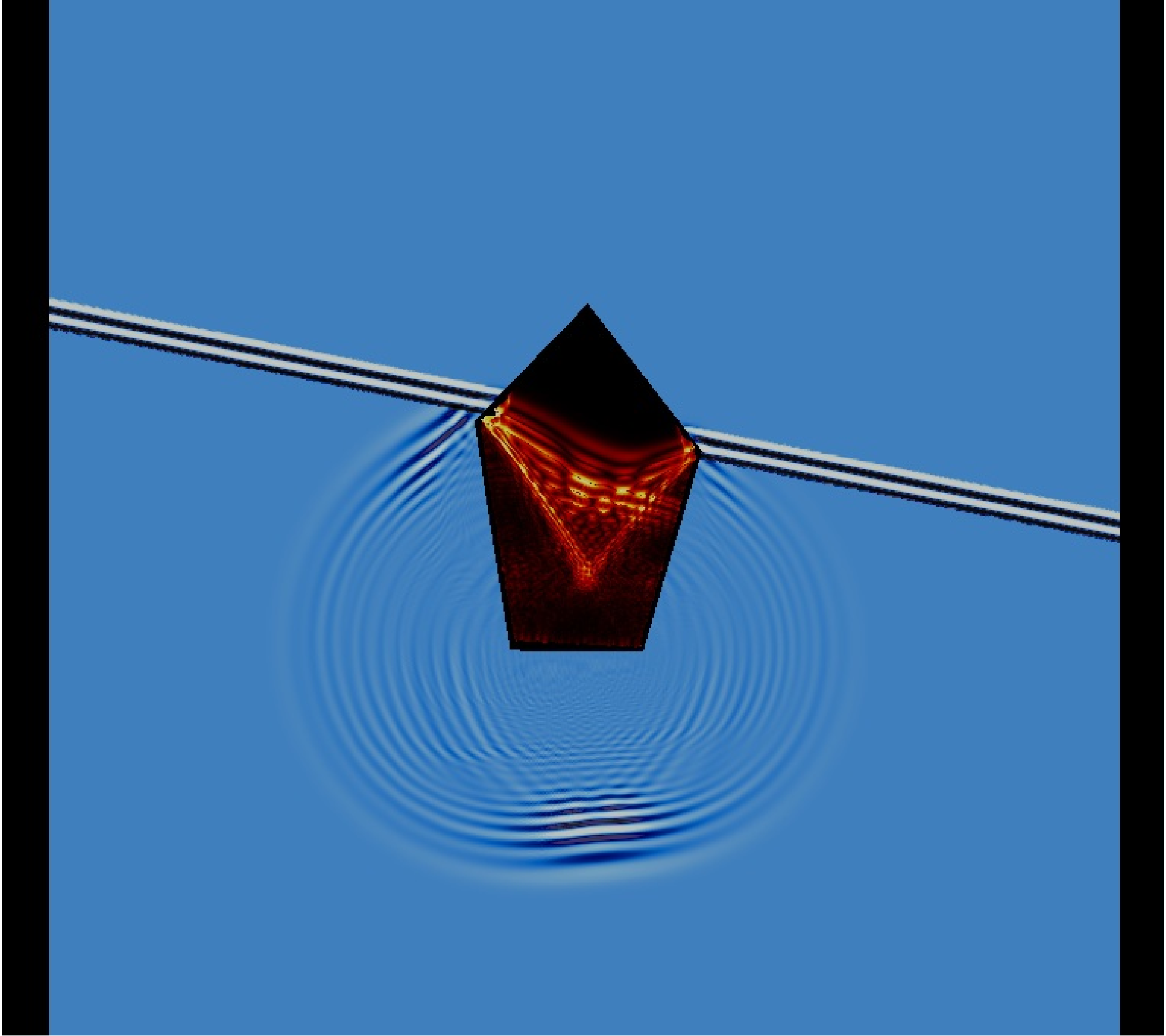}
\includegraphics[height =7cm]{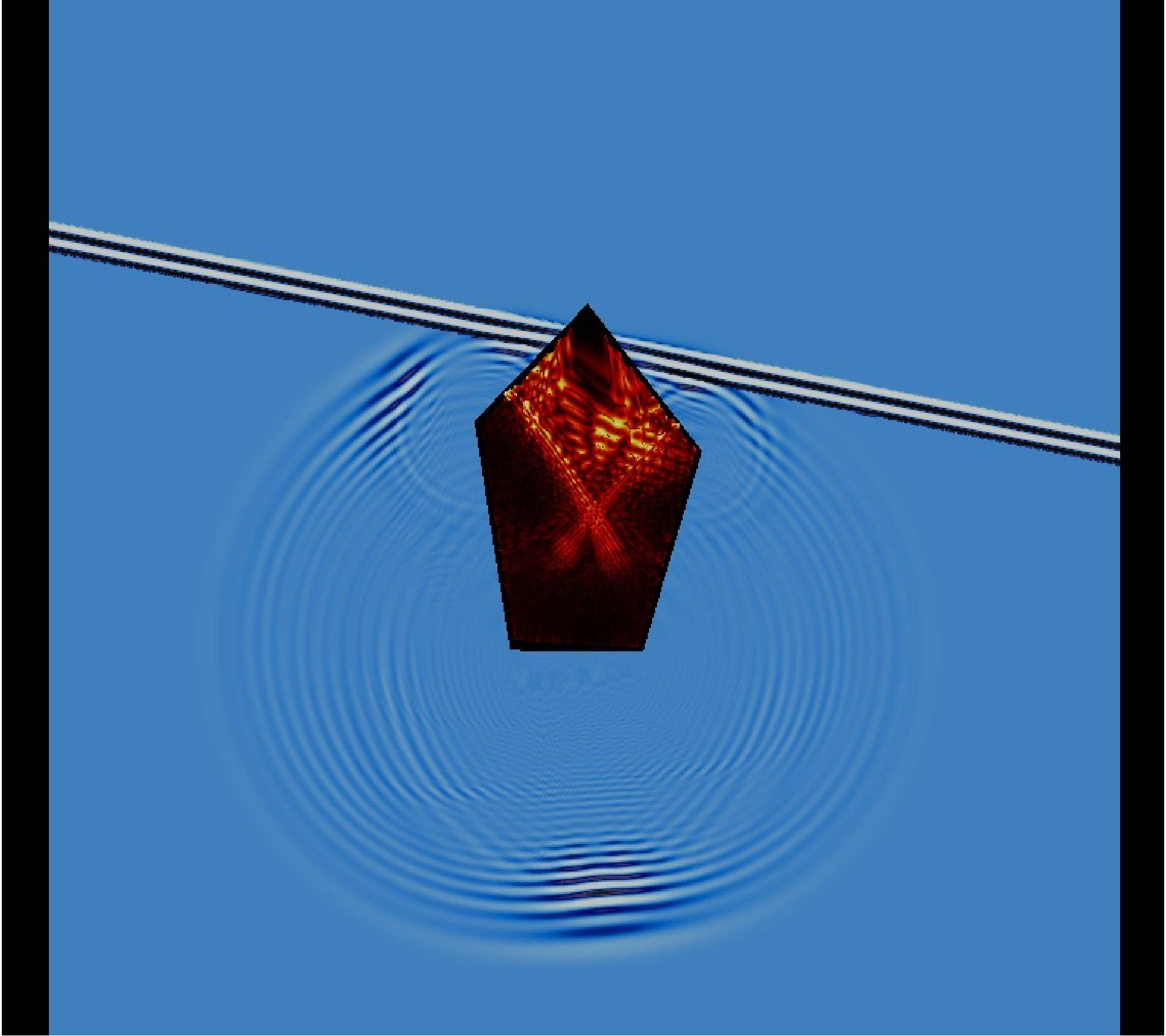}
\includegraphics[height =7cm]{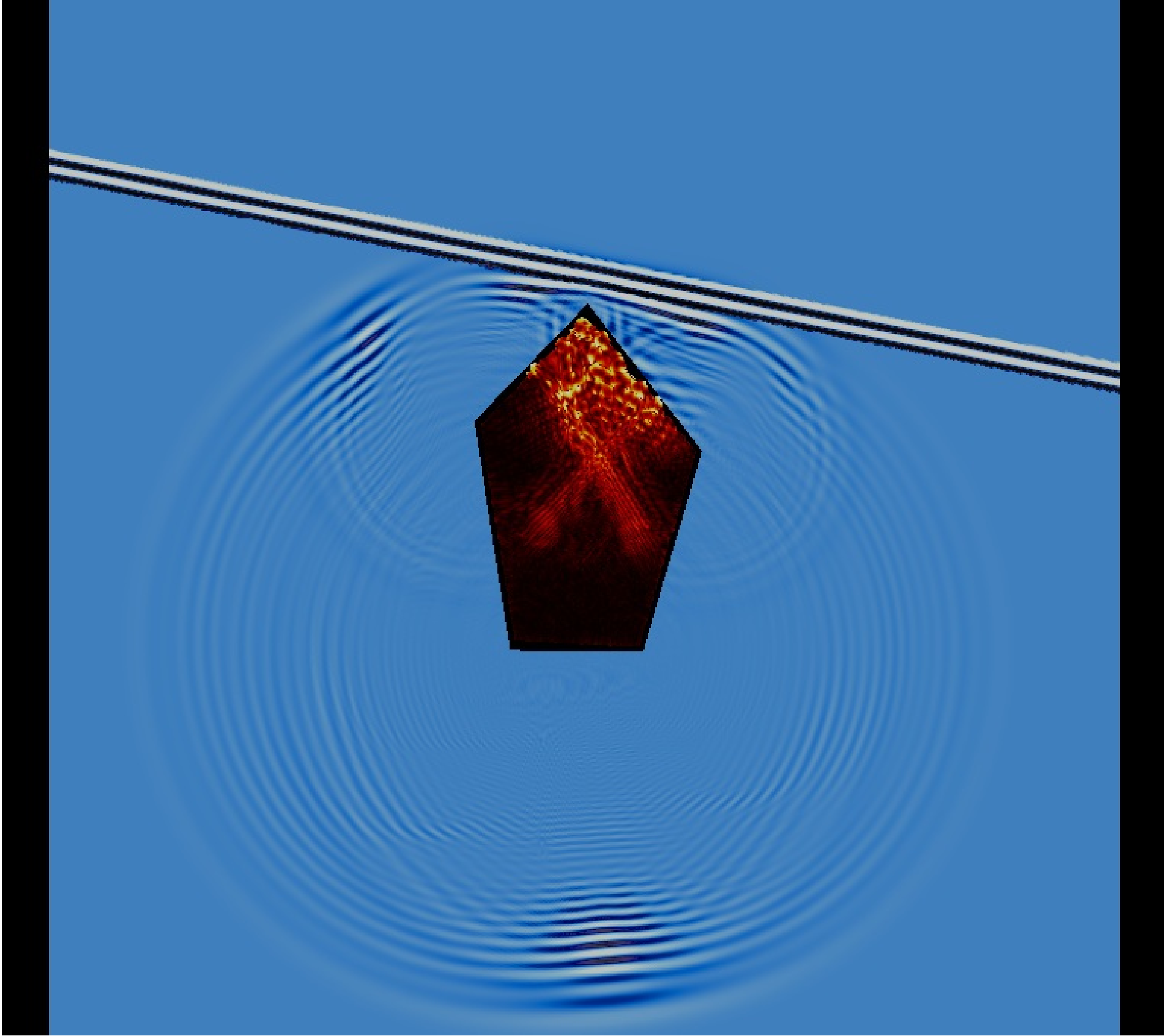}
\caption[Thermoelastic scattering by a pentagon]{{\footnotesize Snapshots of the total acoustic field at times $t=0.25,0.6, 0.95,1.3,1.65,2$. The interior domain shows the norm of the elastic displacement.}}\label{fig:c5:5}
}
\end{figure}

\begin{figure}{\centering
\includegraphics[height =5cm]{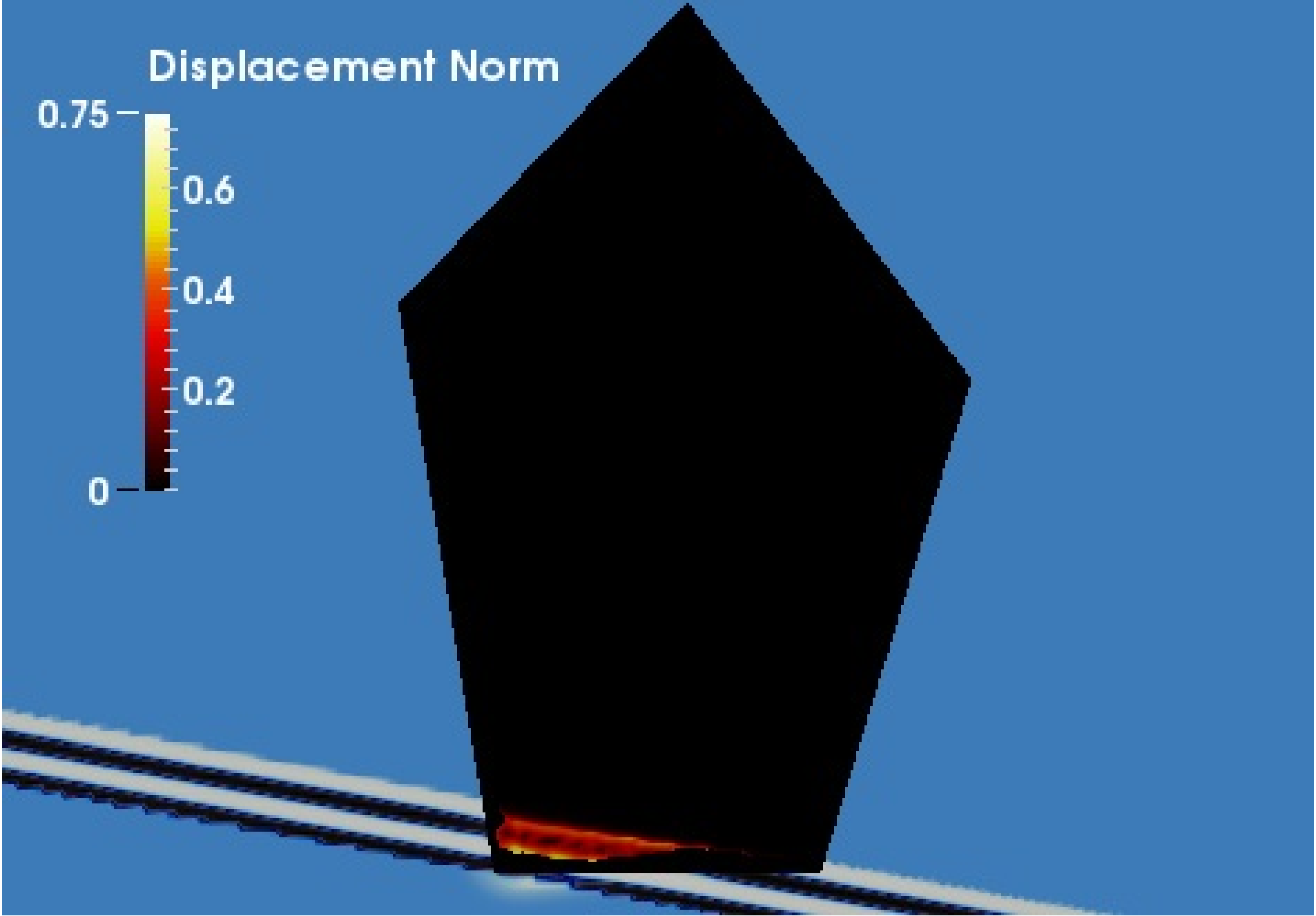}
\includegraphics[height =5cm]{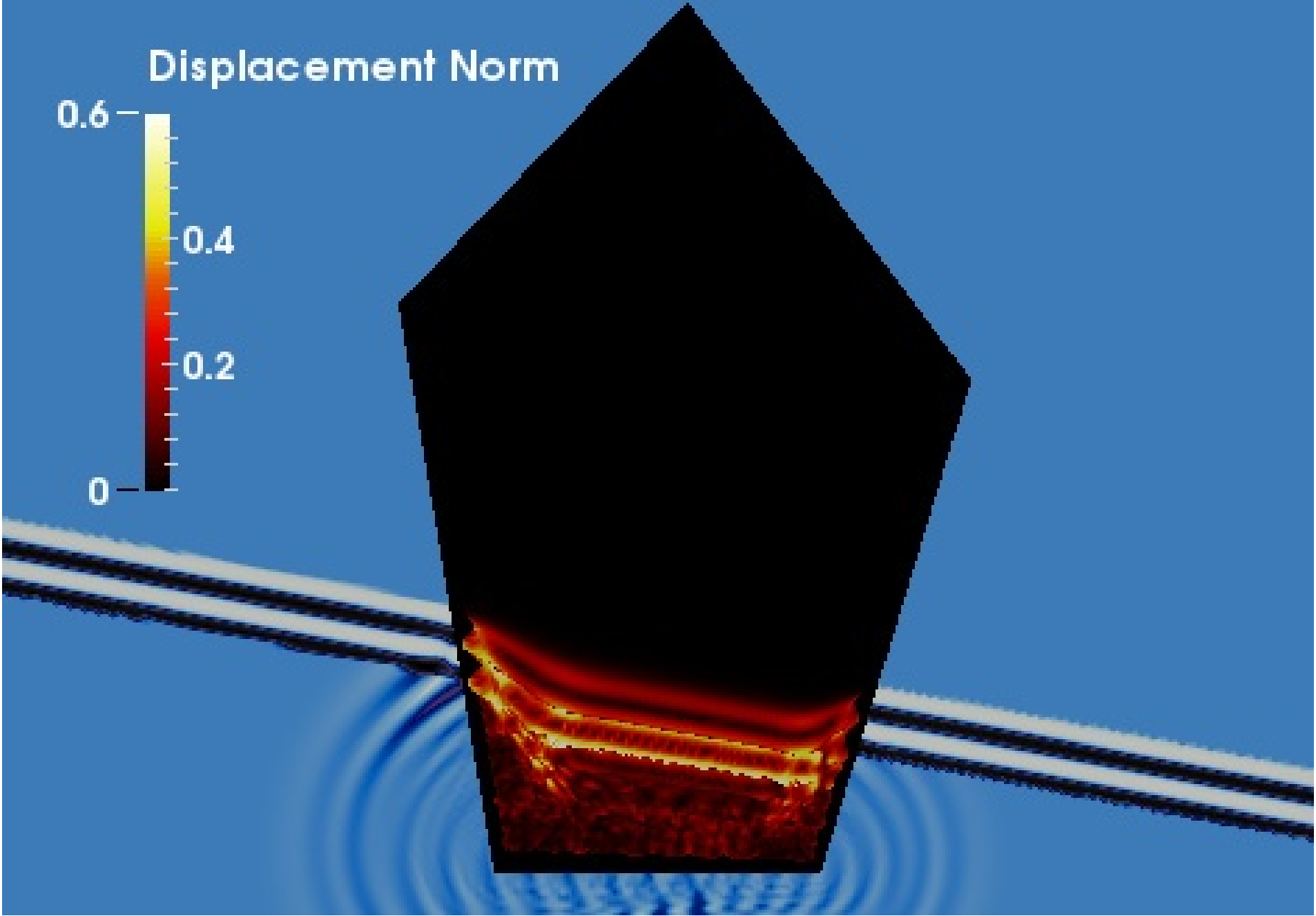}
\includegraphics[height =5cm]{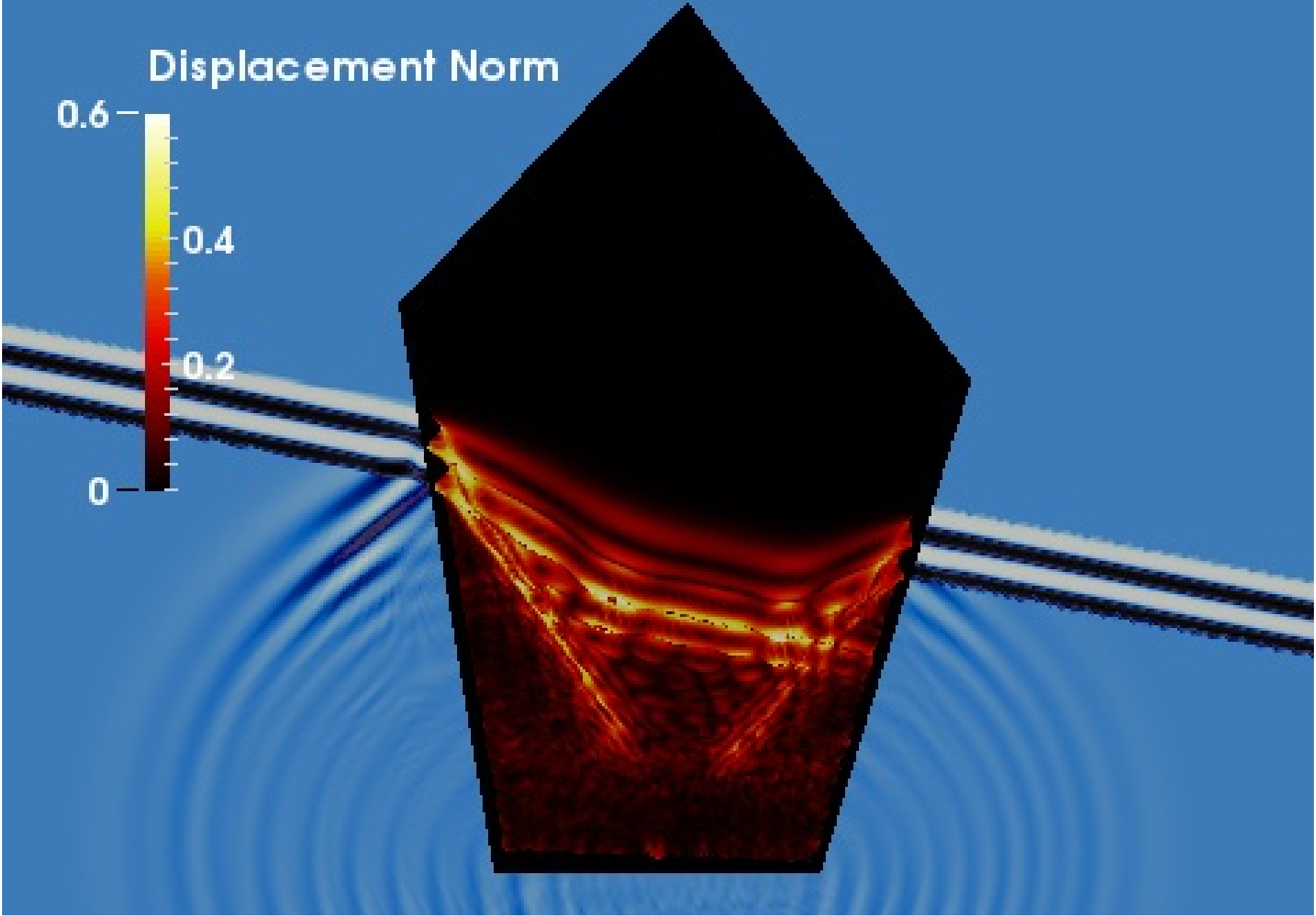}
\includegraphics[height =5cm]{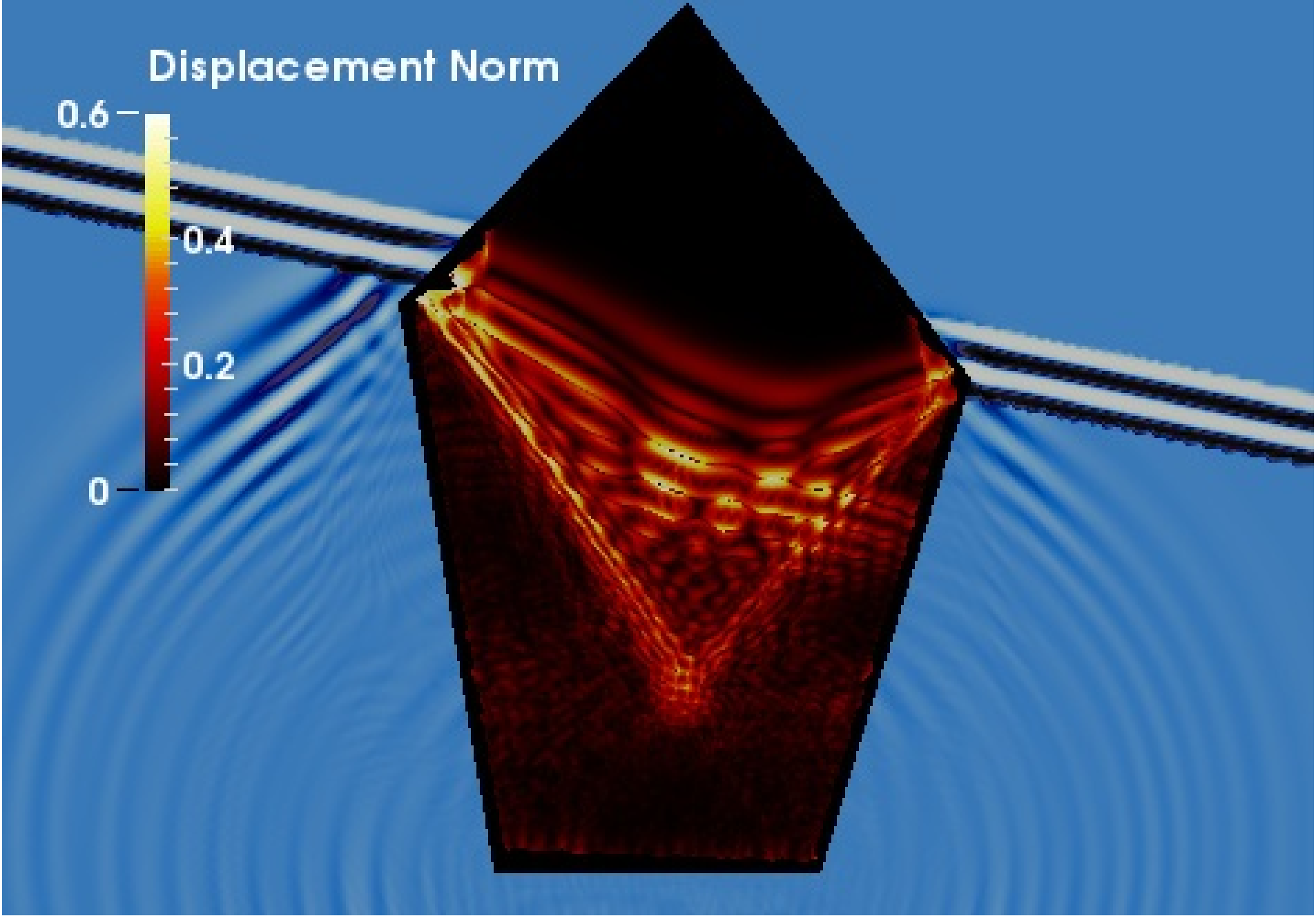}
\includegraphics[height =5cm]{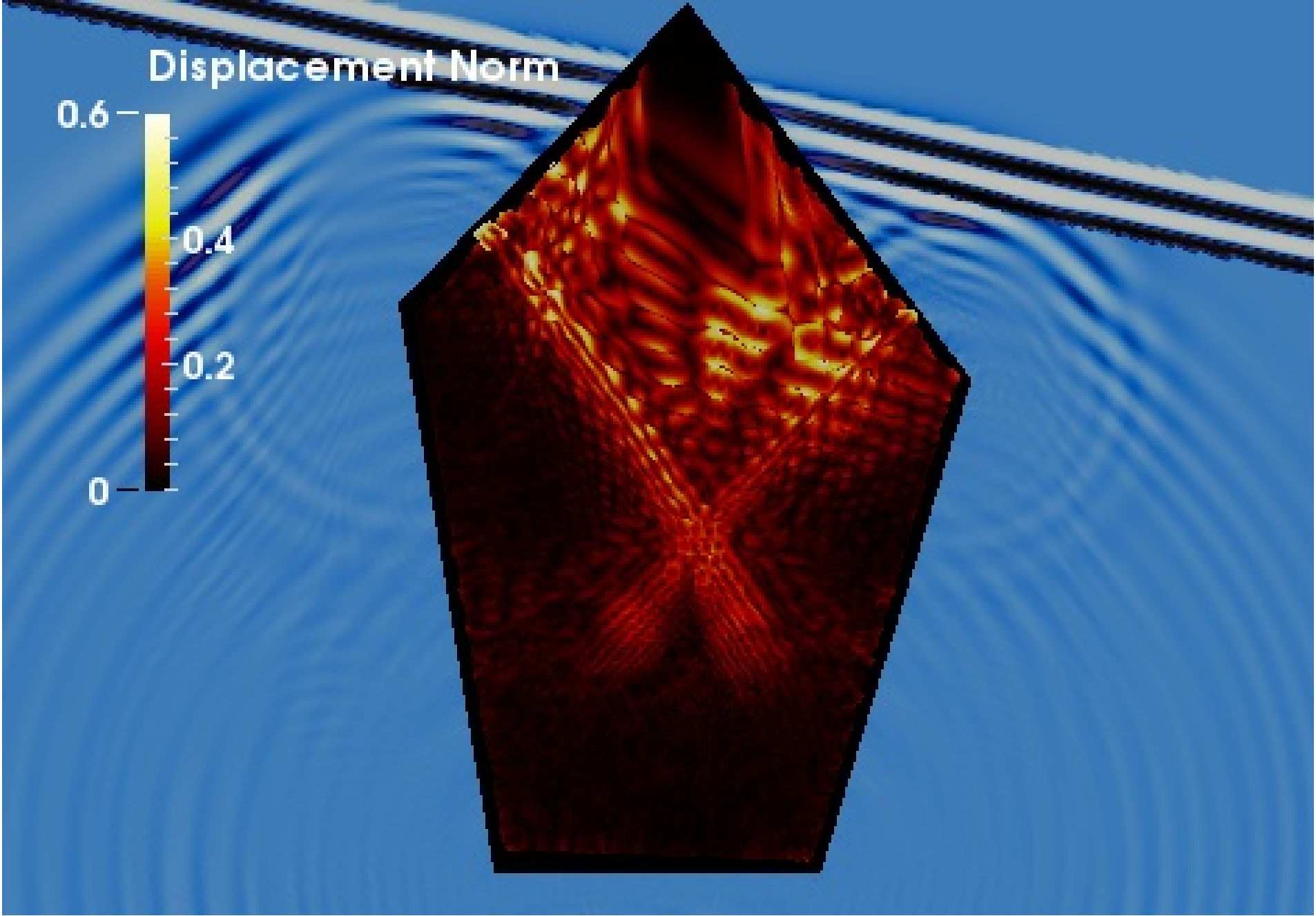}
\includegraphics[height =5cm]{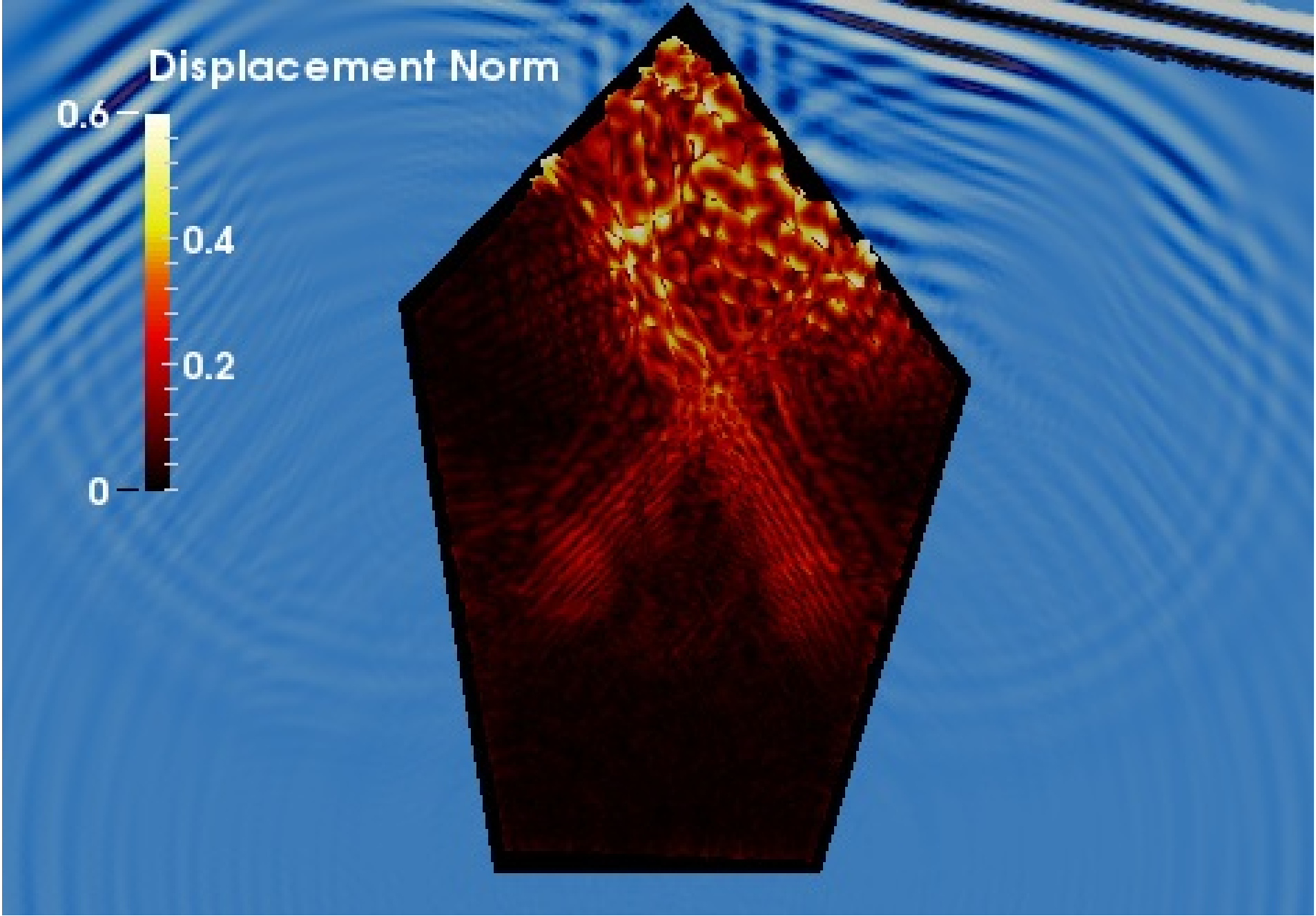}
\caption[Thermoelastic scattering by a pentagon]{{\footnotesize Close up of the norm of the elastic displacement for times $t=0.25,0.6, 0.95,1.3,1.65,2$. Black represents no displacement.}} \label{fig:c5:6}
}
\end{figure}

\begin{figure}{\centering
\includegraphics[height =5cm]{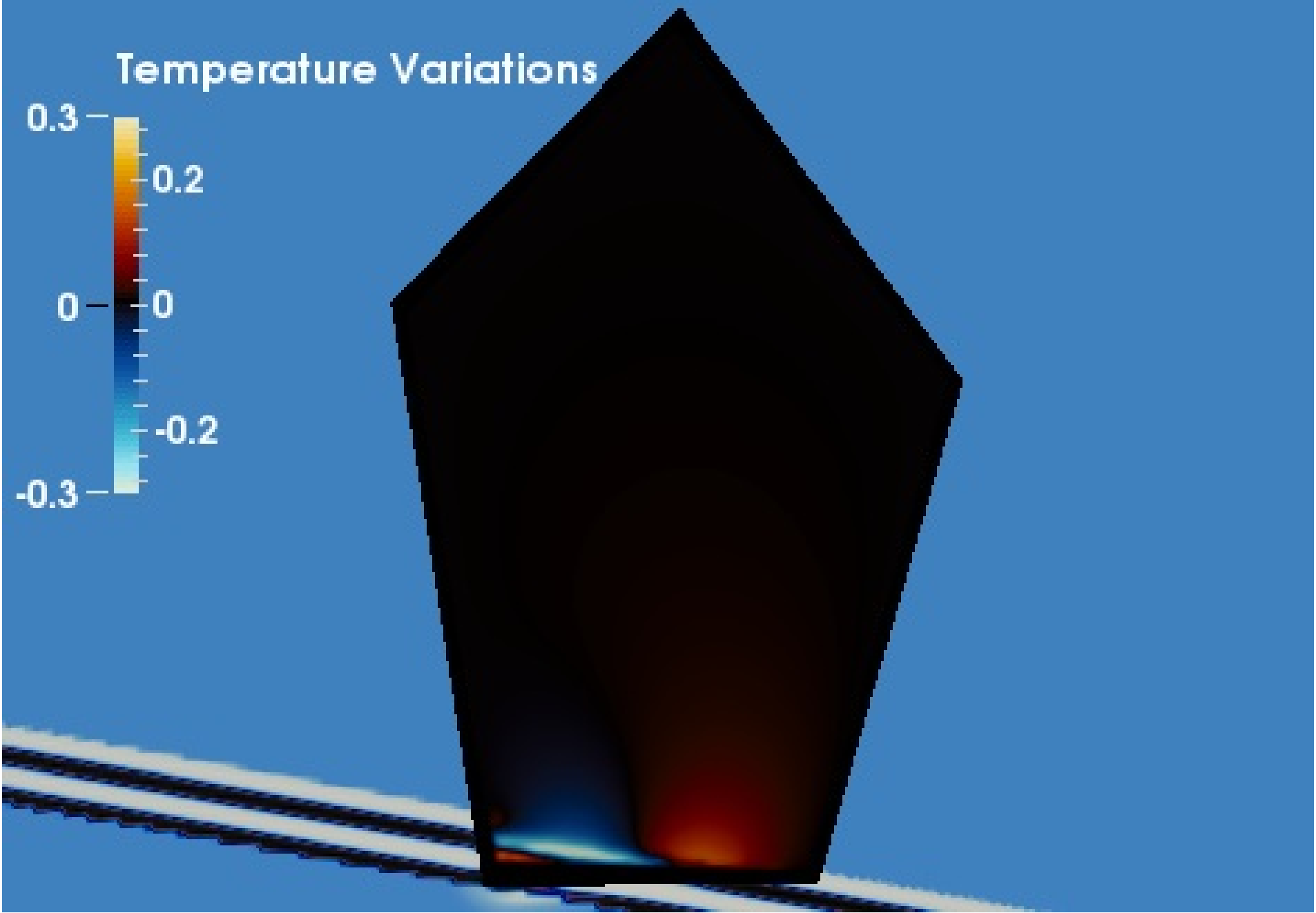}
\includegraphics[height =5cm]{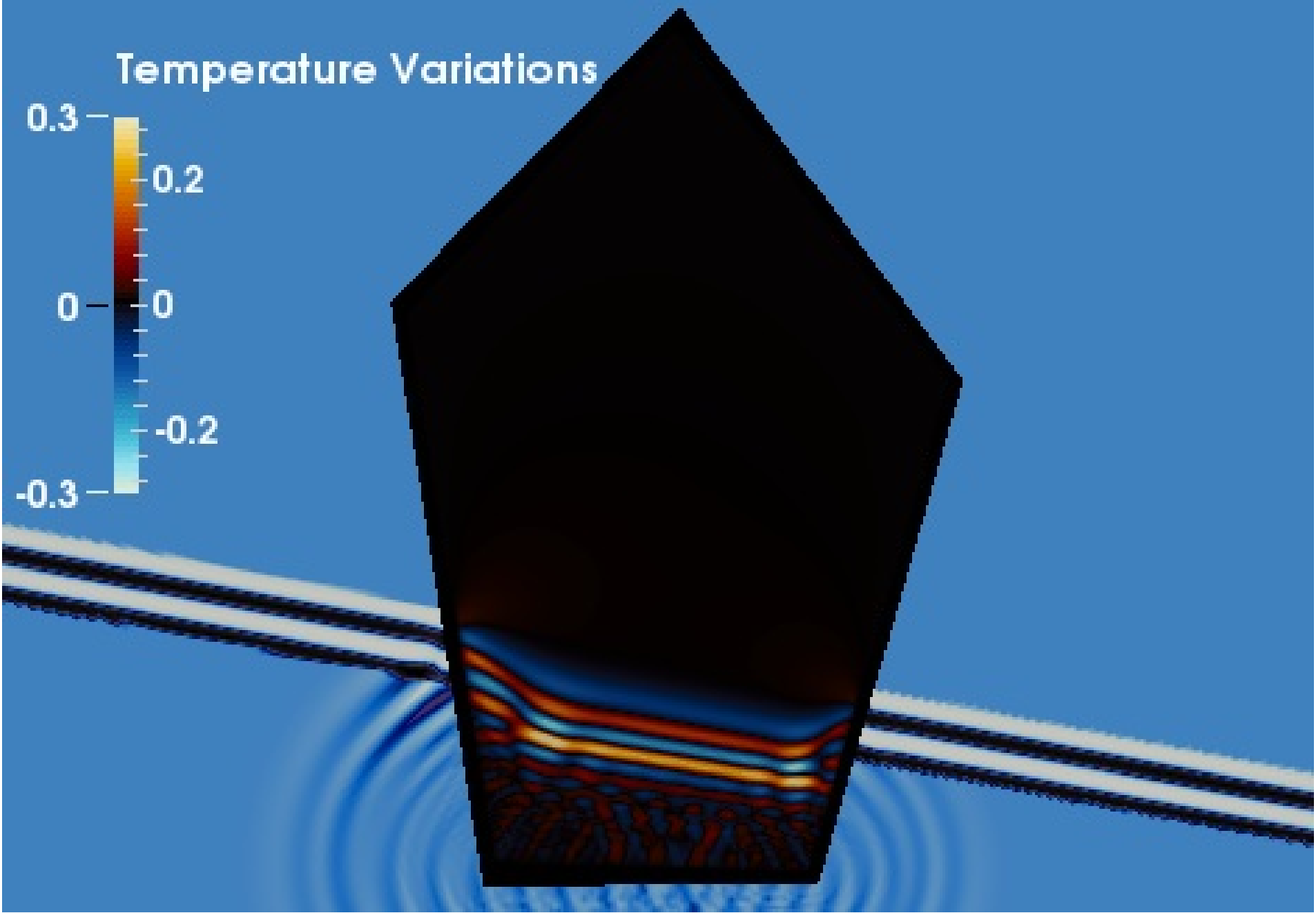}
\includegraphics[height =5cm]{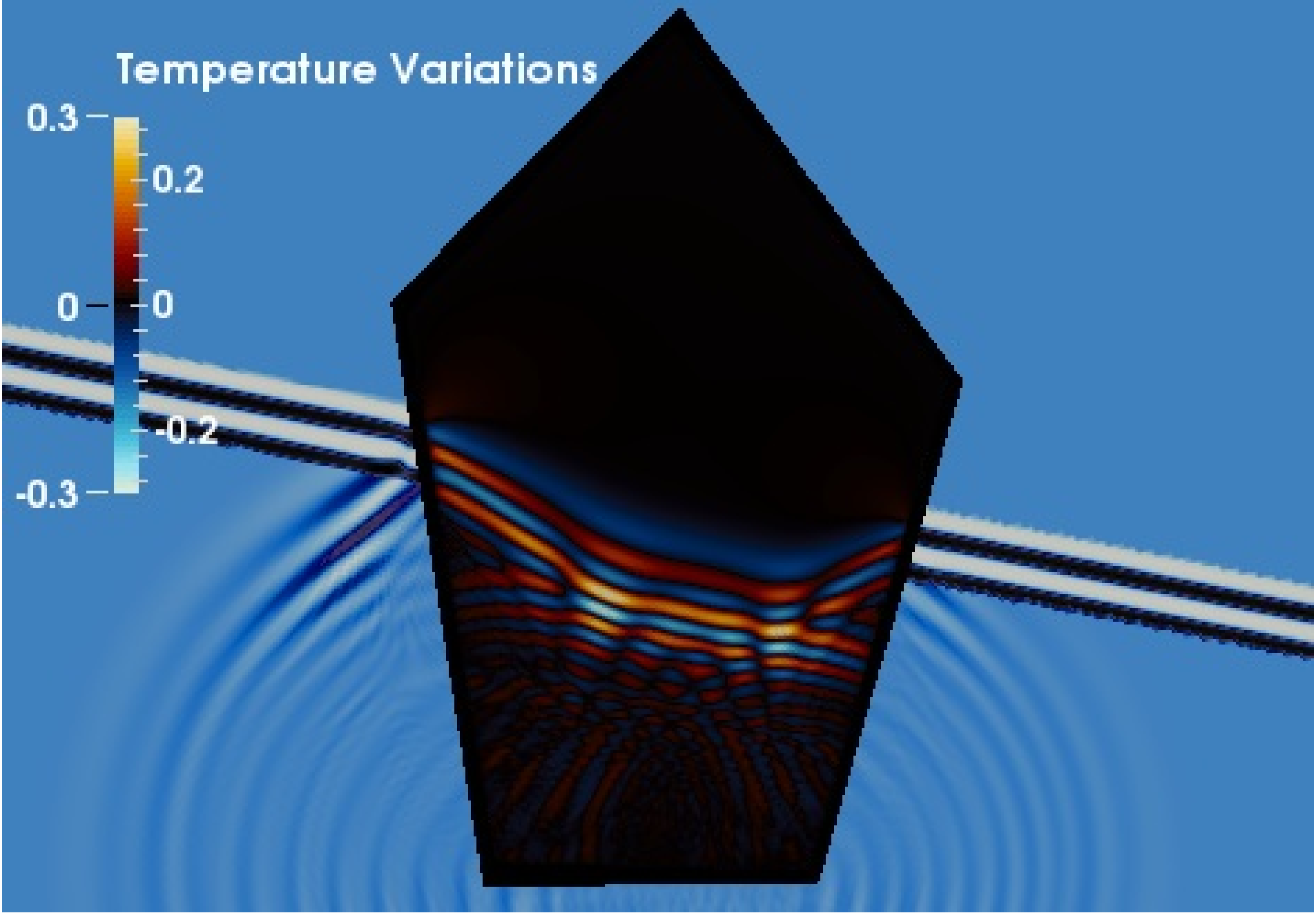}
\includegraphics[height =5cm]{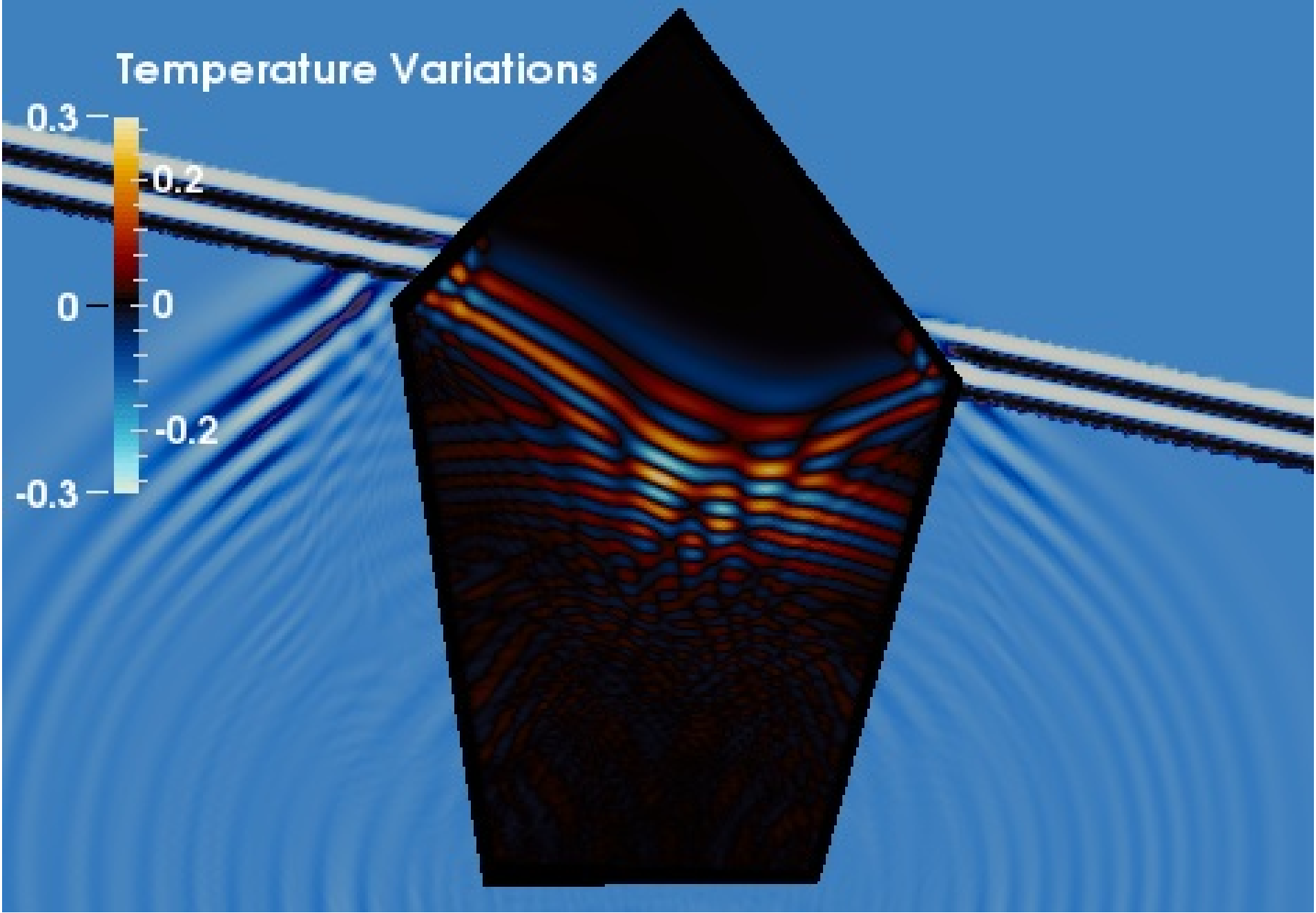}
\includegraphics[height =5cm]{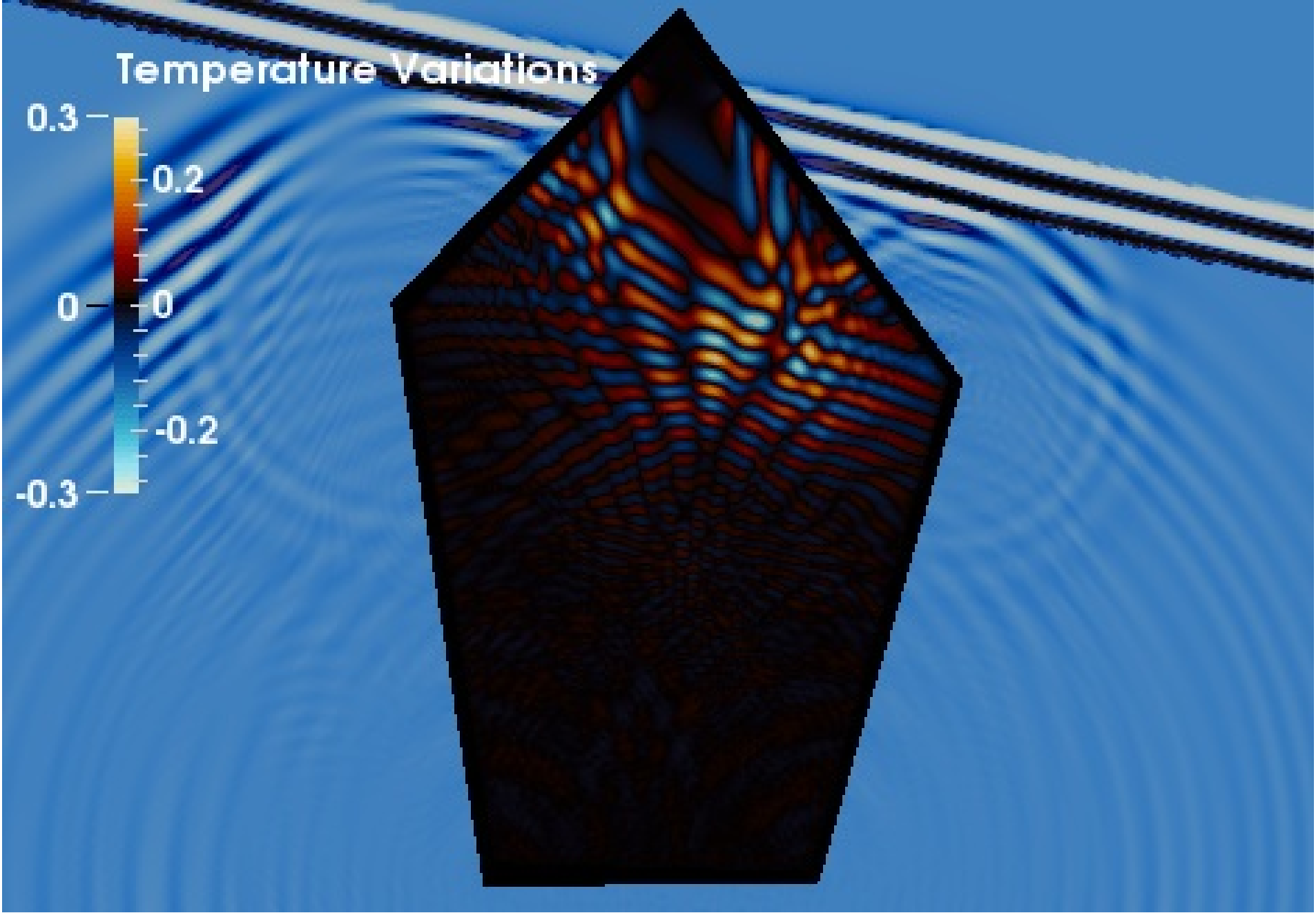}
\includegraphics[height =5cm]{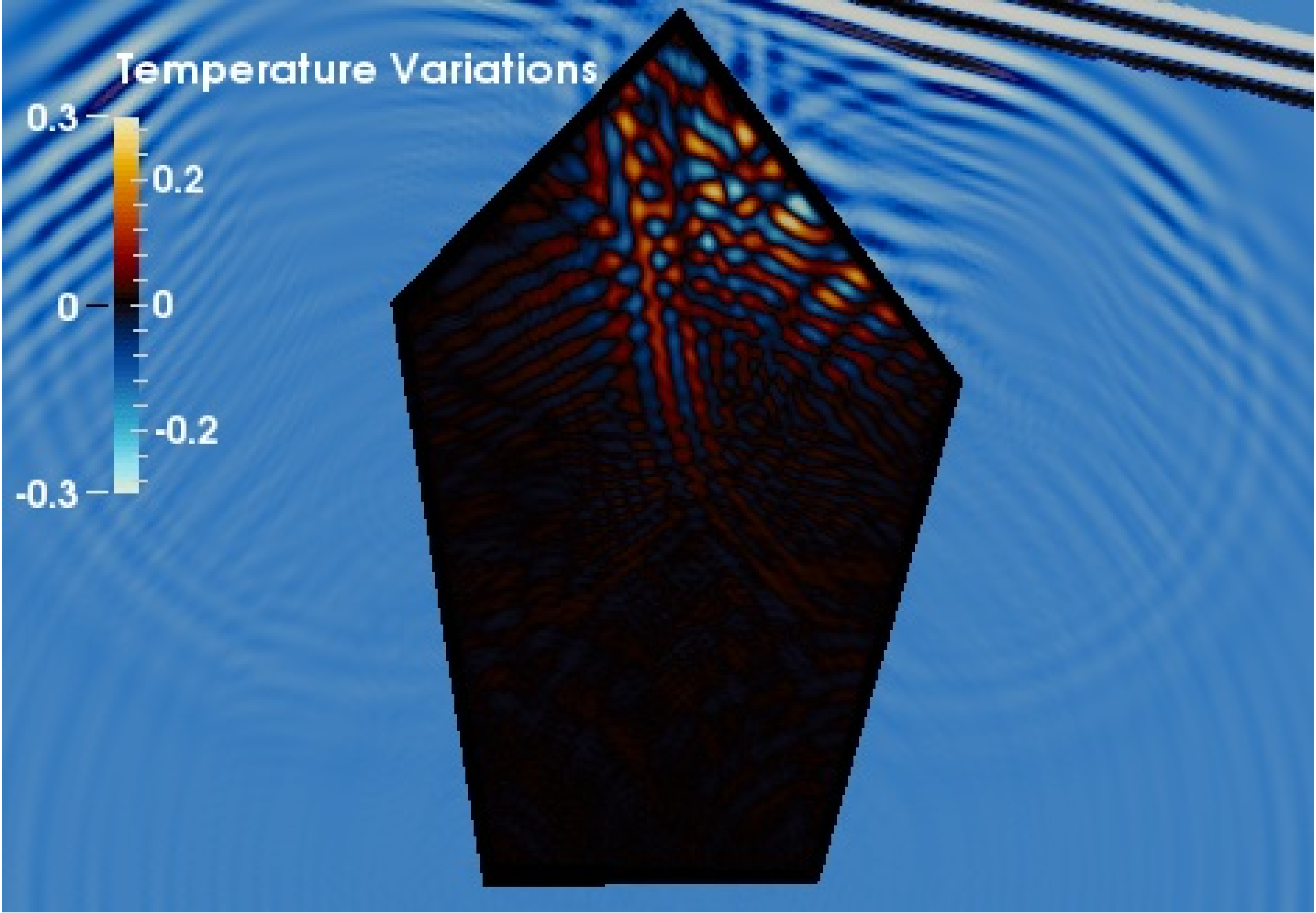}
\caption[Thermoelastic scattering by a pentagon]{{\footnotesize Close up of the norm of the temperature variations with respect to the reference configuration for times $t=0.25,0.6, 0.95,1.3,1.65,2$. Black represents zero, whereas shades of red and blue represent positive and negative variations respectively.}}\label{fig:c5:7}
}
\end{figure}

\begin{figure}{\centering
\includegraphics[height =7cm]{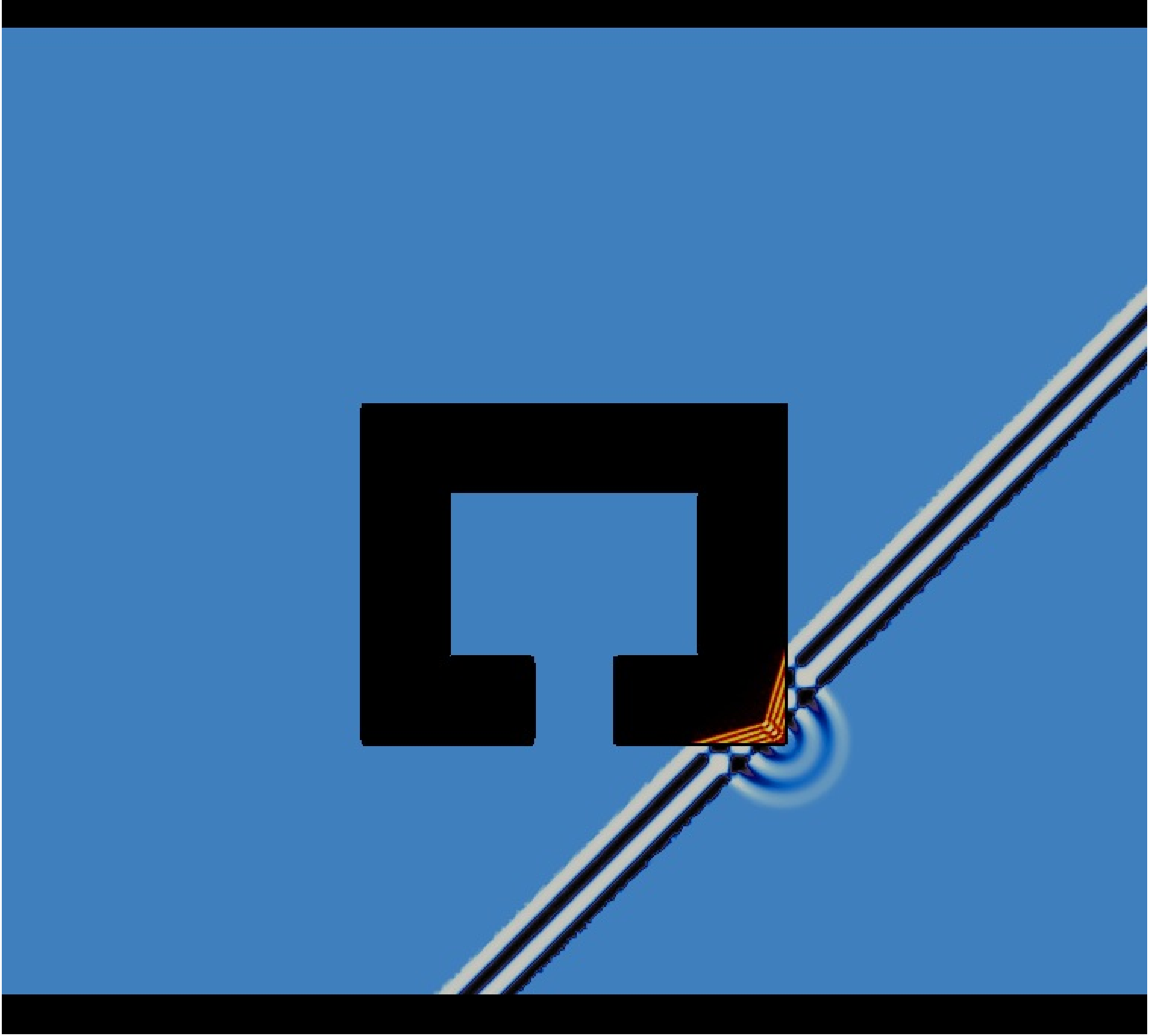}
\includegraphics[height =7cm]{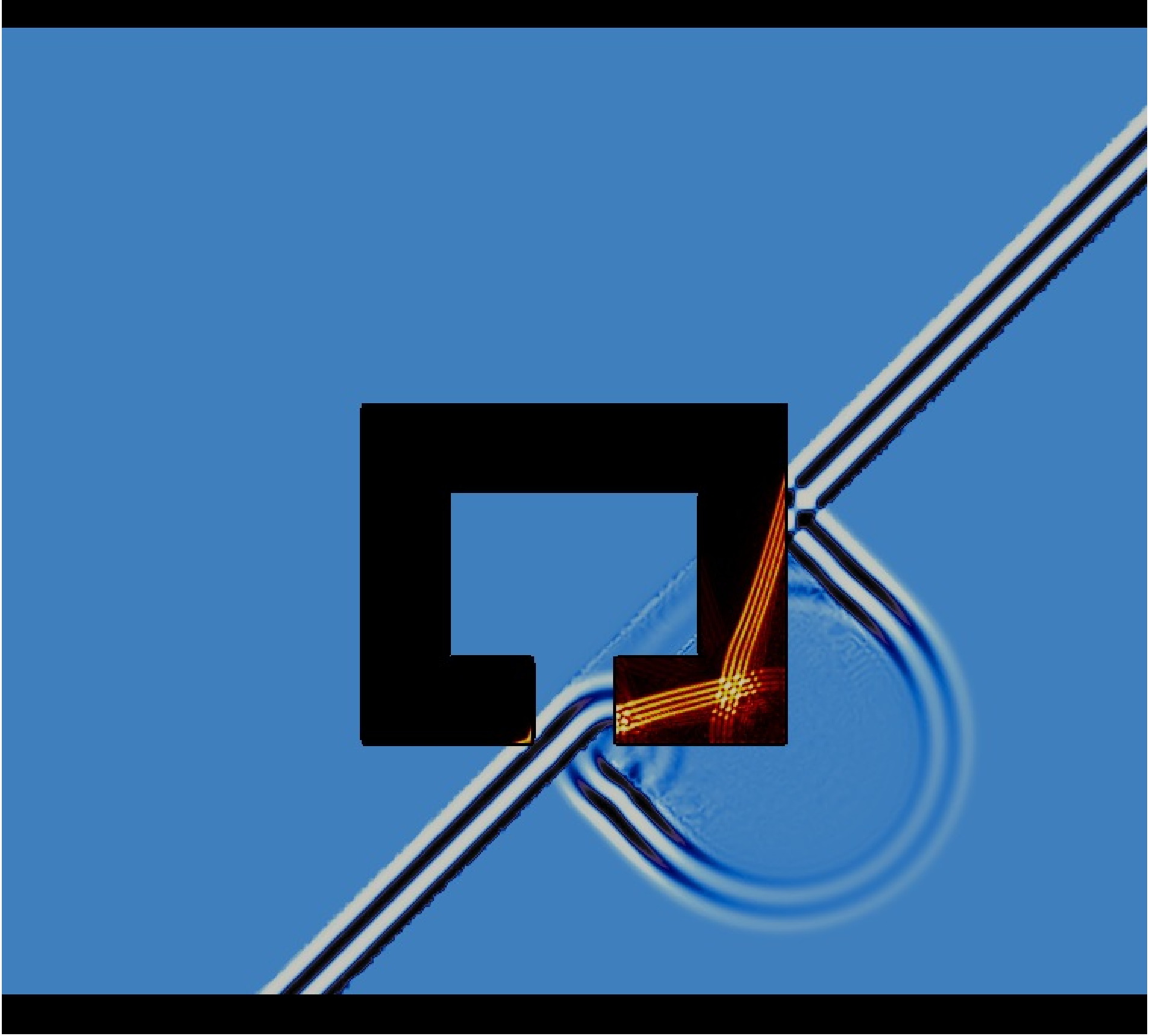}
\includegraphics[height =7cm]{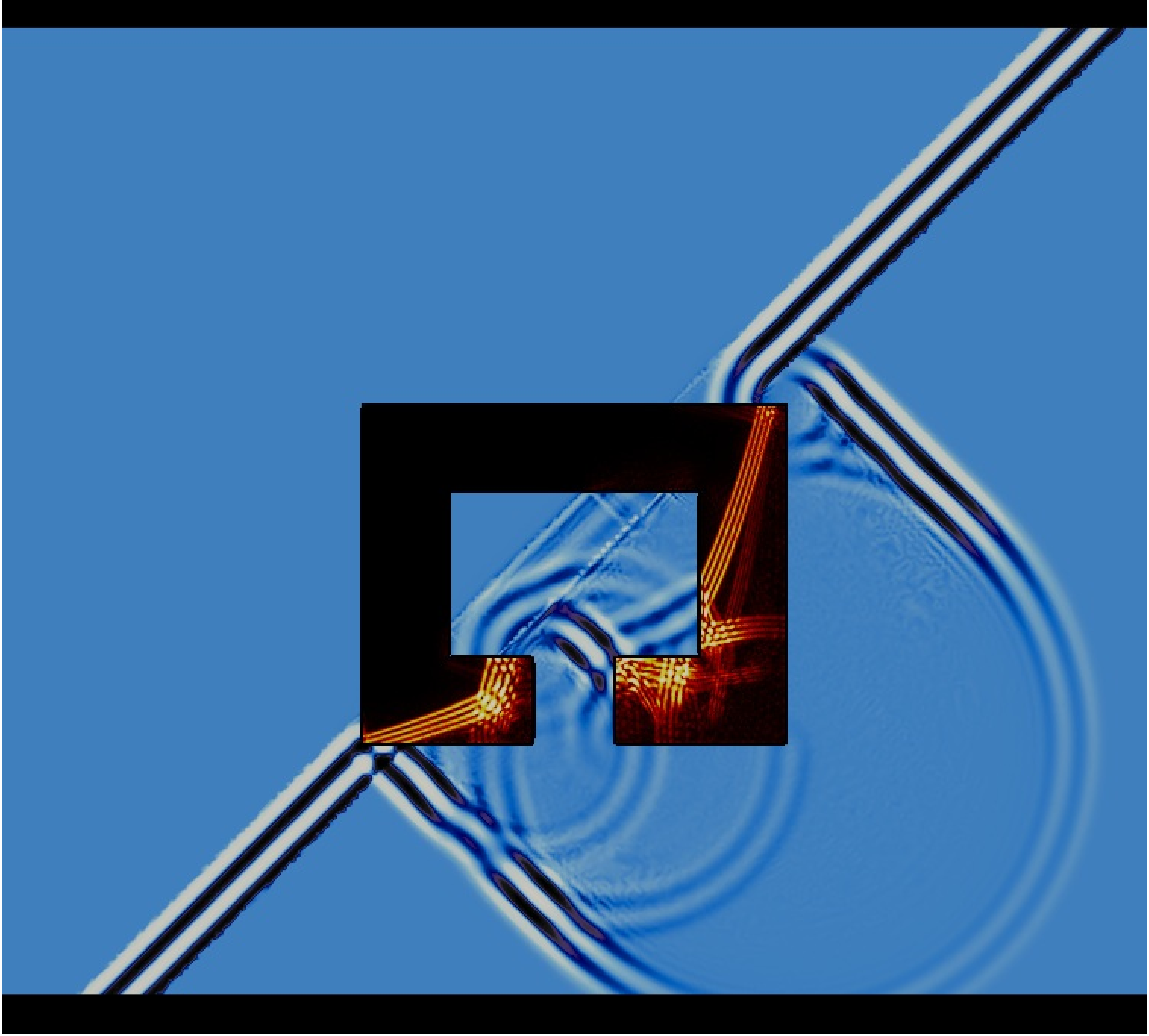}
\includegraphics[height =7cm]{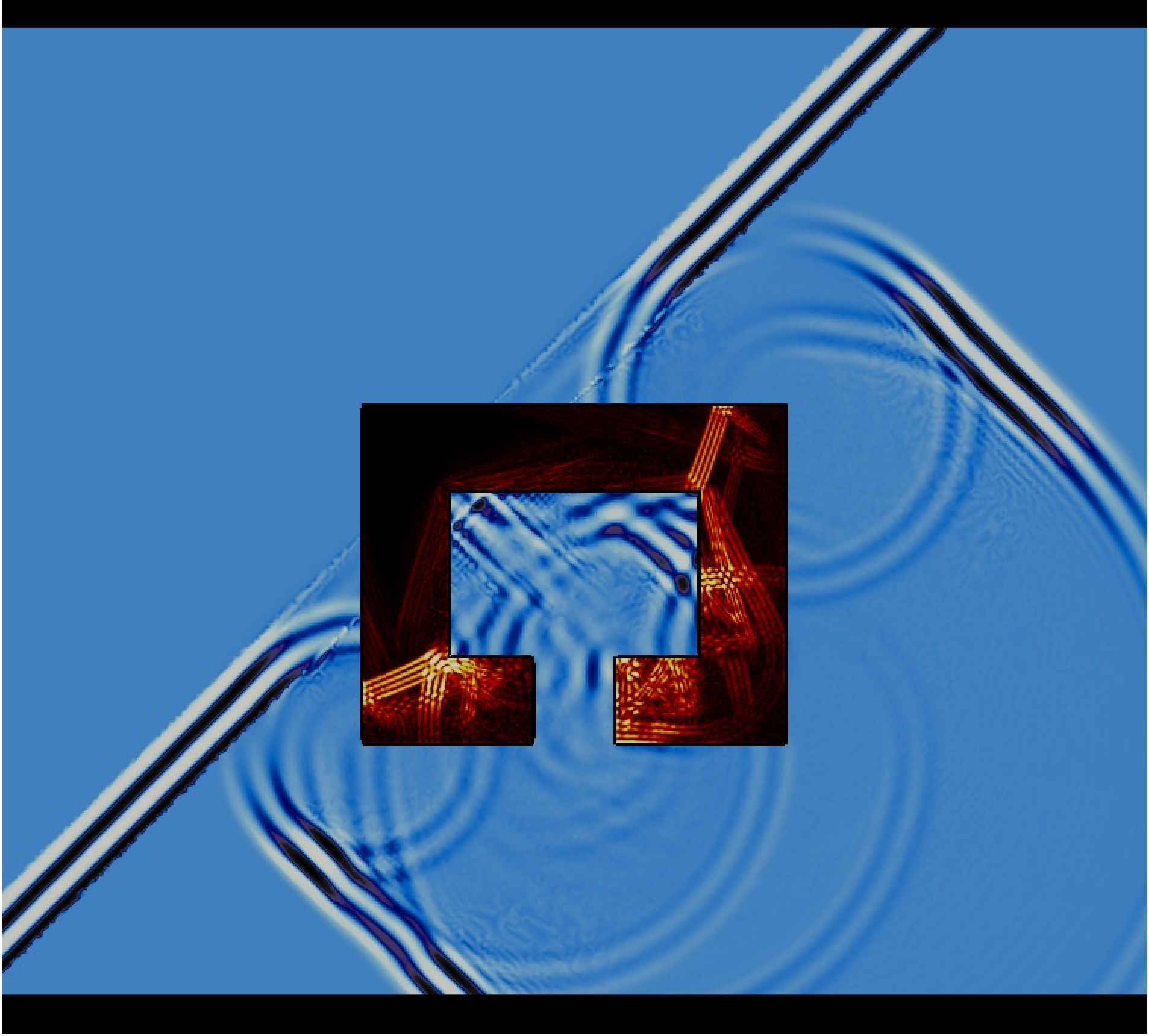}
\includegraphics[height =7cm]{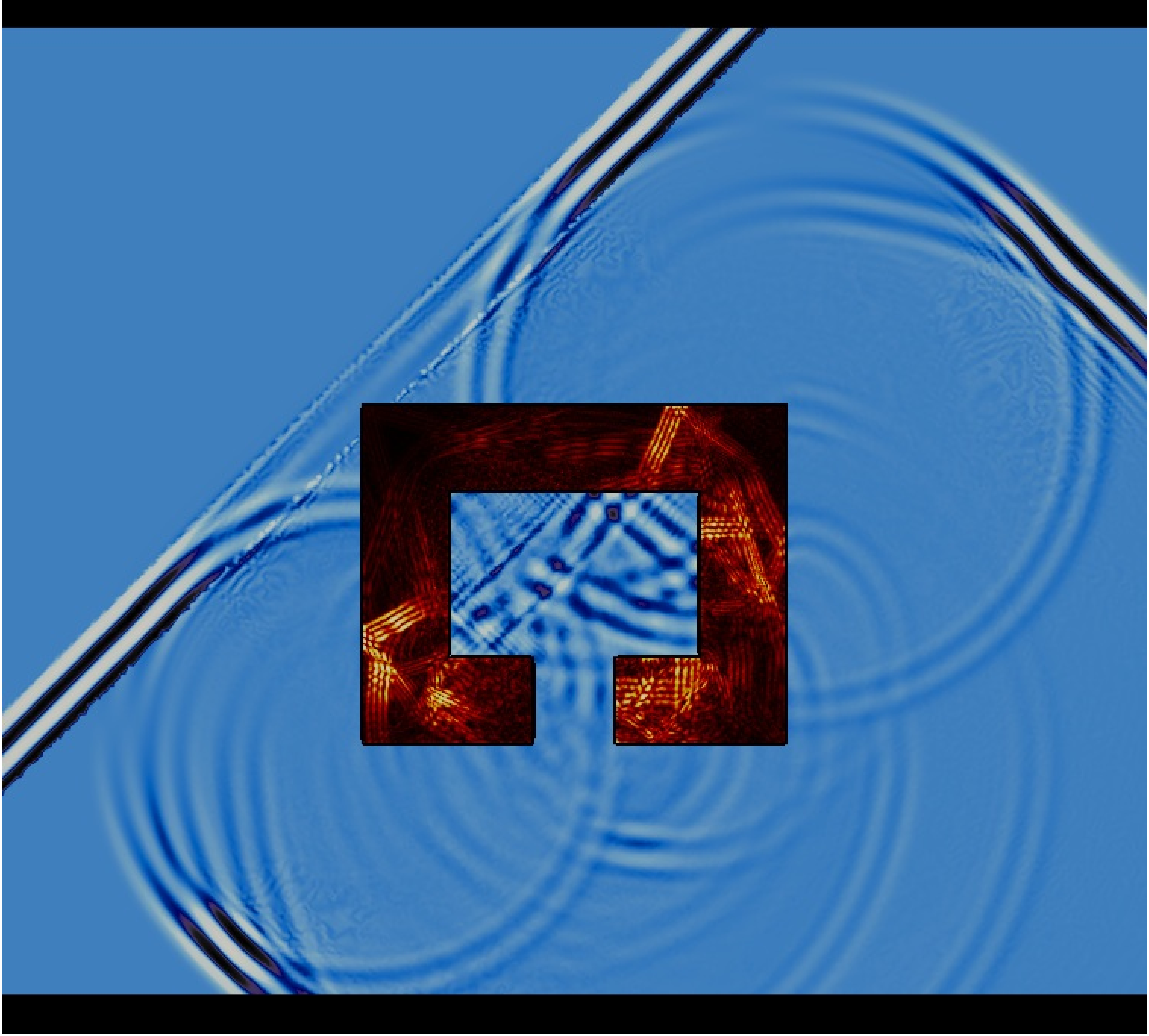}
\includegraphics[height =7cm]{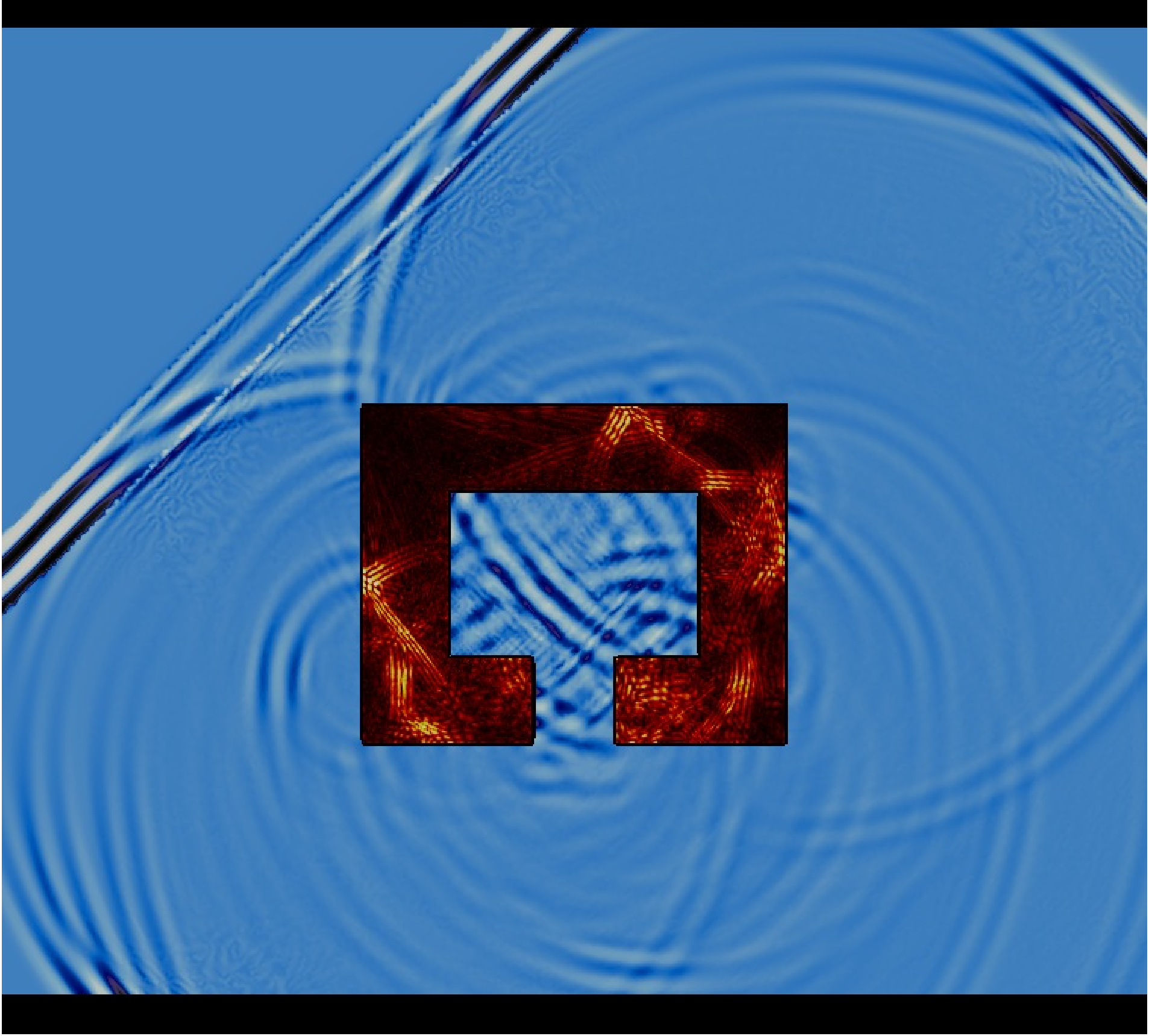}
\caption[Thermoelastic scattering by a pentagon]{{\footnotesize Snapshots of the total acoustic field at times $t=0.3,0.6,0.9, 1.2,1.5,1.8$. The interior domain shows the norm of the elastic displacement.}}\label{fig:c5:8}
}
\end{figure}

\begin{figure}{\centering
\includegraphics[height =5cm]{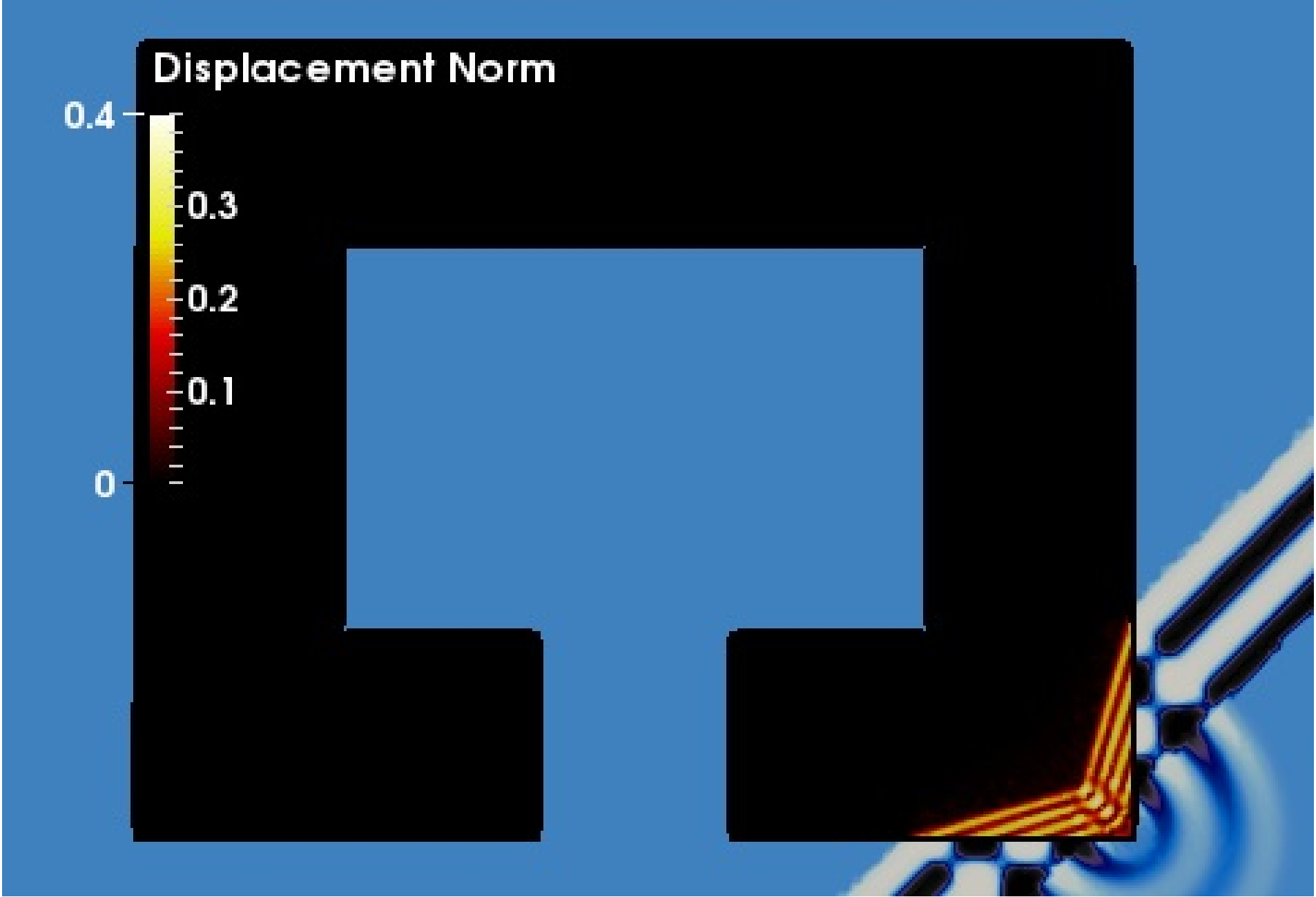}
\includegraphics[height =5cm]{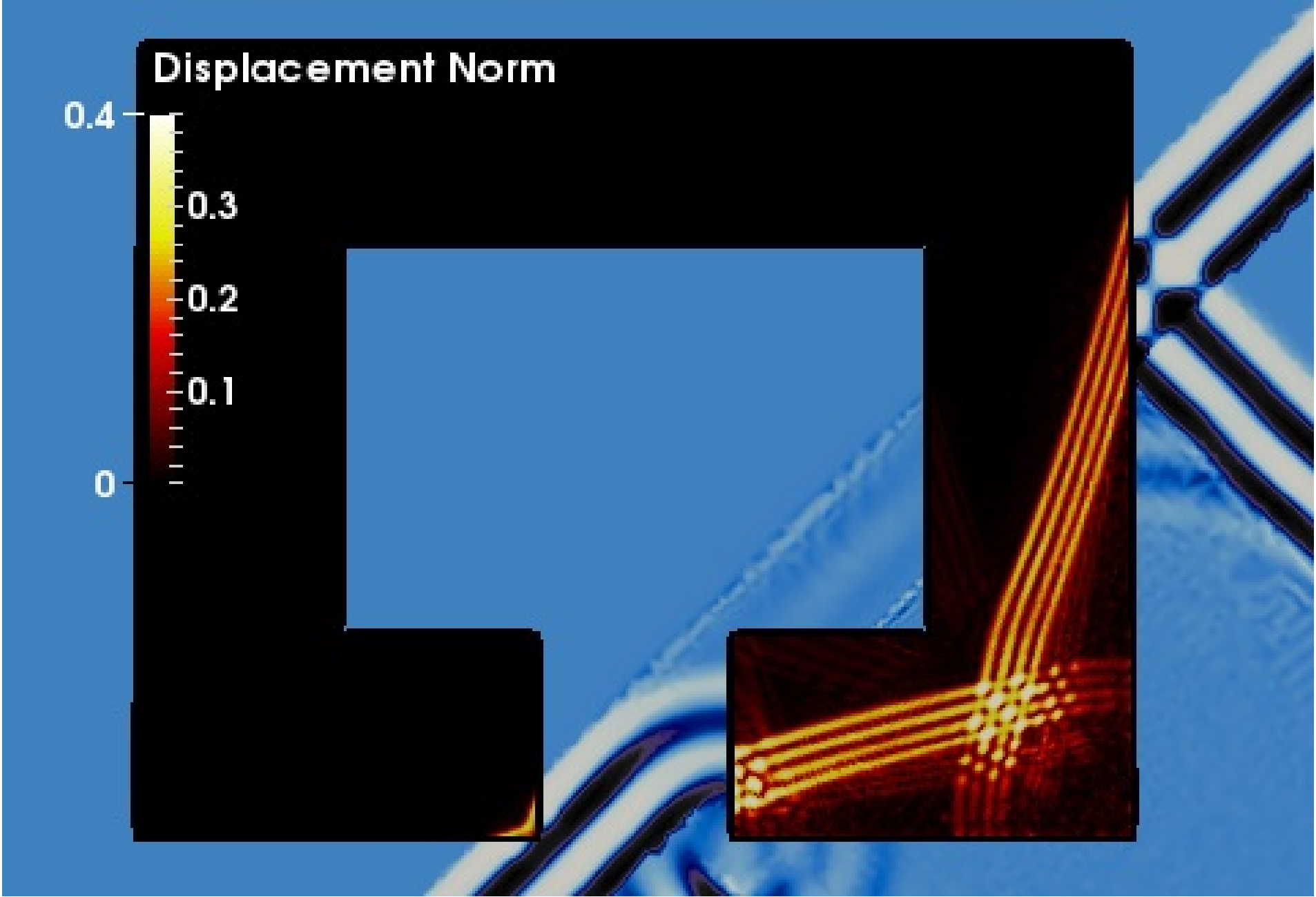}
\includegraphics[height =5cm]{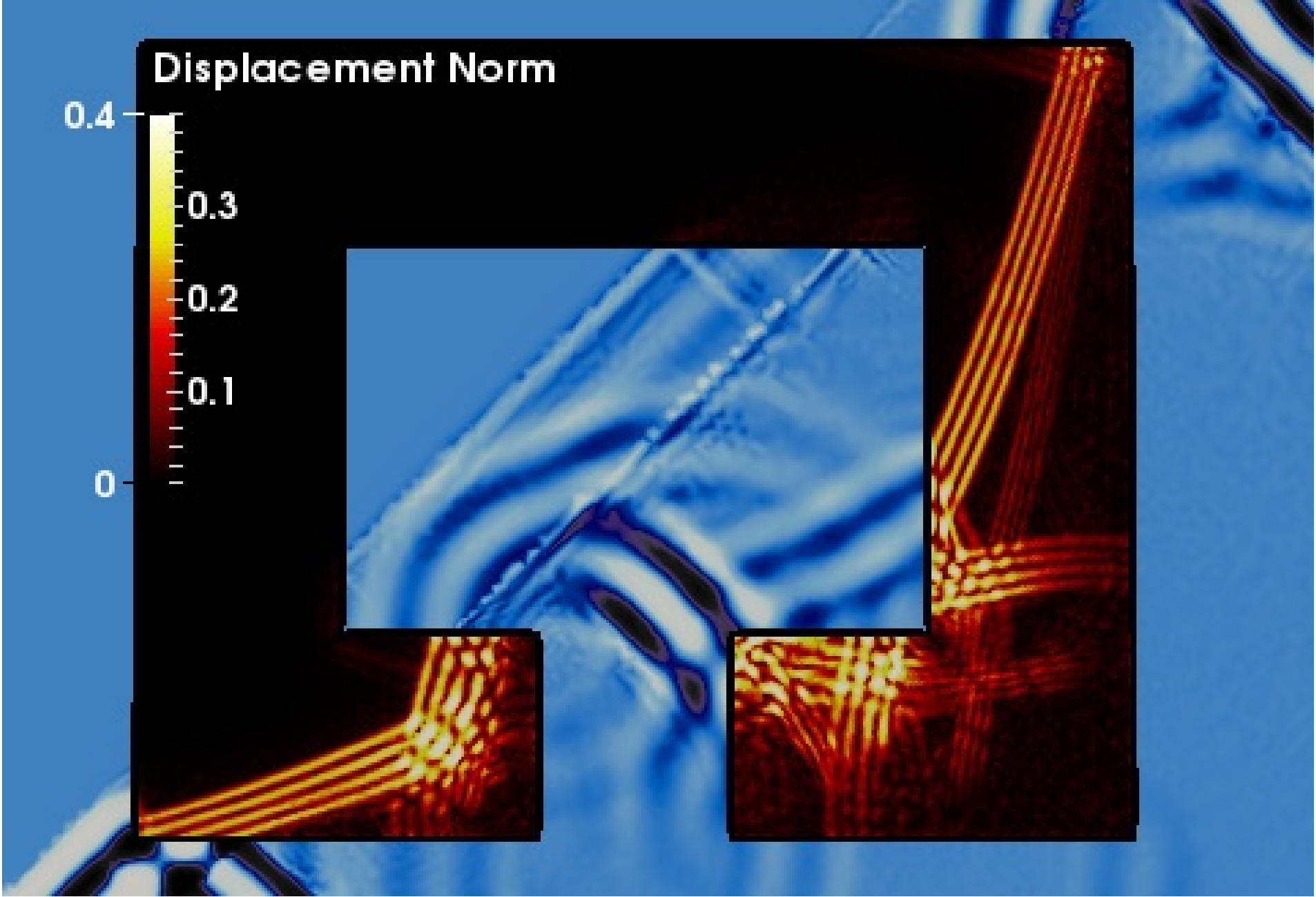}
\includegraphics[height =5cm]{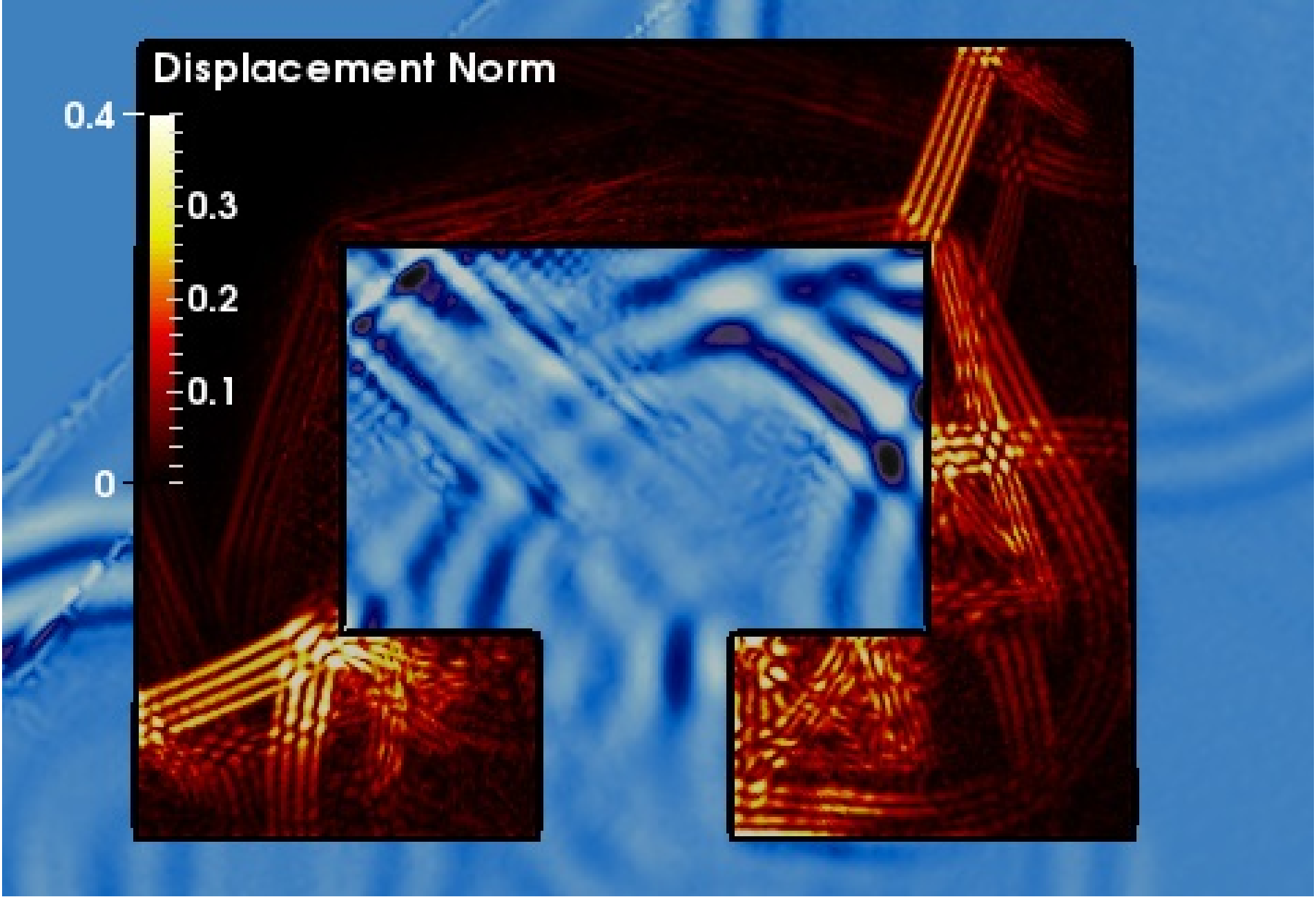}
\includegraphics[height =5cm]{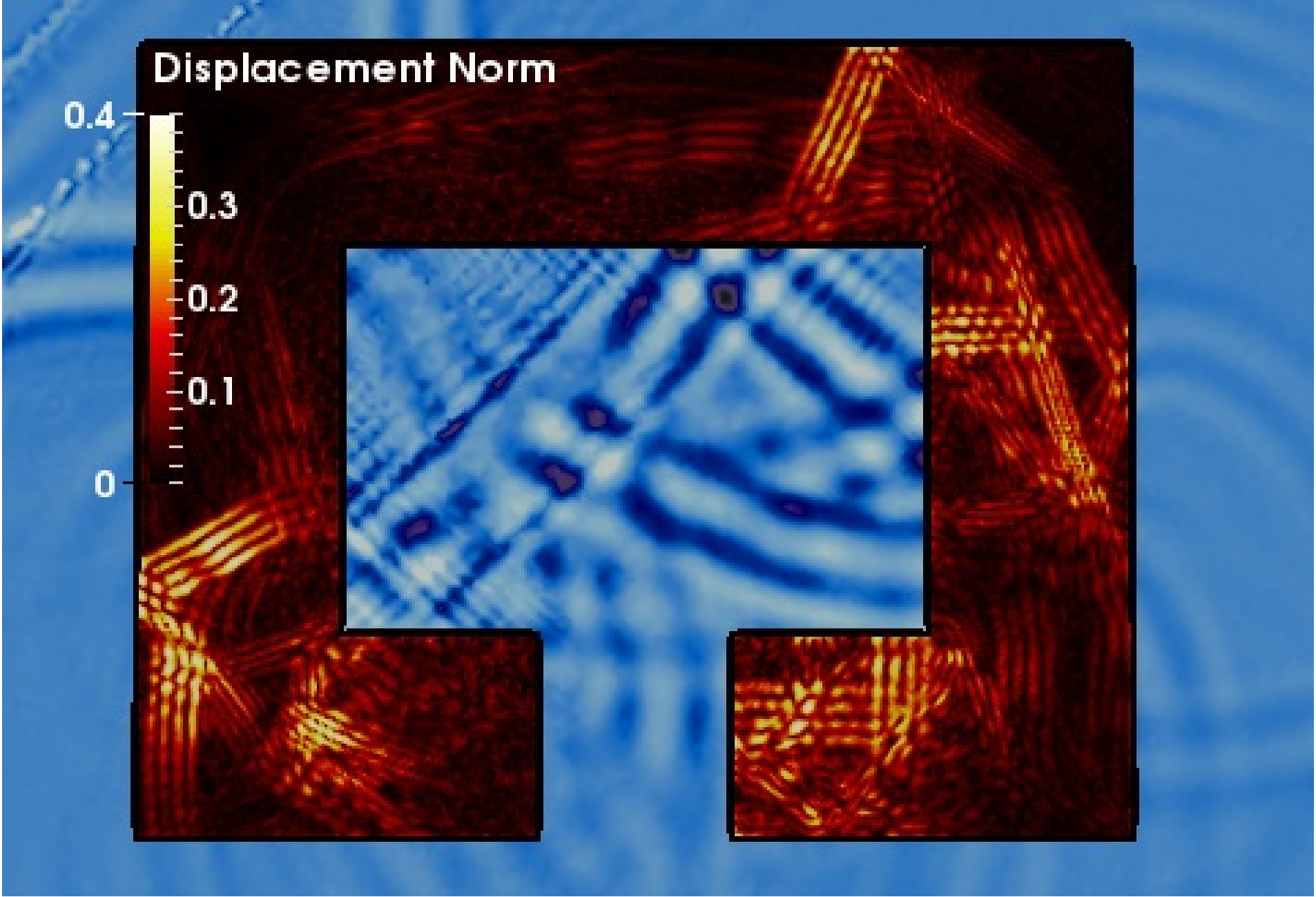}
\includegraphics[height =5cm]{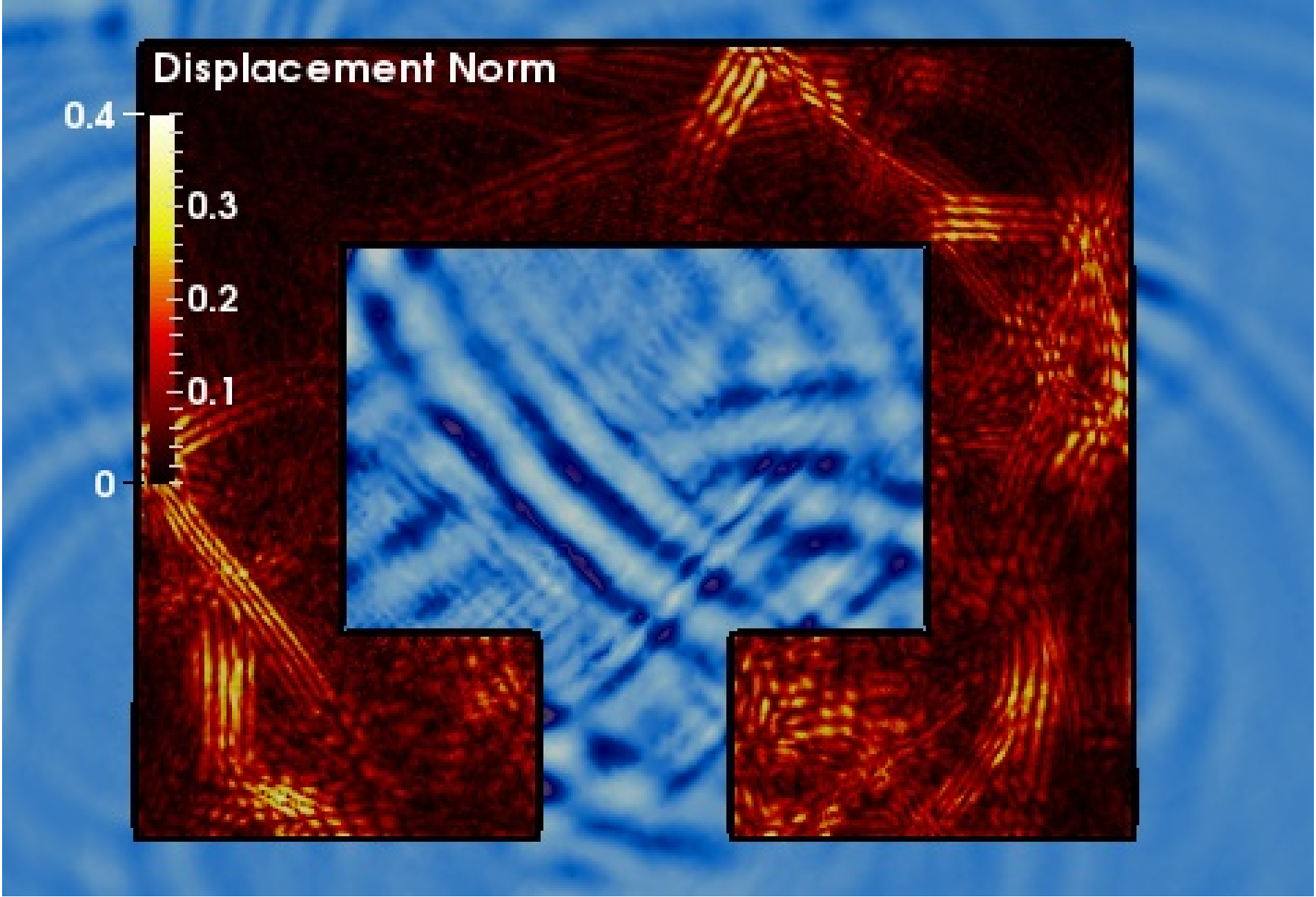}
\caption[Thermoelastic scattering by a pentagon]{{\footnotesize Close up of the norm of the elastic displacement for times $t=0.3,0.6,0.9, 1.2,1.5,1.8$. Black represents no displacement.}}\label{fig:c5:9}
}
\end{figure}

\begin{figure}{\centering
\includegraphics[height =5cm]{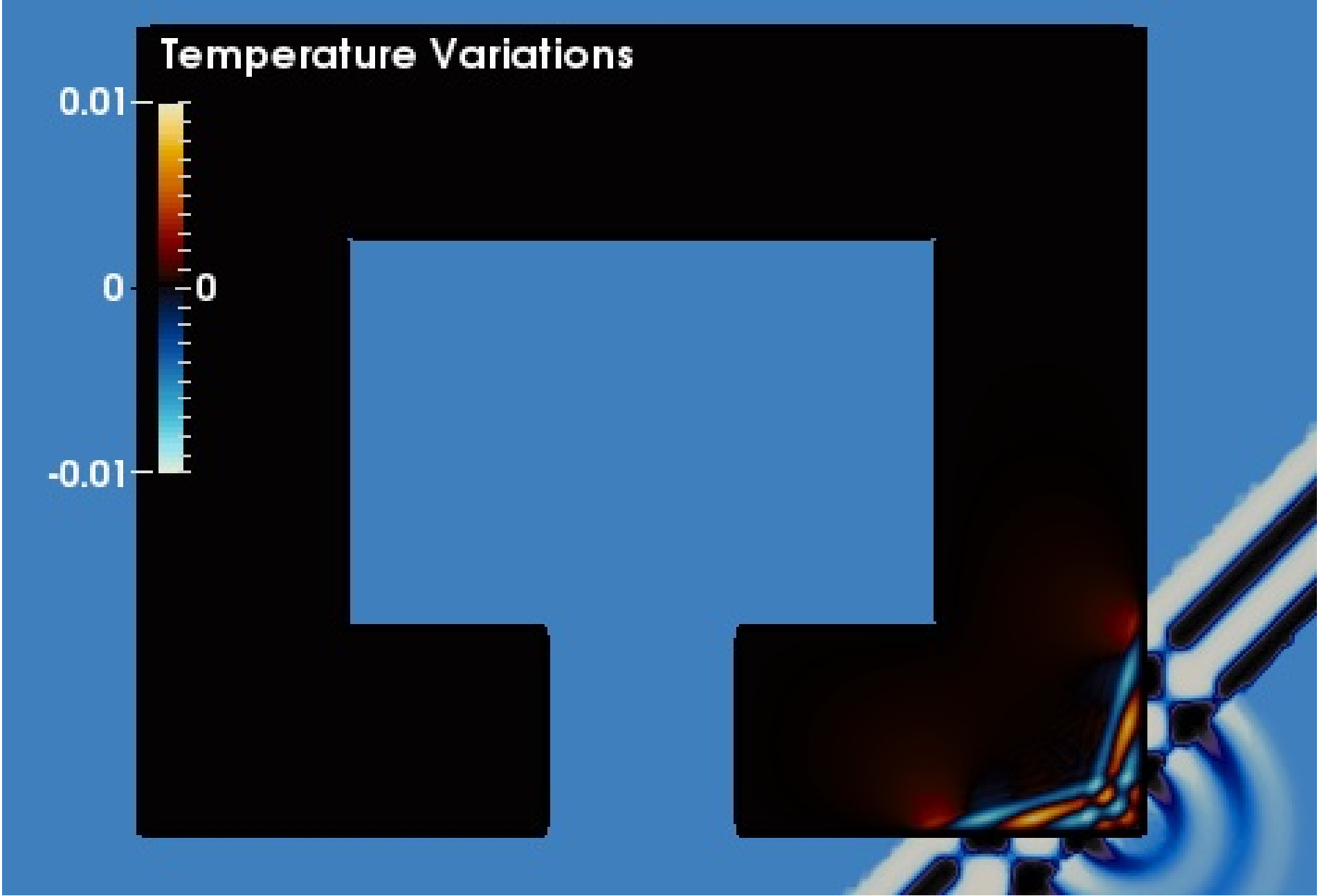}
\includegraphics[height =5cm]{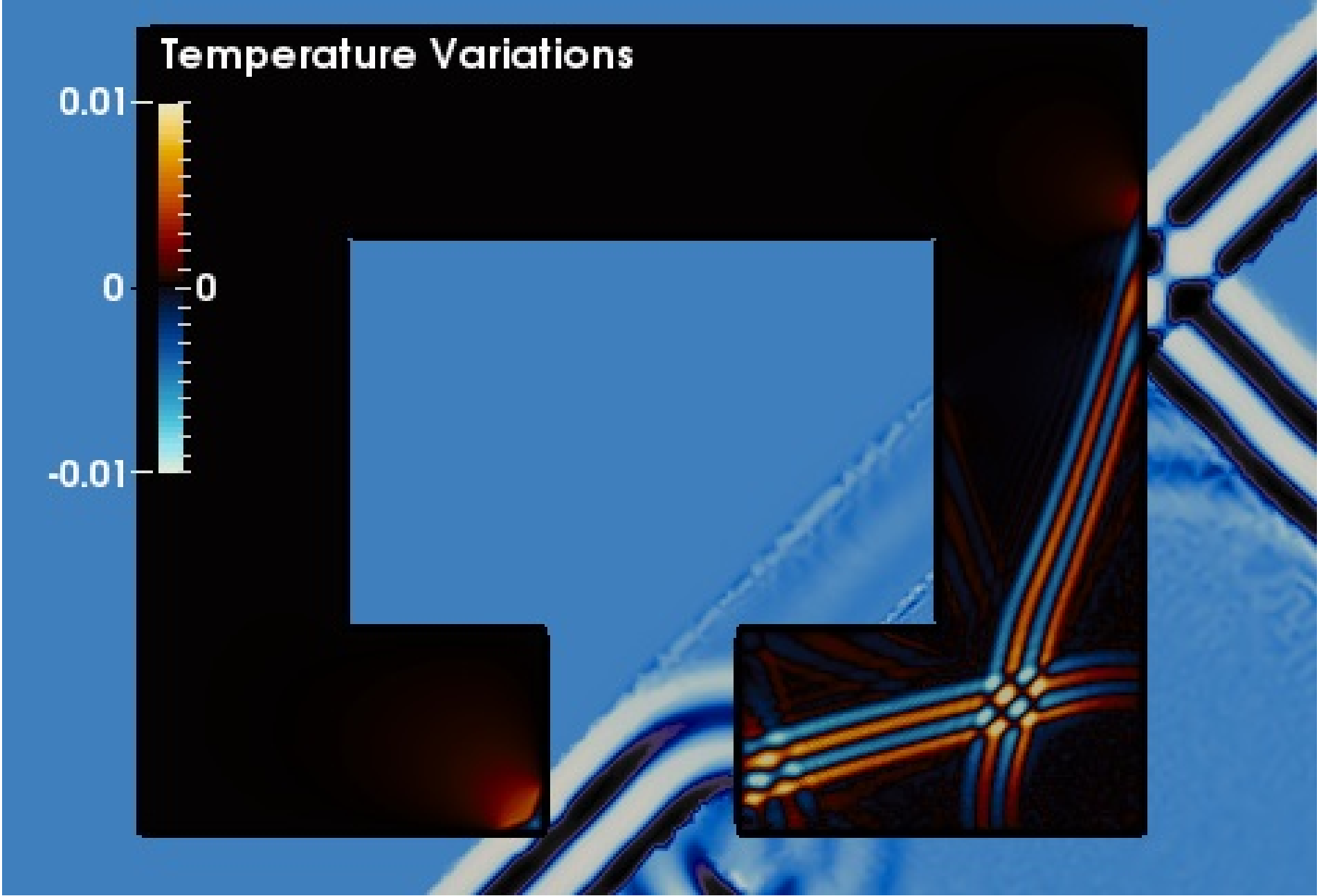}
\includegraphics[height =5cm]{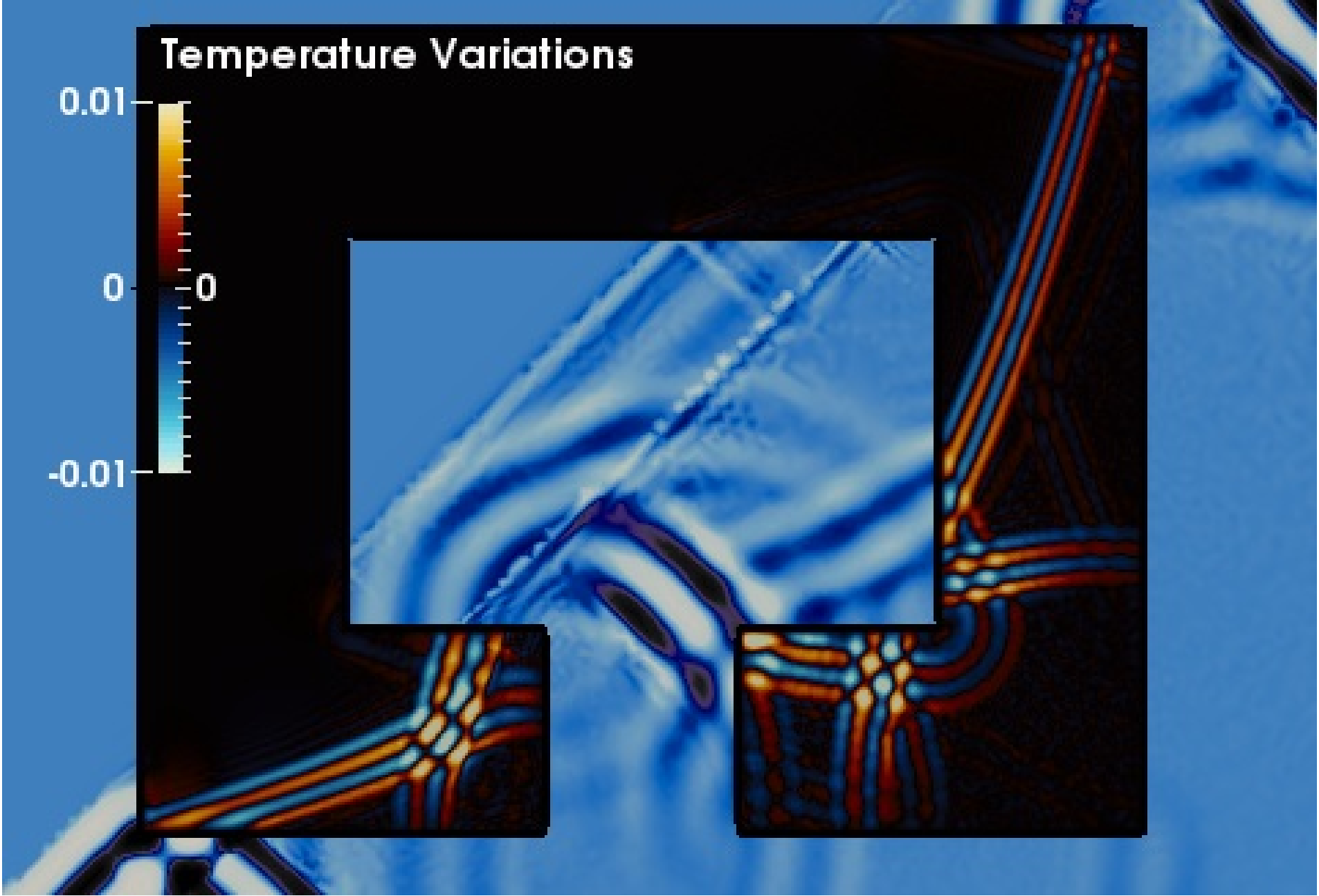}
\includegraphics[height =5cm]{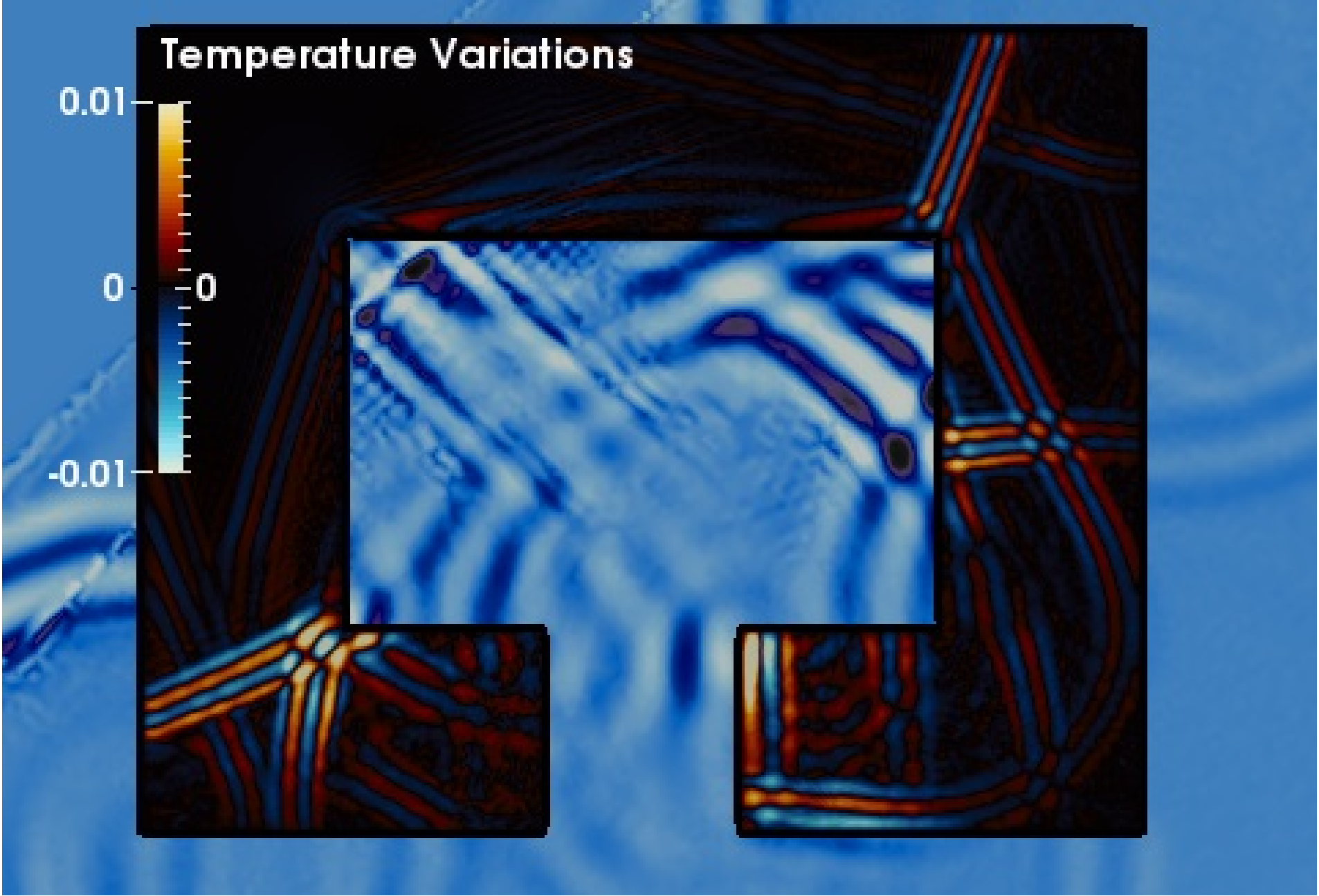}
\includegraphics[height =5cm]{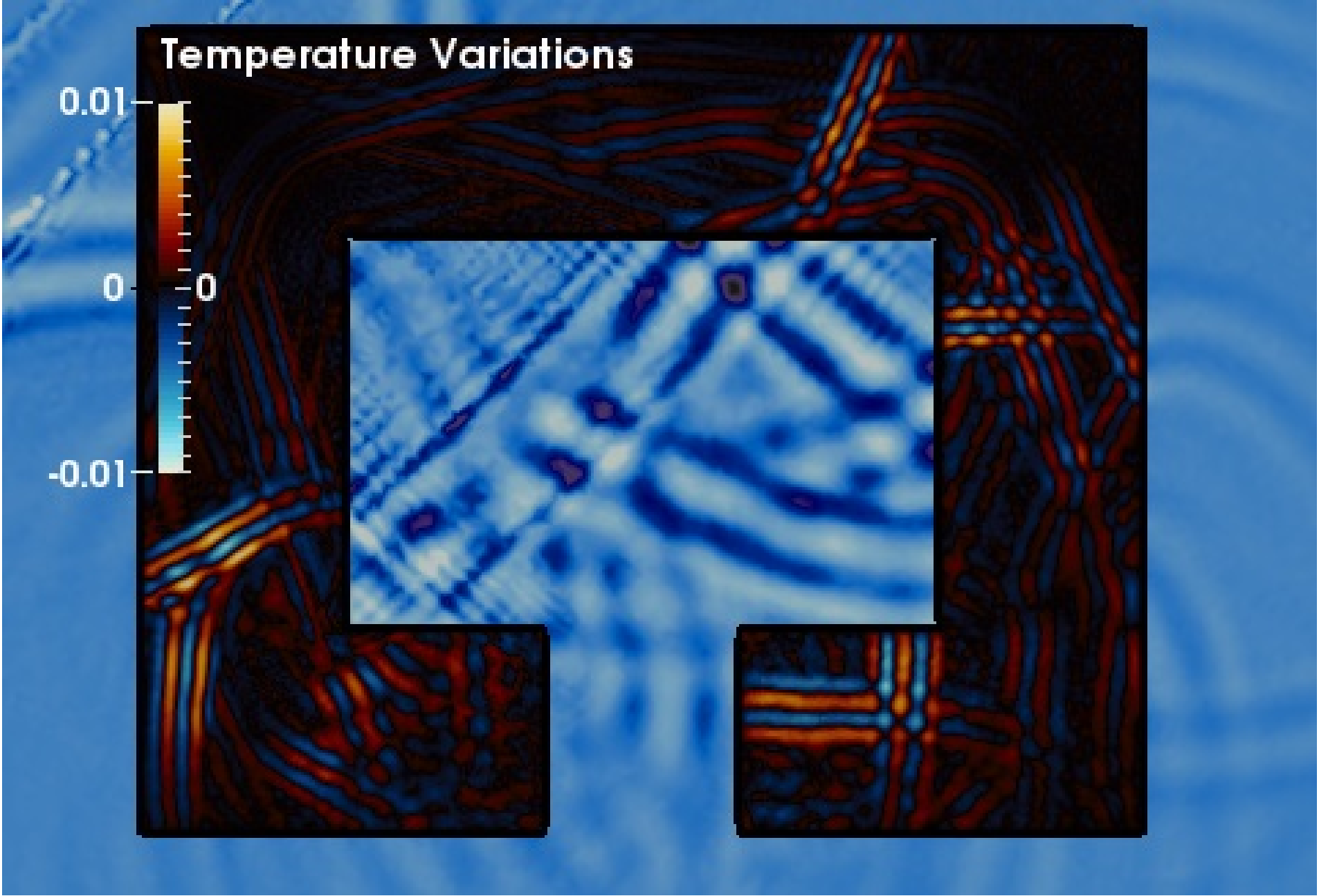}
\includegraphics[height =5cm]{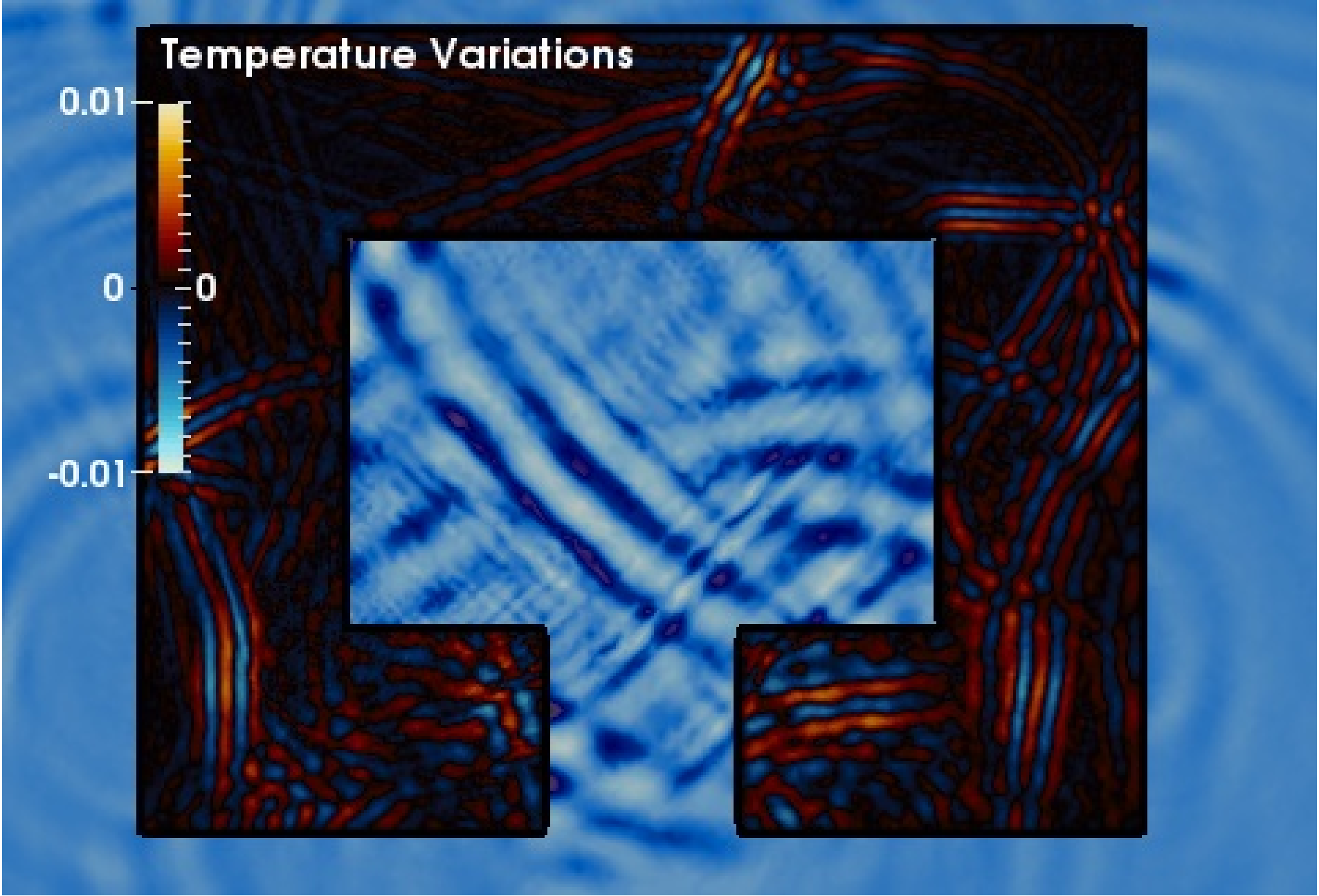}
\caption[Thermoelastic scattering by a pentagon]{{\footnotesize Close up of the norm of the temperature variations with respect to the reference configuration for times $t=0.3,0.6,0.9, 1.2,1.5,1.8$. Black represents zero, whereas shades of red and blue represent positive and negative variations respectively.}}\label{fig:c5:10}
}
\end{figure}

\end{document}